\documentclass[11pt]{article}
\usepackage{a4wide}
\usepackage{amsmath, amsthm}
\usepackage{amssymb}
\usepackage{graphicx}
\usepackage{xfrac}
\usepackage{relsize}

\newcommand{\hidden}[1]{}

\theoremstyle{plain}
\newtheorem{thm}{Theorem}
\newtheorem{cor}{Corollary}
\newtheorem{lem}{Lemma}
\newtheorem{prop}{Proposition}

\newtheorem{lemA}{Lemma}[subsection]
\makeatletter
\@addtoreset{lemA}{subsection}
\makeatother

\theoremstyle{definition}

\theoremstyle{claim}
\newtheorem{claim}{Claim}

\theoremstyle{fact}
\newtheorem{fact}{Fact}

\theoremstyle{remark}
\newtheorem{rem}{Remark}

\newtheoremstyle{case}{}{}{}{}{}{.}{ }{}
\theoremstyle{case}

\newcommand{\ca}[1]{\mathcal{#1}}

\newcommand{\ve}{\varepsilon}
\newcommand{\beq}{\begin{equation}}
\newcommand{\eq}{\end{equation}}

\newcommand{\what}{\widehat}
\newcommand{\R}{\mathbb{R}}
\newcommand{\N}{\mathbb{N}}
\newcommand{\Z}{\mathbb{Z}}
\newcommand{\I}{\mathbb{I}}

\newcommand{\cA}{\mathcal{A}}

\newcommand{\cP}{\mathcal{P}}
\newcommand{\cS}{\mathcal{S}}

\newcommand{\cT}{\mathcal{T}}
\def\cI{\mathcal{I}}

\def\Bad{{\rm {\bf Bad}}}

\newtheorem{thkh}{\rm {\bf Theorem \!\! KS}\!\!}

\newtheorem{thdel}{\rm {\bf Theorem \!\! DEL} (Davenport, Erd\"{o}s \& LeVeque)\!\!}

\newtheorem{thdelcorr}{\rm {\bf Corollary \!\! DEL}\!\!}
\newtheorem{thhjk}{\rm {\bf Theorem \!\! HJK}\!\!}
\newtheorem{thhjkbw}{\rm {\bf Theorem \!\! BW}\!\!}
\newtheorem{thpropa}{\rm {\bf Proposition \!\! 3A}\!\!}

\begin{document}


\title{ Inhomogeneous Diophantine Approximation on $M_0$-sets with restricted denominators}

\author{ Andrew D. Pollington \\ {\small\sc (NSF, Washington) } \and 
Sanju Velani\footnote{Research partly supported by EPSRC Programme grant EP/J018260/1} \\ {\small\sc (York) }  \and
Agamemnon Zafeiropoulos\footnote{Research supported by the Austrian Science Fund (FWF), projects F-5512 and Y-901  }  \\ {\small\sc (TU Graz) }  
 \and ~ Evgeniy Zorin\footnote{Research partly supported by EPSRC Grant EP/M021858/1} \\ {\small\sc (York) } \\ {\small\sc (York)}
}


\date{}

\maketitle

\begin{abstract}
Let $F \subseteq [0,1]$ be a set that supports a probability measure $\mu$ with the property that   $ |\widehat{\mu}(t)|  \ll (\log |t|)^{-A}$  for some constant $ A > 0 $.    Let $\ca{A}= (q_n)_{n\in \N} $  be a sequence of natural numbers.  If   $\ca{A}$ is lacunary and $A >2$, we establish a quantitative inhomogeneous Khintchine-type theorem  in which (i) the points of interest are restricted to $F$ and (ii) the denominators of the `shifted' rationals are restricted to   $\ca{A}$.  The theorem can be viewed as a natural strengthening of the fact that the sequence $(q_nx  {\rm \  mod \, } 1)_{n\in \N} $  is uniformly distributed for $\mu$ almost all $x \in F$.   Beyond lacunary, our main theorem implies the analogous quantitative result for sequences $\ca{A}$  for which the prime divisors  are restricted to a finite set of $k$ primes and $A > 2k$.
\end{abstract}

\renewcommand{\baselinestretch}{1}

\parskip=1ex

\section{Introduction and results}

\subsection{Motivation and lacunary results \label{motlacres} } We start by setting the scene.  Throughout, $F $ will be a subset of the unit interval $\I:=[0,1]$  that supports a non-atomic probability measure $\mu$.  As usual,
 the Fourier transform of $\mu$ is defined by
$$ \widehat{\mu}(t) \, := \int   e^{-2\pi it x}  \, \mathrm{d}\mu (x)      \hspace{1cm} (t\in\mathbb{R} )\, .  \vspace{2mm}$$
The set $F$ is called an {\em $M_0$-set} if $\widehat{\mu}(t) $  vanishes at infinity. It is well known that the decay rate of the Fourier transform is related to the Hausdorff dimension of the support of $\mu$. Indeed, a classical result of Frostman states that if $|\widehat{\mu}(t)|\leq c \, |t|^{-\eta/2}$ for some constants  $c, \eta>0$, then $\dim F  \geq \min \{1, \eta\} $. Further details and references  of can be found in~\cite{falc}.
Throughout,  $\ca{A}= (q_n)_{n\in \N} $  will be an increasing sequence of natural numbers.  Recall, that
$\ca{A}$ is  said to be {\em lacunary} if there exists a  constant $ K > 1 $ such that \vspace{2mm}
\begin{equation} \label{lacunary}
\frac{q_{n+1}}{q_n}  \ge K   \hspace{5mm}  (n \in \N)\, .
\end{equation}
The fundamental theorem of  Davenport, Erd\"{o}s $ \&$ LeVeque \cite{DELpaper} in the theory of uniform distribution, shows that  the generic distribution properties  of  a sequence  $(q_nx)_{n\in \N} $  with  $x$ restricted to the support of $\mu$ are intimately related to the  decay rate of $\widehat{\mu}$.

 \vspace{2mm}

 \begin{thdel} {\em Let $\mu$ be a probability  measure supported on a subset $F$ of $\, \I$. Let $\ca{A}= (q_n)_{n\in \N} $ be a sequence of natural numbers. If
\begin{equation}
\sum_{N=1}^\infty \frac{1}{N^3}  \; \sum_{m,n=1}^N  \widehat{\mu}
(h(q_m-q_n)) \; < \; \infty  \;  \label{del}
\end{equation}
 for all integers $h\ne 0$, then the sequence $
(q_nx)_{n\in \N} $ is uniformly distributed modulo one  for $\mu$--almost all $x \in F$.}
\end{thdel}

\noindent In the case the sequence $\ca{A}$ is lacunary, the theorem gives rise to the following elegant statement.

\vspace*{2mm}

\begin{thdelcorr}
{\em Let $\mu$ be a probability  measure supported on a subset $F$ of
$ \, \I  $.  Let $\ca{A}= (q_n)_{n\in \N} $ be a lacunary sequence of natural numbers. Let  $f: \N \to \R^+$ be a decreasing function such that
\begin{equation} \label{lacdelsum}
\sum_{n=2}^\infty \frac{f(n)}{n \log n} \; < \; \infty  \;
\end{equation}
and suppose that
\begin{equation} \label{lacdelmu}
\widehat{\mu}(t)= O\left( f( |t|)  \right)   \hspace{6mm} as \quad  |t|\to\infty \, . \vspace{2mm}
\end{equation}
Then  the sequence $(q_n x)_{n \in \N}$ is uniformly
distributed modulo one  for $\mu$--almost all $x \in F$.}
\end{thdelcorr}

\noindent The deduction of the corollary from the theorem is reasonably  straightforward.    However, for the sake of completeness and the reader's convenience we provide the details at the end of the paper in \S\ref{app} - Appendix~\!A.

\begin{rem}\label{rem1}  Reinterpreting the corollary in terms of normal numbers, it implies that  $\mu$--almost all  numbers  in $F$ are normal. Thus, it provides a useful  mechanism for proving the existence of normal numbers in a given subset of real numbers.  Indeed, Corollary~\!DEL implies that if a given set $F$ supports a probability  measure $\mu $ such that  for some $\epsilon > 0$
\begin{equation} \label{lacdelmu2}
\widehat{\mu}(t)= O\left( (\log \log |t|)^{-(1+ \epsilon)}     \right)   \hspace{6mm} {\rm as}  \quad  |t|\to\infty \, , \vspace{2mm}
\end{equation}
then $\mu$--almost all numbers in $F$ are normal. This observation is key, for example,  in showing that there are normal numbers which are badly approximable -- see Remark~\ref{rem5a} in \S\ref{cpw} below.  For completeness, we mention that  Theorem DEL (and its corollary) is valid for non-integer sequences and that this is at the heart of addressing the long standing problem of when normality to one base, not necessarily integer, implies normality to another (see \cite{morpoll} and references within).
\end{rem}

\begin{rem}\label{rem1a}
In the language of Kahane \cite{Kahane} and Lyons \cite{Lyons}, the  conclusion of the corollary is equivalent to saying that the $\mu$-measure of every lacunary $W^*$-set is zero -- see  \cite[Theorem~4]{Lyons}.   Basically, a Borel set $F\subset \I$  is a  lacunary $W^*$-set  if there exists a lacunary sequence $\ca{A} $ such that  $(q_n x)_{n \in \N}$ is not  uniformly
distributed modulo one  for any $x \in F$.
\end{rem}

Let $\gamma \in \I$ and let $B=B(\gamma, r) \subseteq \I$ denote the  ball  centred at $ \gamma$ with radius $r \leq 1/2$.  By the definition of uniform distribution, Corollary~\!DEL implies that for $\mu$--almost all $x \in F$ the sequence $(q_n x)_{n \in \N} $ modulo one `hits' the ball $B$ the `expected' number of times. In other words, for $\mu$--almost all $x \in F$
\begin{equation} \label{lacud}
\lim_{N \to \infty}  \  \frac{1}{N}  \ \# \big\{ 1\leq n \leq N : \|q_n x-\gamma \| \leq r  \big\} = 2r \, ,
\end{equation}
where  $\| \alpha \|:=\min \{|\alpha-m|:m\in\Z\}$ denotes  the distance from $\alpha\in \R $ to the nearest integer.  In this paper, we consider the situation in which the radius of the ball is allowed to shrink with time. With this in mind, let $ \psi : \N \to \I$ be a  real, positive function  and  consider the counting function
\begin{equation} \label{countdef} R(x,N) \, = \, R(x,N;\gamma,\psi, \ca{A}) := \# \big\{ 1\leq n \leq N : \|q_n x-\gamma \| \leq \psi(q_n) \big\}.
\end{equation}
As alluded to in the definition, we will often simply write $R(x,N)$ for $R(x,N;\gamma,\psi, \ca{A})$ since the other three dependencies will be clear from the context and are usually fixed.  Our first  result implies  that if $\widehat{\mu}$ decays  quickly enough,  then for $\mu$--almost all $x \in F$ the sequence $(q_n x)_{n \in \N} $ modulo one `hits' the shrinking ball $B(\gamma,\psi(q_n)) $ the `expected' number of times.

\begin{thm} \label{asmplacthm}
Let $\mu$ be a probability  measure supported on a subset $F$ of
$ \ \I \, $.  Let $\ca{A}= (q_n)_{n\in \N} $ be a lacunary sequence of natural numbers.
Let $\gamma\in\I  $ and $\psi:\mathbb{N}\rightarrow \I$ be a real, positive function.  Suppose there exists a constant $A > 2$,  so that
 \begin{eqnarray} \label{decaylac}
\widehat{\mu}(t)= O\left( (\log |t|)^{-A}     \right)   \hspace{6mm} \mbox{as}  \quad  |t|\to\infty \,  .
\end{eqnarray}
Then, for any $\varepsilon>0$, we have that
\begin{eqnarray}\label{countlacresult}
R(x,N) \, = \,  2\Psi(N)+O\Big(\Psi(N)^{2/3}\big(\log(\Psi(N)+2)\big)^{2+\varepsilon}\Big)
\end{eqnarray}
\noindent for $\mu$-almost all $x\in F$, where
\begin{equation} \label{def_Psi}
\Psi(N):=\sum_{n=1}^N\psi(q_n) \, .
\end{equation}
\end{thm}

\begin{rem}\label{rem2}By definition,   $ \Psi(N) = rN$  when $\psi$ is the constant function $ \psi(n) = r  $ and so the theorem trivially implies \eqref{lacud}.  Indeed, it implies  \eqref{lacud}  with an error term.
However, note that to apply the theorem we need  to assume a faster  logarithmic decay rate than that given by \eqref{lacdelmu2} which suffices to conclude \eqref{lacud}.
\end{rem}

Hopefully, it is pretty clear that the theory of uniform distribution, in particular the theorem of Davenport, Erd\"{o}s $ \&$ LeVeque, is a key motivating factor towards establishing statements such as Theorem~\ref{asmplacthm}.  Another key motivating factor, which we now bring to the forefront,  is the theory of Diophantine approximation on manifolds; also known as Diophantine approximation of dependent quantities. In short, this theory refers to the study of Diophantine properties of points in $\R^n$ whose coordinates are confined by functional relations or equivalently are restricted to a sub-manifold
of $\R^n$. Over the last twenty years, the theory has developed at some considerable pace with the catalyst undoubtedly  being the pioneering work of Kleinbock \& Margulis on the Baker-Sprind\v zuk conjecture (see \cite[\S6]{durham}).
Given a real number  $\gamma\in\I  $, a function $\psi:\mathbb{N}\rightarrow \I$   and a sequence $\ca{A} = (q_n)_{n\in \N}$ of natural numbers, consider the set \begin{equation} \label{wellset} W_{\ca{A}}(\gamma; \psi):= \left\{ x\in \I: \|q_nx-\gamma \| \leq \psi(q_n) \text{ for infinitely many } n\in\mathbb{N} \right\}. \vspace{2mm} \end{equation}
By definition, $x \in W_{\ca{A}}(\gamma; \psi) $ if and only if the inequality  $$
\Big| x - \frac{p+ \gamma}{q} \Big|   \leq  \frac{\psi(q)}{q}   $$  is satisfied for infinitely many $ (p,q) \in \Z \times \ca{A} $.  In other words,  $W_{\ca{A}}(\gamma; \psi)$ is the standard set of inhomogeneous $\psi$-well approximable real numbers in which the denominators $q$  of the shifted rational approximates  $(p+ \gamma)/q$ are restricted to the set $\ca{A}$. When $\ca{A} = \N$,  we will drop the subscript
$\ca{A}$  from $W_{\ca{A}}(\gamma; \psi)$.   The fundamental theorem of Khintchine in the theory of metric Diophantine approximation,  provides an elegant criterion for the `size' of the set $W(\gamma; \psi) $  expressed in terms of Lebesgue measure $m$.


%

\medskip

\begin{thkh}
{\em Let $\gamma\in\I  $ and $\psi:\mathbb{N}\rightarrow \I$ be a real, positive  decreasing function.  Then}
$$ m\big(W(\gamma; \psi)  \big) =\left\{
\begin{array}{ll}
  0 & \mbox{if} \;\;\; \sum\limits_{n=1}^{\infty}\psi(n)  <\infty\;
      ,\\[2ex]
  1 & \mbox{if} \;\;\; \sum\limits_{n=1}^{\infty}\psi(n)
      =\infty \; .
\end{array}\right.$$
\end{thkh}

\noindent To be accurate, Khintchine \cite{Khintchine24} proved the homogeneous statement (i.e. when $\gamma = 0$) in 1924. Sz\"{u}sz \cite{szusz} generalized Khintchine's result to the inhomogeneous case in 1954. Ten years later, Schmidt in his far-reaching  paper \cite{schfracparts}, established the quantitative strengthening of Theorem~KS. In short, the main theorem in \cite{schfracparts} implies that with $\psi$ decreasing and $\ca{A} = \N$,  the asymptotic counting statement \eqref{countlacresult} is valid  for $m$-almost all $x \in \I$. In the case $\ca{A}$ is a lacunary sequence of natural numbers,  the assumption that $\psi$ is decreasing can be dropped and so the precise analogue of Theorem \ref{asmplacthm} holds for Lebesgue measure $m$ -- see \cite[Theorem~7.3]{harman}. Motivated by the classical theory of Diophantine approximation on manifolds,  suppose we restrict the points of interest in $ W_{\ca{A}}(\gamma; \psi) $  to lie in some subset $F$ of $\I$.   Assume that $m(F)=0$ -- this is trivially the case if $\dim F < 1$.   Then, the  Lebesgue measure statements  just described provide no information regarding the `size' of the set of inhomogeneous $\psi$-well approximable real numbers restricted to $F$; we always have that $$ m \big(W_{\ca{A}}(\gamma; \psi)  \cap F \big) = 0 $$  irrespective of $\gamma$, $\psi$ and $\ca{A}$.   With this in mind, let $\mu$ be a  non-atomic probability measure supported on $F$. Then,  under some natural conditions on $ F$, $\mu$  and $\ca{A}$, the goal is to obtain an analogue of Theorem KS for  the `size' of $W_{\ca{A}}(\gamma; \psi)  \cap F $  expressed in terms of the  measure $\mu$.  Indeed, generically it would not be unreasonable to  expect that
\begin{equation} \label{generic}
\mu\big( W_{\ca{A}}(\gamma; \psi)  \cap F \big) =  0    \ \  \mbox{(resp. $ = 1$)}   \quad {\rm if} \quad   \textstyle{\sum\limits_{n=1}^{\infty}}\psi(q_n) < \infty    \ \  \mbox{(resp.  $=\infty$)}.
\end{equation}
Such a statement would be precisely in line with the conjectured `Dream Theorem'  \cite[\S6.1.3]{durham} for Diophantine approximation on non-degenerate manifolds. For a basic introduction  to the theory of metric Diophantine approximation including the manifold theory, see \cite{durham} and references within.

 The following statement concerning the `size' of  $W_{\ca{A}}(\gamma,\psi)  \cap F $ is a direct consequence of Theorem \ref{asmplacthm}. It simply makes use of the fact that the theorem implies that for  $\mu$-almost all $x\in F$,  the quantity $R(x,N) $ is bounded if $\Psi(N)$ is bounded and will tend to infinity if  $\Psi(N)$ tends to infinity.

 \begin{cor}  \label{corr-asmplacthm}  Let $\mu$ be a probability  measure supported on a subset $F$ of
$ \ \I \, $.  Let $\ca{A}= (q_n)_{n\in \N} $ be a lacunary sequence of natural numbers.
Let $\gamma\in\I  $ and $\psi:\mathbb{N}\rightarrow \I$ be a real, positive function.  Suppose there exists a constant $A > 2$,  so that   \eqref{decaylac} is satisfied.
Then
$$ \mu\big(W_{\ca{A}}(\gamma; \psi) \cap F \big) =\left\{
\begin{array}{ll}
  0 & \mbox{if} \;\;\; \sum\limits_{n=1}^{\infty}\psi(q_n)  <\infty\;
      ,\\[2ex]
  1 & \mbox{if} \;\;\; \sum\limits_{n=1}^{\infty}\psi(q_n)
      =\infty \; .
\end{array}\right.$$
\end{cor}
\medskip

\begin{rem}\label{rem3} A  consequence of the general convergence result stated in \S\ref{bl}, is that  we can get away with $A>1$ in the  convergence case of the above corollary.  In fact,  the following decay rate suffices: for some $\epsilon > 0$ arbitrarily small
 \begin{equation*} \label{lacdelmu2sv}
\widehat{\mu}(t)= O\left( (\log |t|)^{-1} (\log \log |t|)^{-(1+ \epsilon)}     \right)   \hspace{6mm} {\rm as}  \quad  |t|\to\infty \,  . \vspace{2mm}
\end{equation*}
   \end{rem}

\begin{rem}\label{rem4}  Note that in view of Remark \ref{rem1},  whenever we are in the divergence case of the corollary  we are able to conclude that $\mu$-almost all numbers in $W_{\ca{A}}(\gamma, \psi) \cap F$ are normal.
   \end{rem}

The upshot of the results presented so far is that if $\ca{A}$ is a lacunary sequence of  natural numbers and if $\widehat{\mu}$ decays  quickly enough,  then
we are in pretty good shape with our understanding of the  set $ W_{\ca{A}}(\gamma; \psi) \cap  F$  -- both in terms of counting solutions (cf. Theorem~\ref{asmplacthm}) and size (cf.  Corollary \ref{corr-asmplacthm}).    The obvious question that now comes to mind is: what can we say if the growth of the sequence is slower than lacunary?  More specifically,  is the statement of  Theorem~\ref{asmplacthm}  valid for the sequence $ \mathcal{A}=\{ 2^a 3^b :\, a,b\in\Z_{\ge 0} \} $ if we choose the decay rate constant $A$ in \eqref{decaylac} large enough?  Before addressing this, it is worth comparing the above lacunary results with previous related works.

\vspace{2mm}

\subsubsection{Connection to previous works \label{cpw}}

\noindent Let $x=[a_1, a_2, ...]$ represent the regular continued fraction expansion of $ x \in \I$, and as usual let $p_n/q_n:=  [a_1, a_2, ...,a_n] $ denote its $n$--th convergent. Recall that $ q_n  \| q_nx   \|  \leq 1$ for any $n \in \N$.  Given $M \in \N$, let $F_M$ denote the set of real numbers in the unit interval with partial quotients  bounded above by $M$. Thus
$$ F_M := \left\{ x \in \I \, : \, x=[a_1, a_2, ...] \ \ {\rm with } \ \ a_i \,  \leq \, M \ \ \ \text{for all \ } i \in \N \right\} . $$
It is well known that any $F_M$ is a subset of the set $\Bad$ of badly approximable numbers; indeed
$$
\bigcup_{M \in \N} F_M   \  = \    \Bad:= \big\{ x \in \I :  \liminf_{q \to \infty} q \|q x \| > 0  \big\}  \, .
$$
In a pioneering  paper \cite{Kaufman},  R. Kaufman  showed that for any $M \ge 3$ the set $F_M$ is an $M_0$-set. More precisely,  he constructed a probability measure $\mu$ supported on $F_M$ satisfying the  decay property:
\begin{equation} \label{kaufdec}
\widehat{\mu}(t)= O\left( |t|^{-0.0007}  \right)   \hspace{6mm} {\rm as} \quad  |t|\to\infty \, .  \vspace{2mm}
\end{equation}
 Kaufman's construction was subsequently refined  by Queffel\'ec $\&$ Ramar\'e \cite{QR}. In particular, they showed that $F_2$ also supports a probability measure with polynomial decay.

\begin{rem}\label{rem5a}
 For any $M \ge 2$,  the Kaufman measure $\mu$ supported on $F_M$ trivially satisfies the decay condition \eqref{lacdelmu2} and so it follows (see Remark \ref{rem1}) that $\mu$--almost all numbers in $F_M$ (and thus in $\Bad$)  are normal.  This observation is attributed to R.C. Baker -- see  \cite[Appendix: Problem 45]{ten} for further details.
\end{rem}

\begin{rem}\label{rem51} By construction, for any $M \ge 2$  the Kaufman measure $\mu$ supported on $F_M$  also satisfies the following desirable property:   for any $s < \dim F_M$, there exist constants
 $c, r_0>0$ such that $
\mu(B) \leq c \, r^s \, $
for any  ball $B$ with radius $r < r_0$.  This  together  with the Mass Distribution Principle \cite[\S4.1]{falc} implies that if $F \subset \I$ is such that $\mu(F) > 0$, then
\begin{equation} \label{kaufp5}
\dim F \geq   \dim F_M \, .
\end{equation}
\end{rem}

\vspace*{2ex}
A well known conjecture of Littlewood dating back to the nineteen thirties states that 
 $$
 \liminf_{q \to \infty} q \|q x \|  \,  \|q y \| = 0   \quad  \forall  \  x,y \in \I \, . \vspace*{-1ex}
 $$
  This statement is trivially true if either $x$ or $y$ are not in $\Bad$. The decay property \eqref{kaufdec} of the Kaufman measure was successfully  utilized in \cite{PV} to prove  the following statement for badly approximable numbers:  {\em given $x \in \Bad$, there exists a subset $\mathbb{G}(x)$ of $\Bad$ with full dimension such that for any $y \in \mathbb{G}(x)$,}
 \begin{equation} \label{pveq}
  q \|q x \|  \,  \|q y \| \, \leq \, 1/\log q  \hspace{3mm}  \text{\em for\  infinitely\  many }  q\in \N \, .
\end{equation}
Trivially, such $x, y \in \Bad$ satisfy Littlewood's conjecture with an explicit `rate of approximation' of $1/\log q$. The strategy  behind the proof is simple enough.  Let  $\ca{A}= (q_n)_{n\in \N} $  be the sequence of denominators of the convergents $p_n/q_n$  of the given $ x \in \Bad$. Then  \eqref{pveq} holds for the given $x$ and any $y \in \mathbb{G}(x),$ where
 $$
 \mathbb{G}(x) := \{ y\in \Bad: \|q_ny\| \leq 1/\log q_n \text{ for infinitely many } n\in\mathbb{N} \}
 $$
 The crux of the proof (see \cite[\S3]{PV}  for the details)  boils down to showing that  for any $M \ge 3$
  \begin{equation} \label{pveqsv}
\mu \big(W_{\ca{A}}(\gamma; \psi)  \cap F_M \big) > 0   \quad {\rm with}  \   \ \gamma =0  \quad {\rm and }  \ \    \psi(q) =  1/\log q  \, ,
\end{equation}
where $ \mu$ is the Kaufman measure  supported on $F_M$. Note that by definition,  $\mathbb{G}(x) = W_{\ca{A}}(\gamma; \psi)  \cap \Bad$ and so the desired full dimension statement follows on combining \eqref{pveqsv}, \eqref{kaufp5} and the fact that $\dim F_M \to 1$ as $M\to \infty$.  The proof of  \eqref{pveqsv} given in \cite{PV} makes  use of the explicit  choices of  $\psi$ and $\gamma$, and the fact that the Fourier transform of the Kaufmann measure has polynomial decay. It also explicitly exploits the fact that for any $x \in \Bad$, the sequence $\ca{A}= (q_n)_{n\in \N} $ arising from its convergents $p_n/q_n$ is not only lacunary but satisfies the additional property that
\begin{equation} \label{doublelacunary}
 \frac{q_{n+1}}{q_n}  \le K^*   \hspace{5mm}  (n \in \N)\, ,
\end{equation}
where $K^* \! :=K^*(x) > 1$ is a constant. Corollary \ref{corr-asmplacthm} improves the work carried out in  \cite{PV}  on four key fronts:
\begin{itemize}
  \item[(a)] It gives a full $\mu$-measure statement rather than just a positive $\mu$-measure statement.
  \item[(b)] It is for any  function $\psi:\mathbb{N}\rightarrow \I$  and any  $\gamma \in \I$ rather than the explicit choices given by~\eqref{pveqsv}.
  \item[(c)] It is for any  lacunary sequence $\ca{A}$ of natural numbers rather than those satisfying  the additional property~\eqref{doublelacunary}.
  \item[(d)] It  is for any subset $F$ of
$ \I $ that supports a  probability  measure $\mu$ with sufficient logarithmic transform  decay  rather than requiring polynomial decay.
\end{itemize}
\noindent Indeed, Corollary \ref{corr-asmplacthm} restricted to the homogeneous case ($\gamma=0$) establishes the zero-one law claim made in \cite[\S3.3]{PV}. Indeed, with $\psi(q) = 1/ (\log q \, \log \log q)$ it enables us to replace   \eqref{pveq}  by the stronger statement that
 $
\liminf_{q \to \infty} q \, \log q  \, \|q x \|  \,  \|q y \| = 0  $.
 To the best of our knowledge,  the   quantitative strengthening  of the corollary, namely Theorem \ref{asmplacthm}, is completely new. Within the context of Littlewood's conjecture it implies the following statement: {\em given $x \in \Bad$ and $\gamma \in \I$, there exists a subset $\mathbb{G}(x,\gamma)$ of $\Bad$ with full dimension such that for any $y \in \mathbb{G}(x,\gamma)$,}
 \begin{equation} \label{pveqcount}
   \# \big\{ 1\leq q \leq N :  q \, \|q x \|  \,  \|q y - \gamma \| \, \leq \, 1/\log q    \big\} \gg \log \log N \, .
\end{equation}
With $\gamma =0$, this establishes the  followup quantitative claim made in \cite[\S3.3]{PV}.   To  some extent, in the inhomogeneous case  it would be more natural and desirable to  establish   \eqref{pveqcount} in which $\mathbb{G}(x,\gamma)$ is defined as a subset of
$$
\Bad_\gamma := \big\{ x \in \I :  \liminf_{q \to \infty} q \|q x - \gamma \| > 0  \big\}  \, .
$$
The point is that for $ y \notin \Bad_\gamma $, the  corresponding inhomogeneous
version of Littlewood's conjecture; namely
 $$\liminf_{q \to \infty} q \, \|q x \|  \,  \|q y - \gamma \| = 0  \, \quad  \text{ for all  }  x,y \in \I \, ,
 $$ is trivially true.  The major obstacle preventing us from establishing the desired inhomogeneous statement  \eqref{pveqcount} with  $\mathbb{G}(x,\gamma) \subset \Bad_\gamma $ is that we are unable to prove the existence of a probability  measure  $\mu $  supported on a subset  of $\Bad_\gamma$  with  transform decay as in Theorem~\ref{asmplacthm}.  With this in mind, we suspect that a statement of the following type is true.

 \begin{claim}
 Let $\gamma \in \I$.  Then for any sufficiently small constant $c > 0$, the set $$\Bad_\gamma(c):= \big\{ x \in \I :   q \|q x - \gamma \| > c  \ \ \forall  \ q \in \N  \big\}  \,  $$  supports a probability measure $\mu$ with  $ \widehat{\mu}$ satisfying   \eqref{decaylac} for some $A > 2$.
 \end{claim}

\noindent Trivially,  $ \Bad_\gamma(c) $ is a  subset of $   \Bad_\gamma$  for any $c >0$.  We suspect that the above claim is true with $ \widehat{\mu}$ satisfying polynomial decay as in the homogeneous case.  Having said this, as far as we are aware,  we do not even know that the sets $ \Bad_\gamma(c) $ are  $M_0$-sets. Establishing this is of independent interest and  can be regarded as a first step towards Claim 1.

Note that if Claim 1  is true then  it follows (see Remark \ref{rem1}) that $\mu$--almost all numbers in $\Bad_\gamma(c)$ are normal.  Proving the existence of normal numbers in $\Bad_\gamma$ is an interesting problem in its own right.  To the best of our knowledge, currently we do not know of any such numbers when $\gamma$ is irrational.

 \begin{claim}
For any $\gamma \in \I$,   there exists at least one normal number in $\Bad_\gamma$.
 \end{claim}
\noindent We suspect that the set of normal numbers in $\Bad_\gamma$ is of full dimension irrespective of the choice of $\gamma$. This is true in the homogenous case.

\vspace*{2ex}

 The paper \cite{Haynes}  further develops the ideas from \cite{PV}.   Within the setting of the inhomogeneous version of Littlewood's conjecture stated above,  Haynes, Jensen $\&$  Kristensen establish  the following result.

\medskip

\begin{thhjk}
{\em Fix $\epsilon  > 0$ and let  $X$ be a countable subset of $\Bad$.  Then there exists a subset $\mathbb{G}(X)$ of $\Bad$ with full dimension such that for any $y \in \mathbb{G}(X)$, $x \in X$, and $\gamma \in \I$}
 \begin{equation} \label{pveqhjk}
  q \|q x \|  \,  \|q y - \gamma \| \, \leq \, 1/(\log q)^{1/2 - \epsilon}   \hspace{3mm}  \text{\em for\  infinitely\  many }  q\in \N \, .
\end{equation}
\end{thhjk}

\noindent A few comments are in order.  As we shall see in a moment, the fact that the statement is true for a countable subset $X \subset \Bad$ is not the meat.    The strength of the theorem lies in the fact that set  $\mathbb{G}(X)$ is independent of the inhomogeneous factor $\gamma$.  With this in mind,  we briefly describe the key step in the proof of Theorem HJK.
Fix $\lambda \in (0,1]$, and let $\psi_\lambda (q) :=  1/(\log q)^{\lambda} $ and $\Psi_\lambda(N):=\textstyle{\sum_{n=1}^N\psi_\lambda(q_n)} $.  Note  that  $\Psi_\lambda(N) \to \infty $ as $N \to \infty$ for  $\lambda \le 1$. It is this  that is exploited  in the proof rather than actually `counting solutions'  -- see Remark \ref{rem9sv} below.    Fix $M \ge 3 $, and let  $x  \in F_M$ and   $\ca{A}= (q_n)_{n\in \N} $  be the sequence of natural numbers arising from the denominators of the convergents $p_n/q_n$  of $ x$.  Finally, let $ \mu$ be the Kaufman measure  supported on $F_M$ and
\begin{equation} \label{county}
 \mathbb{G}_M(x,\psi_\lambda ) := \big\{ y \in  F_M:   R(y,N;\gamma,\psi_\lambda, \ca{A})  \geq  2 \Psi_\lambda(N) \  \forall  \, N \gg 1  ,  \, \forall \,  \gamma \in \I \big\}   \, ,
\end{equation}
where $R$ is the counting function given by \eqref{countdef}.   The key to the proof of the theorem lies in showing that for any $ \lambda  \in (0,1/2)$,
 \begin{equation} \label{pveqhjksv}
  \mu(\mathbb{G}_M(x,\psi_\lambda ) ) = 1 \, .
\end{equation}
It follows that $\mu \big( \cap_{x \in X} \mathbb{G}_M(x,\psi_\lambda) \big) = 1$ for
any countable set $X \subset F_M$. In turn, this leads to a theorem that is
valid for any countable set of badly approximable numbers.
The proof of \eqref{pveqhjksv} makes nifty use of the Erd\H{o}s-Tur\'{a}n inequality \cite[Corollary 1.1]{ten}  and a standard effective version of the (divergent) Borel-Cantelli Lemma  (see Lemma 2.3  in \cite{harman}) which  we shall exploit too.  As in \cite{PV},   the proof also  makes use of the explicit  nature  of  the function  $\psi_\lambda$ and the fact that the Fourier transform of the Kaufmann measure has polynomial decay.

 In order to compare the above work of  Haynes, Jensen $\&$  Kristensen  \cite{Haynes} to our work, fix $\gamma \in \I$ and with
 \eqref{county} in mind,  let
 \begin{equation} \label{county23}
 \mathbb{G}_M(x,\psi_\lambda, \gamma ) := \big\{ y \in  F_M:   R(y,N;\gamma,\psi_\lambda, \ca{A})  \geq  2 \Psi_\lambda(N) \  \forall  \, N \gg 1  \big\}   \, .
\end{equation}
Then,
\begin{equation} \label{county23*}
  \mathbb{G}_M(x,\psi_\lambda)   =   \bigcap_{\gamma \in \I}   \mathbb{G}_M(x,\psi_\lambda, \gamma )   \, .
\end{equation}
A simple consequence of Theorem \ref{asmplacthm} is that for any $\gamma \in \I$ and  any $ \lambda  \in (0,1]$,
 \begin{equation} \label{pveqhjksv0}
  \mu(\mathbb{G}_M(x,\psi_\lambda, \gamma ) ) = 1 \, .
\end{equation}
Indeed, this statement  with $\lambda=1$ is at the heart of establishing \eqref{pveqcount}.  Also, note that to apply Theorem  \ref{asmplacthm} we only need $\mu$ to have sufficient logarithmic transform  decay.  However, since the intersection in \eqref{county23*} is uncountable we are unable to conclude the full measure statement \eqref{pveqhjksv}.  Just to reiterate, the approach taken in \cite{Haynes}  yields \eqref{pveqhjksv} as long as  $\lambda < 1/2 $.  Unfortunately, it is not at all clear how to modify the method  for $\lambda \geq 1/2$  even within the context of establishing the weaker statement \eqref{pveqhjksv0}.

\begin{rem}\label{rem8sv}  Since  $\Psi_\lambda(N) \to \infty $ as $N \to \infty$ for  $\lambda \le 1$, it follows via  \eqref{pveqhjksv} that for any  $\lambda \in (0,1/2)$
$$\mu \big(  \bigcap_{\gamma \in \I}W_{\ca{A}}(\gamma, \psi_\lambda) \cap F_M   \big)  = 1    \, , $$
and so
$ \mu \big( W_{\ca{A}}(\gamma, \psi_\lambda) \cap F_M   \big)  = 1 \, $
for any  $\gamma \in \I$.
Clearly,  Corollary  \ref{corr-asmplacthm} implies the latter full measure statement for any $\lambda \in (0,1]$ but it cannot  be exploited to  yield  the former.
\end{rem}

 \begin{rem}\label{rem9sv}
In \cite[\S4]{Haynes}, the authors suggest that the approximating  function   $1/(\log q)^{1/2 - \epsilon}$ in the right hand side of the inequality in \eqref{pveqhjk} can in fact be replaced by $1/\log q$.  This would follow if we could prove \eqref{pveqhjksv} with $\lambda =1$.  In fact, it would suffice to prove this for a weaker form of the set $\mathbb{G}_M(x,\psi_\lambda )$ in which the growth condition  on  the counting function $R$  appearing in \eqref{county} is  replaced by the condition that  $ R(y,N;\gamma,\psi_\lambda, \ca{A}) \to \infty $ as $N \to \infty$.

\end{rem}

For the sake of completeness, we finish with a short discussion on Salem sets. For both consistency and simplicity, we restrict the discussion to subsets $F$ of $\I$.  The {\em Fourier dimension} of $F \subset \I$ is defined by
\begin{equation*}
\dim_F F  := \sup\left\{0\leq \eta\leq 1:  \exists \, \mu \in M_1^+(F)\, \text{ {\rm with }} \hspace{1mm}  \widehat{\mu}(t)= O(|t|^{-\eta/ 2})  \text{ {\rm as }} |t|\to\infty \  \right\}. \vspace{2mm}
\end{equation*}
where $M_1^+(F)$ denotes the set of all positive Borel probability measures with support in  $F$.  A simple consequence of the classical result of Frostman  mentioned right at the start of the paper (namely,  if $\widehat{\mu}(t)= O(|t|^{-\eta/ 2}) $  then $\dim F  \geq \min \{1, \eta\} $), is that the Fourier dimension is bounded above by the Hausdorff dimension. A set $F$ with   $\dim_F F = \dim F $ is called a {\em Salem set}. Observe that Theorem  \ref{asmplacthm} and its corollary are applicable to any Salem set with strictly positive dimension.  In fact, this is also true for the non-lacunary results presented in the next section.
To the best of our knowledge, it is unknown whether or not the badly approximable  subsets $F_M $   are Salem sets.  However, the story is quite different for well approximable subsets of $\I$. To start with, given $\tau \geq 1 $, let $\psi_\tau (q):= q^{-\tau}$ and consider the classical homogeneous set $W(0, \psi_\tau)  $ of $ \tau$-well approximable numbers. By definition, this corresponds to $W_{\ca{A}}(\gamma, \psi)$  given by \eqref{wellset} with $ \ca{A}=  \N$, $\gamma = 0$  and $\psi=\psi_\tau$.   By Dirichlet's theorem,  $W(0, \psi_\tau) = \I$ when $\tau = 1$ and for $\tau > 1$,  a classical theorem of Jarnik and Besicovitch (see \cite[\S1.3.2]{durham})  states that $\dim   W(0, \psi_\tau)   =  2/(\tau + 1) $.
In another elegant paper \cite{K}, Kaufman constructed a probability measure $\mu$ supported on  $W(0, \psi_\tau)$ for any $\tau > 1$ satisfying the  decay property:
\begin{equation} \label{kaufdec5}
\widehat{\mu}(t)=  |t|^{-\frac{1}{\tau+1}} \  o\left(\log|t| \right)   \hspace{6mm} {\rm as} \quad  |t|\to\infty \, .  \vspace{2mm}
\end{equation}
The upshot is that  $W(0, \psi_\tau) $ is a Salem set for any $\tau > 1$.
 Bluhm \cite{Bl} subsequently generalised this statement to  arbitrary decreasing  functions $\psi$. In short, for $\psi$-well approximable sets $W(0, \psi) $ the quantity $\tau$ in the Jarnik-Besicovitch theorem and in \eqref{kaufdec5} is replaced by the quantity $\lambda(\psi) := \liminf_{q \to \infty} -\log \psi(q)/ \log q \,  $; namely the lower order at infinity of the function $1/\psi$.  In the last couple of years,  Hambrook \cite{Hambrook} and independently Zafeiropoulos \cite{malaka} have extended  Bluhm's work to the inhomogeneous setup.  Thus, for any $\gamma \in \I$ and any  real, positive decreasing  function $\psi:\mathbb{N}\rightarrow \I$ with $\lambda(\psi) >1$, we now know that the set  $W(\gamma, \psi) $ is a Salem set.  Hambrook  actually does a lot more;  for example,  he obtains results for the higher dimensional analogues of the  restricted `denominators' sets $W_{\ca{A}}(\gamma, \psi) $.

 \vspace{2mm}

The main purpose of this section was to compare our results for lacunary sequences (namely, Theorem \ref{asmplacthm} and Corollary \ref{corr-asmplacthm}) with previous related works.
 As far as we are aware, if the growth of the sequence  is slower than lacunary, then nothing is known even within the context of Corollary \ref{corr-asmplacthm}, let alone Theorem \ref{asmplacthm}.

\subsection{Beyond lacunarity \label{bl} }

\noindent With reference to Corollary \ref{corr-asmplacthm}, in the case of convergence we are able to prove the following stronger statement for general sequences.

\begin{thm}  \label{mainCONV}  Let $\mu$ be a probability  measure supported on a subset $F$ of
$ \ \I \, $.  Let $\ca{A}= (q_n)_{n\in \N} $ be a sequence of natural numbers.
Let $\gamma\in\I  $ and $\psi:\mathbb{N}\rightarrow \I$ be a real, positive function. Suppose that at least one of the following two conditions is satisfied:
\begin{equation} \label{cond2_2}
\sum_{n=1}^{\infty} \; \max_{k\in\Z / \{ 0\} }|\hat{\mu}(kq_n)|< \infty
\end{equation}
\begin{equation} \label{cond2}
\sum_{n=1}^{\infty}   \, \sum_{k\in\Z / \{ 0\} }  \frac{|\widehat{\mu}(kq_n)|}{|k|}< \infty .
\end{equation}
Then
\begin{equation} \label{convcondition}
\mu(W_{\ca{A}}(\gamma; \psi))=0 \quad   \mbox{ if }  \quad   \sum\limits_{n=1}^{\infty}\psi(q_n)< \infty.
\end{equation}
\end{thm}

\medskip

 It is easily verified that Theorem \ref{mainCONV} implies the convergence case of Corollary \ref{corr-asmplacthm}.  Indeed, if $\mathcal{A}=(q_n)_{n=1}^{\infty}$ is lacunary and $\mu$ satisfies condition \eqref{decaylac}  for some $A > 1$, then
 \begin{eqnarray*}
\sum_{n=1}^{\infty} \; \max_{k\in\Z / \{ 0\} }|\hat{\mu}(kq_n)|  &\ll&
\sum_{n=1}^{\infty} \; \frac{1}{ (\log q_n)^{A} }   \\
&\ll& \sum_{n=1}^{\infty}   \frac{1}{n^{A}}  \ <   \ \infty
\end{eqnarray*}
\noindent and so condition  \eqref{cond2_2} of Theorem \ref{mainCONV} is satisfied. The upshot is that Theorem \ref{mainCONV} implies the convergence case of Corollary \ref{corr-asmplacthm} under the weaker assumption that  \eqref{decaylac} is satisfied  for some $A > 1$.  As pointed out above in Remark \ref{rem3}  we can actually  get away with even slightly less.

\medskip

In the case of divergence and within the context of Theorem \ref{asmplacthm}, we are able to go beyond lacunarity if we restrict the prime divisors of the elements in the given integer sequence  to lie in a finite set.  More precisely, fix $k \in \N$ and let
\begin{equation} \label{def_cS}
\cS:=\left\{p_1,\dots,p_k \right\}
\end{equation}
be a set of $k$ distinct primes $p_1,\ldots, p_k$.  In turn, given $\cS$ let
\begin{equation} \label{kprimes}
\cA_{\cS}:=\left\{\textstyle{\prod\limits_{i=1}^k}   \, p_i^{a_i}  \, : \,  a_1,\ldots, a_k \in\Z_{\geq 0}\right\}
\end{equation}
 be  the set of positive integers with prime divisors restricted to $\cS$.  In other words,  $\cA_{\cS}$ is precisely the set of smooth numbers over $\cS$. Obviously,  if $ p$ is the largest prime among $\cS$  then by definition, every integer in $\cA_{\cS}$ is $p$-smooth.

 The following constitutes our main result for sequences that are not necessarily lacunary.

\begin{thm} \label{main_q}
 Let $\mu$ be a probability  measure supported on a subset $F$ of
$ \ \I \, $.  Let
\begin{equation*} 
\ca{A}= (q_n)_{n\in \N} \subseteq \cA_{\cS}
\end{equation*}
be an increasing  sequence of natural numbers.
Let $\gamma\in\I  $ and $\psi:\mathbb{N}\rightarrow \I$ be a real, positive function.  Suppose there exists a constant $ A > 2k $
so that \eqref{decaylac} is satisfied.
Then, for any $\epsilon>0$ the counting function $R(x,N)$ satisfies
\begin{eqnarray}\label{countFSPresult}
R(x,N) =  2\Psi(N)+O\Big(\Psi(N)^{1/2}\left(\log\big(\Psi(N)+2\right)\big)^{2+\varepsilon}\Big)
\end{eqnarray}
for $\mu$-almost all $x\in F$, where  $ \Psi(N) $ is given by
\eqref{def_Psi}.
\end{thm}

Theorem \ref{main_q}  answers the question raised at the end of \S\ref{motlacres} concerning  the specific  sequence $ \mathcal{A}=\{ 2^a 3^b \,:\, a,b\in\Z_{\ge 0} \} $.   Namely, the analogue of  Theorem~\ref{asmplacthm} is valid for this specific  $\cA$ if the decay rate constant $A$ in \eqref{decaylac} is strictly larger than four.    Note that in the case of one prime $p$, the corresponding sequence $\cA=\{ p^n : n\in \N \}  $ is trivially lacunary. However, due to the fact that $\cS$ is a finite set of primes the above theorem  does not cover arbitrary lacunary sequences unlike Theorem~\ref{asmplacthm}.  Nevertheless,  Theorem \ref{main_q}  does give a better error term than Theorem~\ref{asmplacthm}.

 The following statement  is a direct consequence of Theorem~\ref{main_q}. It follows in exactly the same way as Corollary~\ref{corr-asmplacthm} follows from Theorem~\ref{asmplacthm}.

\begin{cor}  \label{main_K}  Let $\mu$ be a probability  measure supported on a subset $F$ of
$ \ \I \, $.  Let $\ca{A}= (q_n)_{n\in \N} \subseteq \cA_{\cS} $ be an increasing  sequence of natural numbers.
Let $\gamma\in\I  $ and $\psi:\mathbb{N}\rightarrow \I$ be a real, positive   function.  Suppose there exists a constant $A >  2 k $,  so that \eqref{decaylac} is satisfied.
Then
$$ \mu\big(W_{\ca{A}}(\gamma; \psi) \cap F \big) =\left\{
\begin{array}{ll}
  0 & \mbox{if} \;\;\; \sum\limits_{n=1}^{\infty}\psi(q_n)  <\infty\; ,\\[2ex]
  1 & \mbox{if} \;\;\; \sum\limits_{n=1}^{\infty}\psi(q_n)  =\infty \; .
\end{array}\right.$$
\end{cor}

\medskip

It can be verified that any given increasing sequence $\mathcal{A}=(q_n)_{n=1}^{\infty}  \subseteq \cA_{\cS} $ satisfies the growth condition \begin{equation}\label{hyp_q}
\log q_n \,  >  \,  C  \,  n^{1/B}    \qquad \forall \ n \ge 2 \, ,
\end{equation}
for some constants  $B \ge  1$  and $C > 0$. Indeed, we can always get away with
$B=k$ and $C= (\log 2)/2 $ irrespective of the choice of $\mathcal{A} \subseteq \cA_{\cS} $ since $\cA_{\cS}$ satisfies \eqref{hyp_q}  with these choices of $B$ and $C$ (see Appendix~B of \S\ref{app}). Hence,
\begin{eqnarray*}
\sum_{n=1}^{\infty} \; \frac{1}{ (\log q_n)^{A} }
&\ll& \sum_{n=1}^{\infty}   \frac{1}{n^{A/k}}  \ <   \ \infty \,
\end{eqnarray*}
for any $A > k$ and it follows via  Theorem \ref{mainCONV} that we only require that $ A > k$  in the convergence case of the above corollary. The stronger condition that $ A > 2k$  on the decay rate constant $A$ is required to establish Theorem~\ref{main_q}.  This  condition then carries over to the corollary since the divergence case is directly deduced from the theorem.  We suspect that within the context of the corollary,  $ A > k$ would suffice even in the divergence case.

\medskip

\begin{rem} \label{rem19}
 As the following  example indicates, it is  necessary to impose some condition (either directly or indirectly) on the growth of the sequence  in Theorem \ref{main_q} and indeed its corollary.    With reference to \S\ref{cpw}, let $\mu$ be the Kaufman measure supported on the badly approximable subset $F_M \subset \Bad $ for some $ M \geq 2$.   Let $ \cA = \N $, $\gamma = 0$ and $ \psi(q) =  1/ (q \log q) $.  Then the sum appearing in Corollary~\ref{main_K} diverges but in view of the definition of   $\Bad $,  we trivially have that $$
 W_{\ca{A}}(0; \psi) \cap F_M  = \emptyset \, .
 $$
 \end{rem}

 \medskip

 We will deduce Theorem \ref{main_q} from a general statement for sequences satisfying
the growth condition \eqref{hyp_q} and the following  `separation' condition.
Let $\ca{A}= (q_n)_{n\in \N} $  be an increasing sequence of natural numbers and let $\alpha\in(0,1)$ be a real number.  We say that $ \ca{A} $  is \emph{$\alpha$-separated} if there exists a constant  $m_0  \in \N$  so that for any integers $ m_0\leq m<n$, if
$$
1 \leq |sq_m-t q_n|< q_m^{\alpha}
$$
for some $s,t\in\N$,   then
$$
s>m^{12}.
$$

\noindent Note that when  $\mathcal{A}$ is lacunary  then the growth condition \eqref{hyp_q} is trivially satisfied with $B=1$ but $\mathcal{A}$ need not be $\alpha$-separated\footnote{For example, consider the sequence $\ca{A}= (q_n)_{n\in \N} $ given by  $q_n:=2^n+\varepsilon_n$, where $\varepsilon_n=1$ for even $n$ and $\varepsilon_n=0$ for odd $n$. This sequence is clearly  lacunary with $q_{n+1}/q_n\geq 8/5$ (the minimum of the left hand side is attained at $n=2$). Moreover,  $q_{n}-2q_{n-1}=1$  whenever $n$ is even and so $\cA$ is not $\alpha$-separated.}.  Thus, we can not deduce Theorem~\ref{asmplacthm} directly from the following result.

\medskip


 \begin{thm} \label{mainSV}
Let $\mu$ be a probability  measure supported on a subset $F$ of
$ \ \I \, $.  Let $\ca{A}= (q_n)_{n\in \N} $ be an increasing sequence of natural numbers that (i) satisfies the growth condition  \eqref{hyp_q} for some constants  $B \ge  1$  and $C > 0$, and (ii) is $\alpha$-separated.
Let $\gamma\in\I  $ and $\psi:\mathbb{N}\rightarrow \I$ be a real, positive function.  Suppose there exists a constant
\begin{equation} \label{ie_Delta_2B}
A  \, >  \,  2  B \, ,
\end{equation} so that \eqref{decaylac} is satisfied.
Then, for any $\epsilon>0$ the counting function $R(x,N)$ satisfies
\begin{equation} \label{countFSPresultsv}
\begin{array}{ll}
 R(x,N)  \ = \   2\Psi(N)  + O\Big(\left(\Psi(N)+E(N)\right)^{1/2}\left(\log(\Psi(N)+E(N)+2\right)\big)^{2+\varepsilon}\Big)
\end{array}
\end{equation}
for $\mu$-almost all $x\in F$,  where  $ \Psi(N) $ is given by
\eqref{def_Psi} and
\begin{equation} \label{error2sv}
E(N) \ := \ \mathop{\sum\sum}_{1\leq m<n\leq N} (q_m,q_n)  \ \min\left(\frac{\psi(q_m)}{q_m},\frac{\psi(q_n)}{q_n}\right)  \, .
\end{equation}
\end{thm}

\vspace*{3ex}

We have already mentioned that any increasing sequence $\ca{A} \subseteq \cA_{\cS} $   satisfies the  growth condition  \eqref{hyp_q}  with $B = k$.   Thus, Theorem \ref{main_q} will follow from Theorem \ref{mainSV} on showing that any such sequence is $\alpha$-separated and that the gcd term $E(N)$ appearing in the `error'  term is less than the `main' term $\Psi(N)$.   This will be the subject of \S\ref{deducing} and is very much self-contained.  In short,  the  key ingredient towards  showing $\alpha$-separation is an explicit linear forms in logarithms result by Baker and W\"ustholz,  while the key to dealing with the the gcd term $E(N)$ is to show that the sum
\begin{equation} \label{hot}
\sum_{m=1}^{n-1}
\frac{(q_m,q_n)}{q_n}      \qquad  (n \geq 2)
\end{equation}
is bounded from above by an absolute constant that depends only on the set $\cS$ of primes.

\medskip

\begin{rem} \label{rem19onemore}
The conclusion of Theorem \ref{mainSV} is in line with its Lebesgue measure counterpart \cite[Theorem 3.1]{harman}.  The latter essentially dates back to a paper of LeVeque from 1959.
\end{rem}

 \medskip

\begin{rem} \label{remAsy} Observe that under the hypotheses of  Theorem \ref{mainSV}, if $\Psi(N) \to \infty $ as $N \to \infty$ and
\begin{equation}  \label{needasy}
E(N) \, = \,  O \Big( \Psi(N)^{2 - \epsilon} \Big)
\end{equation}
for some $\epsilon>0$, then $ R(x,N)  $ is asymptotically equal to  $  2\, \Psi(N) $
for $\mu$-almost all $x\in F$.  Clearly, this together with Theorem~\ref{mainCONV}  trivially implies the conclusion of Corollary \ref{main_K} under the significantly  weaker assumptions of Theorem~\ref{mainSV}.
\end{rem}

 \medskip

 \begin{rem}  \label{remhypF}

 As already  mentioned,  we  will deduce Theorem \ref{main_q} from Theorem \ref{mainSV} and in order to do so  we show that  any  sequence $\ca{A} \subseteq \cA_{\cS} $    is $\alpha$-separated.  This is precisely the statement of Proposition~\ref{prop_sum_S_m_n_F} in \S\ref{svalpha} and its proof makes direct use of the fact  that  \eqref{hyp_q} automatically holds for such sequences.   We shall show in \S\ref{rem_Hypothesis_F}, that the proof can be  easily adapted to prove  that if an increasing sequence $\cA =(q_n)_{n\in\N}$  of natural numbers  satisfies:
\begin{itemize}
\item[] \vspace{-0ex} \begin{itemize}
\item[(i)] the growth condition  \eqref{hyp_q} for some constants  $B \ge  1$  and $C > 0$,
\item[(ii)] there exist constants $n_0 \in \N$ and  $D \in \N $, such that  for any $n \geq n_0$
\begin{itemize}
\item[] \begin{itemize}
\item[(a)]  $  \ \# \left\{  \;p \ {\rm prime} :  \,  p | q_n \right\}  \, \leq \, D $
\vspace{0.5ex}
\item[(b)] \  if prime $p| q_n$,  then
  $\displaystyle{ \log p  \, \leq \,   \left(\log q_n\right)^{\frac{1-\epsilon}{2D}} \,  } $ for some $\epsilon>0$,
\end{itemize}
\end{itemize}
\end{itemize}
\end{itemize}
then $\cA$ is $\alpha$-separated.  We say that \emph{$\ca{A}$ satisfies Property D} if the growth condition (i) and condition (ii) on its prime divisors are satisfied. It is clear that if $ \ca{A} \subseteq \cA_{\cS} $ then it automatically satisfies Property~D.   Indeed the  latter significantly broadens the explicit class of sequences for which Theorem \ref{mainSV} applies. The draw back is that for an arbitrary sequence satisfying Property~D it is not possible to control the  gcd term $E(N)$ appearing in the `error'  term of \eqref{countFSPresultsv}.
Indeed, there exist  (see Appendix~\!C of \S\ref{app})   sequences satisfying Property D and associated functions $\psi:\mathbb{N}\rightarrow \I$ such that for any $T>0$
\begin{equation}\label{baden}
E(N) \, \gg \, \Psi(N)^T  \,   \qquad \forall \ \  N \in \N.
\end{equation}
As a consequence, in general we are not even able to show \eqref{needasy} with $\epsilon =0$  let alone with $\epsilon =1$ (that enables us to reduce \eqref{countFSPresultsv} to \eqref{countFSPresult})
as in the situation when $\ca{A} \subseteq \cA_{\cS}$.
Nevertheless, it is relatively straightforward  to construct explicit sequences $\ca{A} \nsubseteq \cA_{\cS}$ (for any $\cS$) that satisfy Property D and for which $E(N)  = O ( \Psi(N))$. We stress that for such sequences,  the counting function $R(x,N)$ satisfies
\eqref{countFSPresult}
for $\mu$-almost all $x\in F$.  We  now show that perturbing a given sequence $\cA_{\cS}=(q_n)_{n\in\N}$  in the following manner has the desired effect.  Let ${\cal P} $ be a infinite set of distinct primes not in $\cS$ such that  $ \sum_{p \in \cP} p^{-1} \leq 1$.  Now choose $p_1 \in \cP$ such that $p_1 > q_1 $ and $p_1$ is larger than any of the $k$ primes in $\cS$.  Then choose $s$ sufficiently large so that
\begin{equation} \label{kl}
\displaystyle{ \log p_1  \, \leq \,   \left(\log q_s\right)^{\frac{1}{3(k+1)}} \,  }  \, ,
\end{equation}
and let $t_1 $ be the unique integer such that
$  q_{t_1}  <    \widetilde{q}_{t_1}:=q_sp_1   <  q_{{t_1}+1} \, $.
Now replace $q_{t_1} \in \cA_{\cS} $ by $ \widetilde{q}_{t_1}  $ and observe  that \eqref{hyp_q} holds for  $\widetilde{q}_{t_1}$  since   it holds for any element of $\cA_{\cS} $ with
$B=k$ and $C= (\log 2)/2 $.   The previous terms $ q_1, \ldots, q_{t_1-1} \in \cA_{\cS}$ remain unchanged. Next, choose $p_2 \in \cP$ such that $p_2 > q_{{t_1}+1} $ and choose $s$ sufficiently large so that
\eqref{kl} holds with $p_1$ replaced by $p_2$. Let $t_2 $ be the unique integer such that
$  q_{t_2} \, <  \,  \widetilde{q}_{t_2} \! :=q_sp_2  \, < \,  q_{{t_2}+1} \, $.
Now replace $q_{t_2} \in \cA_{\cS} $ by $ \widetilde{q}_{t_2}  $ and  keep the  previous terms $ q_{{t_1}+1}, \ldots, q_{{t_2}-1} \in \cA_{\cS}$ unchanged.  On repeating the above procedure indefinitely, we end up with  a sequence $\widetilde{\cA}=(\widetilde{q}_n)_{n\in\N}$ which by construction satisfies Property D with $D=k+1$  and contains arbitrarily large prime divisors. In view of the latter, $\widetilde{\cA} $ cannot be a subsequence of smooth numbers over a finite set of primes.  It now remains to  show that $E(N)  = O ( \Psi(N))$  and as already mentioned  earlier  this follows on showing that \eqref{hot} is bounded above  by an absolute constant $ \widetilde{K}$ for $\widetilde{\cA} $. The latter is not hard to show  assuming it holds  (which it does, see Theorem \ref{thm_gcd} in \S\ref{s_gcd}) for the sequence $\cA_{\cS}$.  To see this, fix $n \geq 2 $ and first assume that  $\widetilde{q}_{n} = q_n \in \cA_{\cS}$; that is, $n \neq t_i $ for some $i \in \N$.  Then, it follows that
\begin{eqnarray}  \label{love}
\sum_{m=1}^{n-1}
\frac{(\widetilde{q}_m,\widetilde{q}_n)}{\widetilde{q}_n}    & \leq &\sum_{m=1}^{n-1}
\frac{(q_m,q_n)}{q_n}   \ + \
\mathop{\mathop{\sum}_{m=1  : }}_{ \widetilde{q}_{m} \notin \cA_{\cS} }^{n-1}
\frac{(\widetilde{q}_{m},q_n)}{q_n}    \ \leq   \  K +  \sum_{p \in \cP} p^{-1}  \ \leq  \   K + 1 \, ,
\end{eqnarray}
where $K$ is  the absolute constant associated with $\cA_{\cS}$ that upper bounds \eqref{hot}.  Here we use the fact that  if $\widetilde{q}_{m} \notin \cA_{\cS}$,  then by definition  $\widetilde{q}_{m} = q_s p $ for some   $ s < m$ and $ p \in \cP $,  and it follows that $(\widetilde{q}_{m},q_n) =  (q_s,q_n) \leq q_s$.  Now suppose that  $\widetilde{q}_{n} \neq q_n \in \cA_{\cS}$;  that is $\widetilde{q}_{n} = q_s p $ for some   $ s < n$ and $ p \in \cP $. Let $p_*$ be the smallest prime amongst those in $\cS$ and let $u $ be the unique integer such that $  p_*^u < p < p_*^{u+1} $.  Then,  it follows that $  q_s p_*^{u} < \widetilde{q}_{n} < q_{n_{\!*}} \!\! :=q_s p_*^{u+1}  \in \cA_{\cS} $  for some $n_{\!*} > n $ and also note that  $(\widetilde{q}_{m},\widetilde{q}_n)   \leq (\widetilde{q}_{m},q_{n_{\!*}})$.  This together with \eqref{love} implies that
\begin{eqnarray*}
\sum_{m=1}^{n-1}
\frac{(\widetilde{q}_m,\widetilde{q}_n)}{\widetilde{q}_n}    & \leq & p_* \sum_{m=1}^{n_{\!*}-1}
\frac{(\widetilde{q}_m,q_{n_{\!*}})}{q_{n_{\!*}}}  \  \leq  \  p_* ( K + 1) \,  . \end{eqnarray*}
The upshot is that \eqref{hot} is  bounded above   by the  absolute constant
$ \widetilde{K}:= p_* ( K + 1)$ for all $n \geq 2 $.
\end{rem}

\medskip

\begin{rem} \label{rem19a}
 In the homogeneous case ($\gamma = 0)$, it is possible to give a direct proof of Corollary~\ref{main_K} that enables us to replace the condition that  $\ca{A} \subseteq \cA_{\cS} $ by the significantly milder condition that $\ca{A}$ satisfies Property D.  This will be the subject of a forthcoming note. In short, the overall strategy is to establish local quasi-independence on average -- see \cite[Equation (2.6)]{durham}  which, in itself, is a consequence of \cite[Propositions~1-3]{mem}. A key ingredient, that is  potentially of independent interest, is to first show that the decay property \eqref{decaylac} holds locally in the following sense.  Let  $\mu$ be a probability  measure supported on a subset $F$ of
$ \ \I \, $ and suppose there is a constant $A > 1 $ so that  \eqref{decaylac} is satisfied. Then for any ball $B \subset \I$ with $\mu(B) > 0 $
$$
\widehat{\mu}_B(t)=  \frac{1}{\mu(B)} \, O\left( (\log |t|)^{-A}     \right)  \, ,
$$
where $\mu_B$ is the normalised restriction of $\mu$ to $B$.
   \end{rem}

 \medskip

 %
%
%
%
%
%
%
%
%

\section{Basic estimates and establishing  Theorem \ref{mainCONV} }

In this section we present various basic estimates that will be required in proving our main results.  Indeed, the estimates provided will be enough to deduce  the general convergence statement  Theorem \ref{mainCONV}.    Given $\gamma\in\I  $, $\psi:\mathbb{N}\rightarrow \I$ and $q \in \N$,  let
\begin{equation} \label{def_En}
E_q^{\gamma}  =  E_q^{\gamma} (\psi) :=   \left\{ x\in \I : \| qx-\gamma \|\leq \psi(q) \right\} \, .
\end{equation}

\vspace*{1ex}

\noindent  By definition, given any increasing sequence of natural numbers $\mathcal{A}=(q_n)_{n=1}^{\infty}$,  we have that
$$
R(x,N)  = \# \big\{ 1\leq n \leq N :   x \in E_{q_n}^{\gamma} \big\} $$
where $ R(x,N) $  is the counting function given by \eqref{countdef}.  Also,
 the set $W_{\ca{A}}(\gamma;\psi)$ defined by~\eqref{wellset} is precisely the set of real numbers in $\I$  which lie in infinitely many of the sets $E_{q_n}^{\gamma}$; that is,
\begin{equation*}
W_{\ca{A}}(\gamma;\psi)\, =\, \limsup_{n\to\infty} E_{q_n}^{\gamma}.
\end{equation*}
Thus the sets $E_{q}^{\gamma}$ with $q \in \N$ can be regarded as being the  `building blocks' of the basic objects studied in this paper.    Let $\mu$ be a probability  measure supported on a subset $F$ of $ \ \I \, $.  We now proceed to estimate the $\mu$-measure of these building blocks.

\subsection{Estimating  $\mu ( E_q^{\gamma})$\label{s_approximating_functions}}

Let $\varepsilon$ and $\delta$ be real numbers such that  $0<\varepsilon\leq 1$ and $0< \delta<1/4$.  Let  $ \chi_{\delta} : \I \to \R $  be the characteristic function defined by
$$ 
\chi_{\delta}(x):= \begin{cases}1 \ \text{ if }  \ \|x\|\leq \delta  \\[1ex]
 0 \ \text{ if } \  \|x\|>\delta \, ,  \end{cases}    $$
and let  $\chi_{\delta, \varepsilon}^+  : \I \to  \R    $ and $  \chi_{\delta, \varepsilon}^-: \I \to \R$ be the continuous upper and lower approximations of  $ \chi_{\delta} $ given by  \vspace{2mm}
$$ 
\chi_{\delta, \varepsilon}^+(x):= \begin{cases} 1 & \text{\  if   }  \  \|x\|\leq \delta, \\[1ex]
1+ \dfrac{1}{\delta\varepsilon}(\delta-\|x\|) &\text{\   if }  \  \delta < \|x\| \leq (1+\varepsilon)\delta \\[1.3ex]
0 & \text{\ if } \  \|x\|>(1+\varepsilon)\delta \, , \end{cases}  $$
and

$$ 
 \chi_{\delta, \varepsilon}^-(x): = \begin{cases} 1 & \text{\  if } \  \|x\|\leq (1 -\varepsilon)\delta  \\[1ex]
\dfrac{1}{\delta\varepsilon}(\delta-\|x\|) &\text{\   if } \  (1-\varepsilon)\delta < \|x\| \leq \delta \\[1.3ex]
0  & \text{\ if } \ \|x\|>\delta  \, . \end{cases}
$$

\noindent Clearly, both  $\chi_{\delta, \varepsilon}^+  $ and $  \chi_{\delta, \varepsilon}^- $ are periodic functions with period 1. Next, given a real positive function $\psi:\mathbb{N}\rightarrow \I$ and any integer $q\geq 4$, consider the functions  $W_{q,\gamma,\varepsilon,\psi}^{+} $ and $W_{q,\gamma,\varepsilon,\psi}^{-}$ defined by \vspace{2mm}
\begin{equation} \label{Wdef}
W_{q,\gamma,\varepsilon}^{+}(x)=W_{q,\gamma,\varepsilon,\psi}^{+}(x)\ :=\ \Big( \sum_{p=0}^{q-1}\delta_{\frac{p+\gamma}{q}} (x)  \Big) *\chi_{\frac{\psi(q)}{q},\varepsilon}^{+}(x)
\end{equation}
and
\begin{equation*} \label{W-def}
W_{q,\gamma,\varepsilon}^{-}(x) =W_{q,\gamma,\varepsilon,\psi}^{-} (x) \  :=  \ \Big(\sum_{p=0}^{q-1}\delta_{\frac{p+\gamma}{q}} (x) \Big) *\chi_{\frac{\psi(q)}{q},\varepsilon}^{-} (x)  \, , \vspace{2mm}
\end{equation*}
where as usual $    *  $   denotes convolution and $\delta_x$ denotes the Dirac delta-function at  the point $x\in\mathbb{R}$.  As alluded to in defining the functions $W_{q,\gamma,\varepsilon,\psi}^{+}$ and $W_{q,\gamma,\varepsilon,\psi}^{-}$, we will often exclude stressing their dependance  on the function $\psi$ since it will often be fixed in any given discussion or argument.  With this in mind,  it is easily verified that  \vspace{2mm}
\begin{eqnarray*}
W_{q,\gamma,\varepsilon}^{+}(x) =  \sum_{p=0}^{q-1}\chi_{\frac{\psi(q)}{q}, \varepsilon}^{+}\left(\textstyle{x- \frac{p+\gamma}{q}}\right)
\end{eqnarray*}
and
\begin{eqnarray*}
W_{q,\gamma,\varepsilon}^{-}(x) = \sum_{p=0}^{q-1}\chi_{\frac{\psi(q)}{q}, \varepsilon}^{-}\left(\textstyle{x- \frac{p+\gamma}{q}}\right) \,  .
\end{eqnarray*}
It thus  follows that for any $0<\varepsilon \le 1$ and any integer $q\geq 4$, \vspace{2mm}
\begin{equation} \label{muineq}
\int_{0}^{1}W_{q,\gamma,\varepsilon}^-(x)\mathrm{d}\mu(x)
\; \leq  \;
\mu( E_q^{\gamma}) \;  \leq \; \int_{0}^{1}W_{q,\gamma,\varepsilon}^+(x)\mathrm{d}\mu(x). \vspace{2mm}
\end{equation}

\noindent
We now proceed to evaluate the above integrals by considering the Fourier series expansions of $ W_{q,\gamma,\varepsilon}^{+}  $ and  $W_{q,\gamma,\varepsilon}^{-} $. When there is no  risk of confusion,   we will simply write $ W_{q,\gamma,\varepsilon}^{\pm}  $  to mean both the `upper' and `lower' functions. Similarly, we will write  $  \chi_{\delta,\ve}^{\pm}$  when we refer to both  $\chi_{\delta,\ve}^{+}$ and $\chi_{\delta,\ve}^{-}$.   With this in mind, for $ k \in \Z$ let $ \widehat{\chi}_{\delta,\ve}^{\pm}(k) $ and $\widehat{W}_{q,\gamma,\varepsilon}^{\pm}(k)$ denote the $k$--th Fourier coefficient of $ {\chi}_{\delta,\ve}^{\pm}(k) $ and $ {W}_{q,\gamma,\varepsilon}^{\pm}(k),$
respectively.  A straightforward calculation yields that   \vspace{2mm}
\begin{equation}  \label{slv1}
\widehat{\chi}_{\delta,\ve}^{+}(k)\, = \,
\begin{cases}
(2 + \ve)\delta  & \text{\  if   }  \  k=0 \\[2ex]
\dfrac{\cos(2\pi k \delta) - \cos(2\pi k\delta(1+\ve))}{2\pi^2k^2\delta\ve }  \, & \text{\  if   }  \   k\neq 0  \, ,
\end{cases} \,
\end{equation}

\noindent and

\begin{equation} \label{slv1a}
\widehat{\chi}_{\delta,\ve}^{-}(k)\, = \,
\begin{cases}
(2 - \ve)\delta  & \text{\  if   }  \  k=0  \\[2ex]
\dfrac{\cos(2\pi k\delta(1 -\ve))-\cos(2\pi k \delta) }{2\pi^2k^2\delta\ve }  \, & \text{\  if   }  \  k\neq 0  \, . \vspace{1.8mm}
\end{cases}  \,
\end{equation}

\noindent Since  the functions $ W_{q,\gamma,\varepsilon}^{\pm}  $  are defined via convolution, we have that
\vspace{2mm}
$$\widehat{W}_{q,\gamma,\varepsilon}^{\pm}(k) = \sum\limits_{p=0}^{q-1}\widehat{\delta}_{\frac{p+\gamma}{q}}(k)  \, . \, \widehat{\chi}_{\frac{\psi(q)}{q}, \ve}^{\pm}(k) \, .$$
Trivially,
$$
\widehat{\delta}_{\frac{p+\gamma}{q}}(k)   =  \exp\left( -\dfrac{2\pi  i k (p+\gamma)}{ q}\right)  \, .
$$

\noindent Thus, it follows from \eqref{slv1} that  for $k\neq 0$,
\vspace*{2mm}
\begin{equation} \label{fcoef}
\widehat{W}_{q,\gamma,\varepsilon}^+(k)\, =\,\begin{cases}
\exp\left( -\dfrac{2\pi  i k \gamma}{ q}\right) \dfrac{q\left(\cos(2\pi k\psi(q)q^{-1})-\cos(2\pi k\psi(q)q^{-1}(1+\varepsilon))\right)}{2\pi^2 k^2\psi(q)q^{-1}\varepsilon}   &\text{\  if   }  \ q\mid k \\[3ex]
0 & \text{\  if   }  \  q\nmid k \,  ,
\end{cases}  \,
 \end{equation}
and for $k = 0$,  \vspace*{2mm}
\begin{equation} \label{fcoef_zero}
\widehat{W}_{q,\gamma,\varepsilon}^{+}(0)=(2+\varepsilon) \,  \psi(q)\, .  \vspace{2mm}
\end{equation}

\noindent Similarly, it follows from \eqref{slv1a} that  for $k\neq 0$,
\vspace{2mm}
\begin{equation} \label{fcoef2}
\widehat{W}_{q,\gamma,\varepsilon}^-(k)\, =\,
\begin{cases}
\exp\left( -\dfrac{2\pi  i k \gamma}{ q}\right) \dfrac{q\left(\cos(2\pi k\psi(q)q^{-1}(1-\varepsilon))-\cos(2\pi k\psi(q)q^{-1}) \right)}{2\pi^2 k^2\psi(q)q^{-1}\varepsilon}  &\text{\  if   }  \ q\mid k \\[3ex]
0 & \text{\  if   }  \  q\nmid k  \, ,   \vspace{2mm}
\end{cases}
\end{equation}
and for $k=0$,
\vspace{2mm}
\begin{equation} \label{fcoef_zero2}
\widehat{W}_{q,\gamma,\varepsilon}^-(0)  = (2 - \ve)\, \psi(q)  \, . \vspace{2mm}
\end{equation}

\noindent It is easily seen that $\sum\limits_{k \in \Z} \big| \widehat{W}_{q,\gamma,\varepsilon}^{\pm}(k) \big|  < \infty $, so the Fourier series
$$ \sum_{k \in \Z }\widehat{W}_{q,\gamma,\ve}^{\pm}(k)  \exp(2\pi kix)
$$
converges uniformly to $W_{q,\gamma,\varepsilon}^{\pm}(x)$ for all $ x \in \I$.  Hence, it  follows that
\begin{eqnarray*}
\int_0^1 W_{q,\gamma,\varepsilon}^{\pm}(x)  \; \mathrm{d}\mu(x) \; =  \;  \sum_{k\in \Z} \, \widehat{W}_{q,\gamma,\varepsilon}^{\pm}(k)  \;  \widehat{\mu}(-k)   \, .
\end{eqnarray*}
This together with \eqref{muineq}, \eqref{fcoef_zero},  \eqref{fcoef_zero2} and the fact that $\widehat{\mu}(0)=1$, implies that
\vspace*{3ex}
\begin{equation} \label{mu_ie}
\begin{array}{ll}
   \mu( E_q^{\gamma})  \  \leq \ (2+\varepsilon) \, \psi(q)  \  + \displaystyle{\sum_{k \in \Z \setminus \{0 \} }}\widehat{W}_{q,\gamma,\varepsilon}^{+}(k)  \; \widehat{\mu}(-k)\;
   \\[6ex]
   \mu( E_q^{\gamma})  \ \geq  \  (2-\varepsilon) \,  \psi(q)   \ + \displaystyle{\sum_{k \in \Z \setminus \{0 \} }}  \widehat{W}_{q,\gamma,\varepsilon}^{-}(k) \; \widehat{\mu}(-k)  \, .
    \end{array} \,
\end{equation}

\hidden{
EZ: WE NEED TO DECIDE ON THE NOTATION IN THIS SECTION. IT MIGHT BE GOOD TO SWITCH TO THE SEQUENCE $(q_n)$ AT THIS POINT, TO REPLACE GENERAL DENOMINATORS $q$. THE THING IS THAT IF WE GO WITH GENERAL $q$ HERE, THEN LATER IN THE SECTION WE NEED TO (RE-)INTRODUCE THE NOTATION $E_q^{\gamma}$ (I KEEP IT WITH EXTRA SYMBOL. $(\psi)$ AT THE MOMENT). THEN, WE GET A CONFLICT NOTATION BECAUSE THE PREVIOUSLY DEFINED $E_n^{\gamma}$ IS THE SAME AS THE NEWLY DEFINED $E_{q_n}^{\gamma}$, AND OF COURSE $q_n\ne n$.
}

\vspace*{2ex}

We now proceed to estimate the Fourier coefficient  $ \widehat{W}_{q,\gamma,\varepsilon}^{\pm}(k) $ when $k \neq 0$. In view of \eqref{fcoef}  and \eqref{fcoef2}, we only need to consider the case when $k$ is a multiple of $q$.   With this in mind,  for any $ s\in\mathbb{Z}   \setminus\{0\} $ we claim that
\begin{eqnarray}
\left|\widehat{W}_{q,\gamma,\varepsilon}^{\pm}(s q)\right| &\leq& (2+\varepsilon) \, \psi(q),\label{W_ub_psi}\\[2ex]
\left|\widehat{W}_{q,\gamma,\varepsilon}^{\pm}(s q)\right| &\leq& \frac{1}{\pi^2s^2\psi(q)\varepsilon}.\label{W_ub_1s2}
\end{eqnarray}

\noindent The upper bound~\eqref{W_ub_1s2} follows from~\eqref{fcoef} and~\eqref{fcoef2} by using the trivial fact that $|\cos(x)|\leq 1$ for all $x\in\R$.  Note that \eqref{W_ub_1s2} is stronger than \eqref{W_ub_psi} for large values of $|s|$; namely when
$$
s^2   \ge  \frac{1}{\pi^2\psi(q)^2    \varepsilon  ( 2+ \varepsilon)}  \, . $$
In order to establish the upper bound~\eqref{W_ub_psi}, note that
\begin{eqnarray*}
| \cos\big(2\pi k\psi(q)q^{-1}\big)-\cos\big(2\pi k \psi(q)q^{-1}(1 + \varepsilon)\big)| & =& \left|\int_{\frac{2\pi k\psi(q)}{q}}^{\frac{2\pi k\psi(q)}{q}(1 + \varepsilon)}\!\!\sin(x) \;  {\rm d}x \right|\\[1ex]
& \leq & \int_{\frac{2\pi k\psi(q)}{q}}^{\frac{2\pi k\psi(q)}{q}(1 + \varepsilon)}  |x| \;  {\rm d}x \\[1ex]
&=& 2\pi^2 k^2 \psi(q)^2q^{-2} \, \varepsilon (2+\varepsilon)
\end{eqnarray*}
This together with~\eqref{fcoef} yields~\eqref{W_ub_psi} for the upper function $ W_{q,\gamma,\varepsilon}^{+}  $.  The proof for the lower function follows the same steps, with appropriate modifications such as using~\eqref{fcoef2} instead of~\eqref{fcoef}.

\noindent The above estimates enable us to prove the following two useful lemmas.

\begin{lem} \label{main_part_bounded}

Let $0< \varepsilon,\tilde{\varepsilon} \le 1$.  Then,  for any integers $q, r \geq 4$

\begin{eqnarray} \label{ub_single_sum}
\sum_{s \in \Z} \big|\what{W}_{q,\gamma,\varepsilon}^{\pm}(sq)\big| \  <  \ \frac{3}{\varepsilon^{1/2}}
\end{eqnarray}
and
\begin{equation} \label{ub_double_sum}
\sum_{s\in \Z} \,  \sum_{t\in \Z}  \, \big|\what{W}_{q,\gamma,\varepsilon}^{\pm}(sq)\big| \ \big|\what{W}_{r,\gamma,\tilde{\varepsilon}}^{\pm}(tr)\big|  \ \leq \  \frac{9}{\varepsilon^{1/2}\cdot\tilde{\varepsilon}^{1/2}}  \, .
\end{equation}
\end{lem}

\vspace*{2ex}

\begin{proof}  For any $N \in \N $, we trivially have that
$$
\sum_{s=-N}^N \big|\what{W}_{q,\gamma,\varepsilon}^{\pm}(sq)\big| \ \leq \hspace*{3ex}  2 \!\!\!\!\!\!\!\!\!\!  \sum_{ 0 \le s\le \frac{1}{\sqrt{2}\pi\psi(q)\varepsilon^{1/2}}}    \!\!\!\!\!\!\! \big|\what{W}_{q,\gamma,\varepsilon}^{\pm}(sq)\big| \ \ + \
\hspace*{3ex}  2 \!\!\!\!\!\!\!\!\!\!  \sum_{\frac{1}{\sqrt{2}\pi\psi(q)\varepsilon^{1/2}}<s\leq N}     \!\!\!\!\!\!\!
\big|\what{W}_{q,\gamma,\varepsilon}^{\pm}(sq)\big|
$$
On using~\eqref{fcoef_zero} and~\eqref{W_ub_psi} to estimate the first sum appearing on the right hand side of the above inequality and \eqref{W_ub_1s2}
for the second, it follows that
\begin{equation*} \label{ub_single_sum_N}
\begin{aligned}
\sum_{s=-N}^N \big|\what{W}_{q,\gamma,\varepsilon}^{\pm}(sq)\big| & \leq    \hspace*{3ex}  2 \!\!\!\!\!\!\!\!\!\!  \sum_{ 0 \le s\le \frac{1}{\sqrt{2}\pi\psi(q)\varepsilon^{1/2}}}   \!\!\!\! 3 \psi(q)  \ \ + \
\hspace*{3ex}  2 \!\!\!\!\!\!\!\!\!\!  \sum_{\frac{1}{\sqrt{2}\pi\psi(q)\varepsilon^{1/2}}<s\leq N}
\frac{1}{\pi^2 s^2\psi(q)\varepsilon}
\\[3ex]
&\leq  \  \ \frac{3 \sqrt{2}}{\pi\varepsilon^{1/2}}  \ + \ \frac{ 2 \sqrt{2}}{\pi\varepsilon^{1/2}} \ <  \  \frac{3}{\varepsilon^{1/2}}   \, .
\end{aligned}
\end{equation*}
The desired inequality \eqref{ub_single_sum} now follows on letting $N\to \infty$. The other inequality \eqref{ub_double_sum}, trivially follows from  \eqref{ub_single_sum} and the fact that
$$
\sum_{s\in \Z} \,  \sum_{t\in \Z}  \, \big|\what{W}_{q,\gamma,\varepsilon}^{\pm}(sq)\big| \ \big|\what{W}_{r,\gamma,\tilde{\varepsilon}}^{\pm}(tr)\big|
\   \leq   \
\Big(\sum_{s \in \Z} \big|\what{W}_{q,\gamma,\varepsilon}^{\pm}(sq)\big|\Big) \ \Big(\sum_{t  \in \Z} \big|\what{W}_{r,\gamma,\tilde{\varepsilon}}^{\pm}(tr)\big|\Big) \, .
$$
\end{proof}

\begin{lem} \label{lem2}
Let
 $\mu$ be a probability measure supported on a subset $F$ of  $\I$. Let  $E_{q}^{\gamma}$ be given by~\eqref{def_En}. Then, for any integer $ q \ge 4$
\begin{eqnarray}
\mu( E_{q}^{\gamma}) &\leq&  3\psi(q) + 3\max_{s\in\N}|\widehat{\mu}(sq)| \label{lem2_result1}\\[2ex]
\mu( E_{q}^{\gamma}) &\leq&  3\psi(q) + 2\sum_{s=1}^{\infty}\frac{|\widehat{\mu}(sq)|}{s}  \; .  \label{lem2_result2}
\end{eqnarray}
\end{lem}
\begin{proof}
It follows from  \eqref{mu_ie} with $\varepsilon= 1$ and \eqref{fcoef},  that
\begin{equation} \label{lem2_ie2}
\mu( E_{q}^{\gamma})   \ \le  \  3\psi(q)     + \sum_{s \in \Z \setminus \{0 \}  }\widehat{W}_{q,\gamma,1}^{+}(sq)   \  \widehat{\mu}(-sq).
\end{equation}
The desired estimate~\eqref{lem2_result1} now immediately follows from this, together with the fact that by  Lemma~\ref{main_part_bounded},
\begin{equation*} \label{lem2_ie3}
\Big|\sum_{s \in \Z \setminus \{0 \}  }\widehat{W}_{q,\gamma,1}^{+}(sq)   \  \widehat{\mu}(-sq)\Big|   \  \stackrel{\eqref{ub_single_sum}}{\leq}  \  3 \, \max_{s\in\N}|\widehat{\mu}(sq)|  \, .
\end{equation*}

\noindent To prove~\eqref{lem2_result2}, note that the geometric mean of the inequalities~\eqref{W_ub_psi} and~\eqref{W_ub_1s2} implies that
$$
\big|\widehat{W}_{q,\gamma,1/2}^{\pm}(s q)\big|  \ \leq \  \frac{\sqrt{3}}{\pi |s|}    \ \leq \   \frac{1}{|s|} \, .
$$
This together with~\eqref{lem2_ie2} implies~\eqref{lem2_result2} .
\end{proof}

\subsection{Proof of Theorem \ref{mainCONV}}

Armed with  Lemma~\ref{lem2},  it is easily seen that Theorem~\ref{mainCONV} is a straightforward  consequence of the classical Borel-Cantelli Lemma \cite[Chapter 1]{harman} from probability theory.

\vspace{2mm}
\begin{lem}[Borel-Cantelli]\label{lem_BC} Let $(X,\ca{B},\mu)$ be a probabilty space and $(A_n)_{n=1}^{\infty}\subseteq\ca{B}$ be a sequence of subsets of~$X$. If
\begin{equation*} \label{lem_BC_convcondition}
\sum\limits_{n=1}^{\infty}\mu(A_n)<\infty
\end{equation*}
then
$$ \mu\left( \limsup_{n\to\infty}A_n\right)=0 . $$
\end{lem}

\vspace*{3ex}

To establish  Theorem~\ref{mainCONV}, we first  observe  that Lemma~\ref{lem2} together with the hypotheses of the theorem guarantees that $$
\sum\limits_{n=1}^{\infty}\mu(E_{q_n}^{\gamma})<\infty  \, .
$$
Hence, the Borel-Cantelli Lemma with  $A_n:=E_{q_n}^{\gamma}$ implies that  $$  \mu\big( \limsup_{n\to\infty}E_{q_n}^{\gamma}\big)=0 \, .   $$
This completes the proof since by definition   $
W_{\ca{A}}(\gamma;\psi)\, =\, \limsup\limits_{n\to\infty} E_{q_n}^{\gamma}  $.


We now move onto proving the counting results: Theorem~\ref{asmplacthm} and Theorem~\ref{mainSV}.  As already mentioned, Theorem \ref{main_q}  is deduced from Theorem \ref{mainSV}  in  \S\ref{deducing}.

\section{Establishing the counting results modulo `independence'}


\medskip

To begin with, we observe  that Theorem~\ref{mainCONV} (which we have already proved) implies both Theorems~\ref{asmplacthm} $\&$ \ref{mainSV} if  $\Psi (N)  $ is bounded;  that is,  if
$$
 \textstyle{\sum_{n=1}^{\infty}}  \psi(q_n)  < \infty \, . $$
Indeed, it is easily verified that the  hypotheses on the measure $\mu$ and the sequence $\cA$ within the statements of Theorems~\ref{asmplacthm} $\&$~\ref{mainSV}  guarantees the convergence condition \eqref{cond2_2},
and so Theorem~\ref{mainCONV} implies that for $\mu$-almost all $x \in F$
$$
R(x,N) =  R(x,N;\gamma,\psi, \ca{A})   <  \infty   \ \ {\rm as \  \ }  N \to \infty \, .
$$
This is consistent with the conclusions of the counting theorems; namely the statements  associated with \eqref{countlacresult} and \eqref{countFSPresultsv}.  Thus, during the course of proving Theorems~\ref{asmplacthm} $\&$ \ref{mainSV}  we can assume that $\Psi (N)  $ is unbounded.

 \medskip

\begin{fact}  \label{fact1}  Without loss of generality, we can assume that
$$
\textstyle{\sum_{n=1}^{\infty}}  \psi(q_n)  = \infty   \quad { \ or  \ equivalently   \ } \quad  \Psi(N) \to \infty \ \ {\rm as \  \ }  N \to \infty \, .    $$

\end{fact}

\medskip

Before describing the main mechanism for establishing the desired counting results,  we mention three more useful facts that we can assume during the course of their proofs.


\medskip

\begin{fact} \label{rem_psi_big}   Without loss of generality, we can assume that for any given   $\tau > 1$
\begin{equation}
\label{psi_is_not_small}
\psi(q_n)\geq 3 \, n^{-\tau}      \quad \forall   \  \ n \in \N \, . \end{equation}
\end{fact}

It is easily seen that this follows on showing that there is no loss of generality in assuming that
\begin{equation} \label{psi_is_not_smallA}
\psi(q_n)\geq   \omega (q_n)   \quad \forall  \ \ n \in \N \,  ,
\end{equation}
where $ \omega : \N \to \I$  is any real, positive function such that $\textstyle{\sum_{n=1}^{\infty}}  \omega(q_n)  < \infty $.
With this in mind,  consider the auxiliary function
$$
\psi^* : q_n  \to \psi^*(q_n):=\max(\psi(q_n),\omega(q_n))  \,  .
$$
Trivially,   the function  $ \psi^* $   satisfies  \eqref{psi_is_not_smallA}  and by the definition of the counting function (see~\eqref{countdef}),  we have that
\begin{equation*} \label{ie_RRR}
R(x,N;\gamma,\psi, \ca{A})\leq R(x,N;\gamma,\psi^*, \ca{A}) \leq R(x,N;\gamma,\psi, \ca{A})+R(x,N;\gamma,\omega, \ca{A})  \, .
\end{equation*}
For $\mu$-almost all $x \in F$,  Theorem~\ref{mainCONV} implies that    $R(x,N;\gamma,\omega, \ca{A}) $  remains bounded as $N \to \infty$ and so it follows that
\begin{equation*} \label{rem_psi_big_lhs}
R(x,N;\gamma,\psi^*, \ca{A}) = R(x,N;\gamma,\psi, \ca{A})+O(1)  \, .
\end{equation*}
Now in view of Fact \ref{fact1}, we can assume that the sum $\textstyle{\sum_{n=1}^{\infty}}  \psi(q_n) $ diverges.   Hence $\textstyle{\sum_{n=1}^{\infty}}  \psi^*(q_n) $ diverges and so the desired statements associated with \eqref{countlacresult} and \eqref{countFSPresultsv} for the function $\psi$ are equivalent to the analogous statements for the modified function $\psi^*$.  In short, within the context of the right hand sides of~\eqref{countlacresult} and~\eqref{countFSPresultsv}, for $\mu$-almost all $x \in F$ the additional contribution from the counting function associated with $\omega$ is negligible.  Hence, without loss of generality  we can assume \eqref{psi_is_not_smallA}  and thus  \eqref{psi_is_not_small}.

\medskip

\begin{fact} \label{q_n_is_big}
 Without loss of generality, we can assume that for any given increasing sequence  $\ca{A}= (q_n)_{n\in \N} $   of natural numbers,   $q_1 >4 $.
\end{fact}

To see this, simply observe that there are only a finite number of terms $q_n \in \ca{A}$ with $q_n \leq 4$.   Thus removing these `small' terms from $\ca{A}$ and working with the resulting sequence introduces at most an additional $ O(1)  $  term on  the right hand sides of~\eqref{countlacresult} and~\eqref{countFSPresultsv}. However, this  is negligible since by Fact \ref{fact1}  we are assuming that  $\Psi(N)\rightarrow\infty$ as $N\rightarrow\infty$.

\medskip

\begin{fact} \label{q_n_is_bigSV}
 Without loss of generality, we can assume that for any given  $\alpha$-separated increasing  sequence  $\ca{A}= (q_n)_{n\in \N} $   of natural numbers,   the associated implicit  constant $m_0=1 $.  In other words,   we can assume that for any integers $ 1 \leq m<n$, if
$
1 \leq |sq_m-t q_n|< q_m^{\alpha}
$
for some $s,t\in\N$,   then
$
s>m^{12}.
$
\end{fact}

To see this, if  $m_0 \geq 2$ we simply remove the first $m_0-1$ terms of $ \ca{A}$ and observe that the resulting sequence $ (q_{n+(m_0-1)})_{n\in \N} $ has the desired properties.  We have already seen in the justification of Fact~\ref{q_n_is_big}, that removing a finite number of terms of $\ca{A}$ is negligible within the context of establishing~\eqref{countlacresult} and~\eqref{countFSPresultsv}.

\subsection{A mechanism for establishing counting results}

The following  statement~\cite[Lemma~1.5]{harman} represents an important  tool in the theory of metric Diophantine approximation for establishing counting statements  (along the lines of Theorems~\ref{asmplacthm}~$\&$~\ref{mainSV}).  It has its bases in the familiar variance method of probability theory and can be viewed as the quantitative form of the (diveregnce)  Borel-Cantelli Lemma \cite[Lemma~2.2]{durham}.

\begin{lem} \label{ebc}
Let $(X,\ca{B},\mu)$ be a probability space, let $(f_n(x))_{n \in \N}$ be a sequence of non-negative $\mu$-measurable functions defined on $X$, and $(f_n)_{n \in \N },\ (\phi_n)_{n  \in \N}$ be sequences of real numbers  such that
$$ 0\leq f_n \leq \phi_n \hspace{7mm} (n=1,2,\ldots).  $$

\noindent 
Suppose that for arbitrary  $a,b \in \N$ with $a <  b$, we have
\begin{equation} \label{ebc_condition1}
\int_{X} \left(\sum_{n=a}^{b} \big( f_n(x) -  f_n \big) \right)^2\mathrm{d}\mu(x)\, \leq\,  C\!\sum_{n=a}^{b}\phi_n
\end{equation}

\noindent for an absolute constant $C>0$. Then, for any given $\varepsilon>0$,  we have
\begin{equation} \label{ebc_conclusion}
\sum_{n=1}^N f_n(x)\, =\, \sum_{n=1}^{N}f_n\, +\, O\left(\Phi(N)^{1/2}\log^{\frac{3}{2}+\varepsilon}\Phi(N)+\max_{1\leq k\leq N}f_k\right)
\end{equation}
\noindent for $\mu$-almost all $x\in X$, where $
\Phi(N):= \sum\limits_{n=1}^{N}\phi_n
$. 
\end{lem}

\begin{rem}
Note that in statistical terms, $f_n$ is the mean of $f_n(x)$ and \eqref{ebc_condition1} deals with the variance.
\end{rem}

\medskip

Given  a real number $\gamma\in\I  $,  a real, positive function $\psi:\mathbb{N}\rightarrow \I$  and an increasing sequence of natural numbers $\ca{A}= (q_n)_{n\in \N} $, we consider Lemma \ref{ebc}  with
\begin{equation} \label{Harman_choice_parameters}
X := \I \, , \qquad f_n(x):= \chi_{_{E_{q_n}^{\gamma}}}\!(x)  \qquad {\rm and } \qquad
 f_n=2\psi(q_n) \, ,
\end{equation}
where  $  \chi_{E_{q_n}^{\gamma}} \!\!$   is the characteristic function  of the set  $E_{q_n}^{\gamma}$  given by~\eqref{def_En}.   Then, clearly for any $ x \in \I$ and $N \in \N$  we have that the
$$
{\rm l.h.s. \ of \ } \eqref{ebc_conclusion}  \ =   \ R(x, N) \, ,
$$
where $R(x,N)$ is the counting function given by \eqref{countdef}.  Also,  the main term on the ${\rm r.h.s. \ of \ }$ \eqref{ebc_conclusion} is precisely $2\Psi(N)$ where  $\Psi(N) $ is the partial sum given by \eqref{def_Psi}.    Furthermore, it is easily verified  that for any $a,b \in \N$ with $a <  b$

\begin{equation*}
\begin{aligned}
\left(\sum_{n=a}^{b}(f_n(x)-f_n) \right)^2  &=  \ \left(\sum_{n=a}^{b}f_n(x)\right)^2 \, + \, \left(\sum_{n=a}^{b}f_n\right)^2 -  \ 2\sum_{n=a}^{b}f_n(x) \cdot \sum_{n=a}^{b}f_n \\[2ex]
   =  \  \sum_{n=a}^{b}f_n(x)  \  &+  \ 2\mathop{\sum\sum}_{a\leq m < n\leq b}f_m(x)f_n(x) \  +  \  \left(\sum_{n=a}^{b}f_n\right)^2  \ -  \ 2\sum_{n=a}^{b}f_n\cdot\sum_{n=a}^{b}f_n(x) \,  ,  \vspace{2mm}
\end{aligned}
\end{equation*}

\vspace{2mm}

\noindent and so it follows that
\vspace{2mm}
\begin{equation} \label{integral_square_estimate}
\begin{aligned}
  {\rm l.h.s. \ of \ } \eqref{ebc_condition1} 
  \ =  \   \sum_{n=a}^{b}\mu(E_{q_n}^{\gamma})   \, + \, 2 &\mathop{\sum\sum}_{a\leq m<n\leq b}  \mu(E_{q_m}^{\gamma}\cap E_{q_n}^{\gamma})  \\[2ex] \, &- \ 4 \sum_{n=a}^{b} \psi(q_n) \left( \sum_{n=a}^{b}\mu(E_{q_n}^{\gamma})  - \sum_{n=a}^{b} \psi(q_n) \right)
\end{aligned}
\end{equation}

\vspace{2mm}

\noindent where $\mu$ is a non-atomic probability  measure supported on a subset of
$ \ \I \, $.
The upshot of this is that in view of Lemma \ref{ebc},  the proof of  Theorem~\ref{asmplacthm} and Theorem~\ref{mainSV} boils down to  `appropriately' estimating  the right hand side  of \eqref{integral_square_estimate}. Estimating the measure of the intersection of the sets  $E_{q_n}^{\gamma} $ is where the main difficulty lies. In short we need to show that these `building block' sets are independent on average with an acceptable error term. Suppose for the moment that there was  no error term; in other words
\begin{equation} \label{heaven}
2 \mathop{\sum\sum}_{a\leq m<n\leq b} \mu(E_{q_m}^{\gamma}\cap E_{q_n}^{\gamma}) \ \leq  \  \left( \sum_{n=a}^{b}   \mu(E_{q_n}^{\gamma})  \right)^2  .
\vspace{2mm}
\end{equation}
Then,
and as we shall see in a moment, it is easy to deduce  the desired counting statements \eqref{countlacresult} and \eqref{countFSPresultsv} from Lemma~\ref{ebc} once  we have established the following  statement.


\begin{lem} \label{lem_sum_mu_good}
Let $\mu$ be a probability  measure supported on a subset $F$ of
$ \ \I \, $.  Let $\ca{A}= (q_n)_{n\in \N} $ be an increasing sequence of natural numbers that satisfies the growth condition  \eqref{hyp_q} for some constants  $B \ge  1$  and $C > 0$.  Furthermore, assume that  $q_1 > 4 $.
Let $\gamma\in\I  $ and $\psi:\mathbb{N}\rightarrow \I$ be a real, positive function that satisfies~\eqref{psi_is_not_small}.  Suppose there exists a constant $
A  \, >  \,  2  B $ so that \eqref{decaylac} is satisfied. Then, for arbitrary  $a,b \in \N$ with $a <  b$, we have
\begin{equation} \label{lem_sum_mu_good_result}
\sum_{n=a}^b\mu(E_{q_n}^{\gamma}) = 2\sum_{n=a}^b\psi(q_n) \, + \, O\left( \min\Big(1,\sum_{n=a}^b\psi(q_n)\Big) \right).
\end{equation}
\end{lem}

 \bigskip

 \noindent Note that in view of Facts~1-3 at the start of this section, we can assume the hypotheses   of Theorems~\ref{asmplacthm} $\&$ \ref{mainSV} satisfy those of the lemma. Hence,  if we genuinely  had  independence on average (i.e., \eqref{heaven} holds), then Lemma~\ref{lem_sum_mu_good}  would imply that
 \vspace{2mm}
\begin{eqnarray*} \label{integral_square_estimategh}
{\rm r.h.s. \ of \ } \eqref{integral_square_estimate}
   & \leq  & \   \sum_{n=a}^{b}\mu(E_{q_n}^{\gamma})   \, + \, \left(\sum_{n=a}^{b}\mu(E_{q_n}^{\gamma})  \right)^2  - \ 4 \sum_{n=a}^{b} \psi(q_n) \left( \sum_{n=a}^{b}\mu(E_{q_n}^{\gamma})  - \sum_{n=a}^{b} \psi(q_n) \right)  \\[2ex]
   & = &  O\left( \sum_{n=a}^b\psi(q_n) \right)  \, .
\end{eqnarray*}

\vspace{2mm}

\noindent Thus, we conclude that
 $${\rm l.h.s. \ of \ } \eqref{ebc_condition1} \, \ll  \,    2 \,  \sum_{n=a}^{b}\psi(q_n)  \, := \,  \sum_{n=a}^{b}  f_n$$
 and  \eqref{countlacresult} and \eqref{countFSPresultsv} follow on applying Lemma~\ref{ebc} with $\phi_n := f_n$. In fact, we would be able to improve the associated  error terms.
 However, we do not have \eqref{heaven} and establishing the following estimates is at the heart of proving  the desired counting statements.

 Throughout the paper,  we use the notation $\log^+(x):=\max(0,\log(x))$.

\begin{prop}  \label{indiethm1}  Let  $F$, $\mu$, $\ca{A}= (q_n)_{n\in \N} $, $\gamma $ and $\psi$ be as in  Theorem~\ref{asmplacthm}.
Furthermore, assume $\psi$ satisfies \eqref{psi_is_not_small} and that  $q_1 > 4$. Then, for arbitrary  $a,b \in \N$ with $a <  b$, we have
\begin{multline} \label{delta_set_frac23}
 2 \mathop{\sum\sum}_{a\leq m<n\leq b} \mu(E_{q_m}^{\gamma}\cap E_{q_n}^{\gamma}) \ \leq   \  \left( \sum_{n=a}^{b}   \mu(E_{q_n}^{\gamma})  \right)^2 \\[2ex]
+
O\left(\left(\sum_{n=a}^{b}\psi(q_n)\right)^{4/3}\left(\log^+\left(\sum_{n=a}^{b}\psi(q_n)\right)+1\right)+\sum_{n=a}^{b}\psi(q_n)\right).
\end{multline}
\end{prop}

\vspace*{3ex}

\begin{prop} \label{indiethm4} Let  $F$, $\mu$, $\ca{A}= (q_n)_{n\in \N} $, $\gamma $ and $\psi$ be as in  Theorem~\ref{mainSV}.
Furthermore, assume $\psi$ satisfies \eqref{psi_is_not_small}$, q_1 > 4$ and that  $\ca{A}$ is $\alpha$-separated with the implicit constant $m_0=1$.   Then, for arbitrary  $a,b \in \N$ with $a <  b$, we have
\begin{multline}\label{overlapfinal_finite_set_of_primesSV}
2 \mathop{\sum\sum}_{a\leq m<n\leq b} \mu(E_{q_m}^{\gamma}\cap E_{q_n}^{\gamma}) \ \leq \  \left( \sum_{n=a}^{b}   \mu(E_{q_n}^{\gamma})  \right)^2 + O\left(\mathop{\sum\sum}_{a\leq m<n\leq b} (q_m,q_n)\min\left(\frac{\psi(q_m)}{q_m},\frac{\psi(q_n)}{q_n}\right)\right) \\[2ex]
+ O\left(\left(\sum_{n=a}^{b}\psi(q_n)\right)\log^+\left(\sum_{n=a}^{b}\psi(q_n)\right)+\sum_{n=a}^{b}\psi(q_n)\right).
\end{multline}
\end{prop}

\bigskip

\begin{rem}\label{remedy}
We stress that within the context of establishing Theorems~\ref{asmplacthm} $\&$ \ref{mainSV}, there is no harm in assuming the additional hypotheses imposed in the statements of Propositions~\ref{indiethm1}~$\&$~\ref{indiethm4}.  The justification for this is provided by Facts~1-4 at the start of this section.
 \end{rem}

\medskip

The  proofs of the above  propositions are technically a little  involved and will be the subject of \S\ref{secpresv} and \S\ref{svcountproof}.  In the remaining part of this section we shall first prove Lemma~\ref{lem_sum_mu_good} and then go onto  completing  the proofs of Theorems \ref{asmplacthm} $\&$ \ref{mainSV} modulo the above   propositions.

\medskip


\subsection{Proof of Lemma \ref{lem_sum_mu_good}}
For any given sequence  of real numbers $(\varepsilon_n)_{n=a}^b$ in $(0,1]$, it follows via~\eqref{mu_ie} that
\begin{equation} \label{lem_sum_mu_good_ie1}
\left|\sum_{n=a}^b\mu(E_{q_n})-2\sum_{n=a}^b\psi(q_n)\right|\leq\sum_{n=a}^b\psi(q_n)\varepsilon_n+
\max_{\circ\in\{+,-\}}\sum_{n=a}^b\left|\sum_{k\in \Z \setminus \{0\}}\widehat{W}_{q_n,\gamma,\varepsilon_n}^{\circ}(k)\widehat{\mu}(-k)\right|.
\end{equation}

\noindent Now, in view of \eqref{fcoef} and  \eqref{fcoef2}, we have that
$\widehat{W}_{q_n,\gamma,\varepsilon_n}^{\pm}(k)=0$
unless $k = s q_n$ for some integer $s$. Also, on using \eqref{decaylac} and  \eqref{hyp_q}
it  can be verified  that $ |\hat{\mu}(sq_n)| \ll n^{-A/B} $ for any non-zero integer~$s$ and $ n \in \N$.  Hence  by Lemma~\ref{main_part_bounded},  it follows that
\begin{equation} \label{lem_sum_mu_good_ie2sv}
\begin{aligned}
\left|\sum_{n=a}^b\mu(E_{q_n})-2\sum_{n=a}^b\psi(q_n)\right|& \ \leq \ \sum_{n=a}^b\psi(q_n)\varepsilon_n  \, +  \,
\sum_{n=a}^b\frac{3}{n^{A/B}\varepsilon_n^{1/2}} \, .
\end{aligned}
\end{equation}

\noindent In turn, this together with~\eqref{psi_is_not_small} and the fact that $A>2B$
implies that
\begin{equation*} 
\left|\sum_{n=a}^b\mu(E_{q_n})-2\sum_{n=a}^b\psi(q_n)\right|\leq\sum_{n=a}^b\psi(q_n)\varepsilon_n+\sum_{n=a}^b\frac{\psi(q_n)}{n^{A/2B}\varepsilon_n^{1/2}}.
\end{equation*}

\noindent On letting  $\varepsilon_n=1$ for all $a \leq n \leq b$,  we   obtain the upper bound
\begin{equation} \label{lem_sum_mu_good_result_half1}
\left|\sum_{n=a}^b\mu(E_{q_n})-2\sum_{n=a}^b\psi(q_n)\right|\ll\sum_{n=a}^b\psi(q_n) \, .
\end{equation}

\bigskip

\noindent To complete the proof of the lemma, it remains to establish the `other' upper bound; that is

\begin{equation} \label{lem_sum_mu_good_result_half2}
\left|\sum_{n=a}^b\mu(E_{q_n})-2\sum_{n=a}^b\psi(q_n)\right|\ll 1.
\end{equation}

\noindent With this in mind,  for any $n\in\N$ let
$$
\Psi(n):=\sum_{k=1}^n\psi(q_k)  \quad {\rm and } \quad \varepsilon_n:=\min\left(1,\Psi(n)^{-2}\right).
$$
By definition,  $|\psi(q_n)|\leq 1$  and so
$
\varepsilon_n^{-1}\leq n^2 \, .
$
Hence, with reference to the second term on the r.h.s. of \eqref{lem_sum_mu_good_ie2sv} we have that
\begin{equation} \label{lem_sum_mu_good_second_sum_bounded}
\sum_{n=a}^b\frac{3}{n^{A/B}\varepsilon_n^{1/2}} \, \leq  \,  3\cdot\sum_{n=1}^\infty\frac{1}{n^{\frac{A}{B}-1}} \, <  \, \infty  \, .
\end{equation}
The last inequality makes use of the fact that $A > 2B$.
We now turn out attention to the first term on the r.h.s. of \eqref{lem_sum_mu_good_ie2sv}. A straight forward application of
Lemma~\ref{cor_sum_S_square} in  Appendix~D of~\S\ref{app} with $s_n=\psi(q_n)$ and $\gamma=1$, yields that
\begin{equation} \label{lem_sum_mu_good_ie3}
\sum_{n=a}^b\psi(q_n)\varepsilon_n \, < \, \sum_{n=1}^{\infty}\frac{\psi(q_n)}{\max\left(1,\Psi(n)^2\right)} \, < \, 3 \, .
\end{equation}
Combining the inequalities~\eqref{lem_sum_mu_good_ie2sv}, \eqref{lem_sum_mu_good_second_sum_bounded} and~\eqref{lem_sum_mu_good_ie3} gives the
desired upper bound  \eqref{lem_sum_mu_good_result_half2}.


\subsection{Proof of Theorem \ref{asmplacthm} modulo Proposition \ref{indiethm1}} \label{subsection_deduction_thm1}

 On using  \eqref{lem_sum_mu_good_result} and \eqref{delta_set_frac23} on the right hand side of~\eqref{integral_square_estimate}, we find that for any $a,b\in\N$  with $a< b$
\begin{equation} \label{integral_square_lacunary_O}
 {\rm l.h.s. \ of \ } \eqref{ebc_condition1} \, =  \, O\left(\left(\sum_{n=a}^{b}\psi(q_n)\right)^{4/3}
 \left(\log^+\left(\sum_{n=a}^{b}\psi(q_n)\right)+1\right)+
 \sum_{n=a}^{b}\psi(q_n)\right)  \, .
\end{equation}
We now estimate the above term on the right and show that
\begin{equation} \label{sum_phi_is_big}
\left(\sum_{n=a}^{b}\psi(q_n)\right)^{4/3}
\left(\log^+\left(\sum_{n=a}^{b}\psi(q_n)\right)+1\right)+
\sum_{n=a}^{b}\psi(q_n) \ \leq \  6\sum_{n=a}^b\phi_n  \, ,
\end{equation}

\noindent where
\begin{equation} \label{def_varphi}
\phi_n : = \psi(q_n)\Psi(n)^{1/3}\left(\log^+\Psi(n)+1\right)+2\psi(q_n)  \, .
\end{equation}

\noindent To this end,
denote by $m \in \N $ the smallest integer satisfying  $a\leq m\leq b$ such that
\begin{equation} \label{ie1_c}
\sum_{n=a}^{m}\psi(q_n)  \, \geq   \, \frac{1}{2}\sum_{n=a}^{b}\psi(q_n).
\end{equation}
Note that by the definition of $m$, we have that
\begin{equation} \label{ie2_c}
\sum_{n=m}^{b}\psi(q_n)  \, \geq  \, \frac{1}{2}\sum_{n=a}^{b}\psi(q_n)
\end{equation}
and that for any integer $n$ such that $m\leq n\leq b$
\begin{equation} \label{thm1_Psi_is_large}
2\Psi(n) \, \geq  \,  \sum_{k=a}^b\psi(q_k).
\end{equation}
Then, by using~\eqref{ie2_c} and then~\eqref{thm1_Psi_is_large}, it is easily verified that
$$
\begin{aligned}
\left(\sum_{n=a}^{b}\psi(q_n)\right)^{4/3}&\left(\log^+\left(\sum_{n=a}^{b}\psi(q_n)\right)+1\right)+\sum_{n=a}^{b}\psi(q_n)
\\[2ex]
&\leq \
2\sum_{n=m}^{b}\left(\psi(q_n)\left(\sum_{k=a}^{b}\psi(q_k)\right)^{1/3}\left(\log^+\left(\sum_{k=a}^{b}\psi(q_k)\right)+1\right)\right)+\sum_{n=a}^{b}\psi(q_n)
\\[2ex]
&\leq \
2\sum_{n=m}^{b}\left(\psi(q_n)\left(2\Psi(n)\right)^{1/3}\left(\log^+\left(2\Psi(n)\right)+1\right)\right)+\sum_{n=a}^{b}\psi(q_n)
\\[2ex]
&\leq  \
4\cdot 2^{1/3}\cdot\sum_{n=m}^{b}\left(\psi(q_n)\Psi(n)^{1/3}\left(\log^+\left(\Psi(n)\right)+1\right)\right)+\sum_{n=a}^{b}\psi(q_n)
\\[2ex]
&\leq \
6\sum_{n=a}^{b}\left(\psi(q_n)\Psi(n)^{1/3}\left(\log^+\left(\Psi(n)\right)+1\right)\right)+\sum_{n=a}^{b}\psi(q_n) \, .
\end{aligned}
$$
This establishes~\eqref{sum_phi_is_big} which together with \eqref{integral_square_lacunary_O} implies that condition
\eqref{ebc_condition1} of Lemma~\ref{ebc} is satisfied with $X, f_n(x)$ and
 $ f_n$ given by~\eqref{Harman_choice_parameters} and $\phi_n$ by  \eqref{def_varphi}.   Also note that for any $n \in \N$, we trivially have that $f_n \leq \phi_n$, $f_n \leq 2 $ and
 \begin{eqnarray*}
\Phi(N):=\sum_{n=1}^N\phi_n  &\leq&  \sum_{n=1}^N\psi(q_n)\Psi(N)^{1/3}\left(\log^+\Psi(N)+1\right)+\sum_{n=1}^N\psi(q_n) \\[2ex] & =  & \Psi(N)^{4/3}\left(\log^+\Psi(N)+1\right)+\Psi(N).
\end{eqnarray*}
Hence, Lemma~\ref{ebc} implies that for any $\varepsilon>0$
\begin{eqnarray*}
R(x,N) &=& {\rm l.h.s. \ of \ } \eqref{ebc_conclusion}  \\[2ex]
&=&2\Psi(N)+O\left(\Psi(N)^{2/3}\left(\log\Psi(N)+2\right)^{2+\varepsilon}\right),
\end{eqnarray*}
and this completes the proof of Theorem~\ref{asmplacthm} assuming the truth of Proposition~\ref{indiethm1}.

\subsection{Proof of Theorem \ref{mainSV} modulo Proposition \ref{indiethm4}}

The proof is similar to that  in the previous subsection.  On using  \eqref{lem_sum_mu_good_result} and~\eqref{overlapfinal_finite_set_of_primesSV} on the right hand side of~\eqref{integral_square_estimate},  we find that for any $a,b\in\N$  with $a< b$
\begin{equation}
\begin{aligned}\label{integral_square_main_O}
 {\rm l.h.s. \ of \ } \eqref{ebc_condition1} \, =  \, O\Bigg( \ \Big(\sum_{n=a}^b\psi(q_n)\Big) &\Big(\log^+\Big(\sum_{n=a}^b\psi(q_n)\Big)+1\Big)
\\[1ex]
&+  \ \mathop{\sum\sum}_{a\leq m<n\leq b} (q_m,q_n)\min\left(\frac{\psi(q_m)}{q_m},\frac{\psi(q_n)}{q_n}\right)  \ \Bigg)
\end{aligned}
\end{equation}
We estimate the above term on the right and show that
\begin{equation}
\begin{aligned}\label{integral_square_main_Osv}
 \left(\sum_{n=a}^b\psi(q_n)\right)& \left(\log^+\left(\sum_{n=a}^b\psi(q_n)\right)
 +1\right)
\\[1ex]
&+ \ \mathop{\sum\sum}_{a\leq m<n\leq b} (q_m,q_n)\min\left(\frac{\psi(q_m)}{q_m},\frac{\psi(q_n)}{q_n}\right) \ \leq \  2 \sum_{n=a}^b\phi_n  \, ,
\end{aligned}
\end{equation}
\noindent where
\begin{equation} \label{def_varphi_main}
\phi_n:= \psi(q_n)\left(\log^+\Psi(n)+2\right)
+\sum_{m=1}^{n-1}(q_m,q_n)\min\left(\frac{\psi(q_m)}{q_m},\frac{\psi(q_n)}{q_n}\right) \, .
\end{equation}

\noindent Clearly, this will immediately follow on showing that
\begin{equation} \label{ub_thm4_technical}
\left(\sum_{n=a}^b\psi(q_n)\right)\left(\log^+\left(\sum_{n=a}^b\psi(q_n)\right)
+1\right)  \ \leq \ 2 \sum_{n=a}^b\psi(q_n)\left(\log^+\Psi(n)+2\right).
\end{equation}
\noindent  As in the previous subsection, let $m  \in \N $ be the smallest integer satisfying  $a\leq m\leq b$ such that \eqref{ie1_c} holds.  Then, by using~\eqref{ie2_c} and~\eqref{thm1_Psi_is_large}, it is easily verified that
$$
\begin{aligned}
\left(\sum_{n=a}^b\psi(q_n)\right)\left(\log^+\left(\sum_{n=a}^b\psi(q_n)\right)+1\right)
&\leq  \
2\left(\sum_{n=m}^b\psi(q_n)\right)\left(\log^+\left(\sum_{n=a}^b\psi(q_n)\right)+1\right)
\\
&\leq \
2\left(\sum_{n=m}^b\psi(q_n)\left(\log^+\left(2\Psi(n)\right)+1\right)\right)
\\
&\leq \
2\left(\sum_{n=m}^b\psi(q_n)\left(\log^+\Psi(n)+2\right)\right)
\\
&\leq  \
2\left(\sum_{n=a}^b\psi(q_n)\left(\log^+\Psi(n)+2\right)\right)  \, .
\end{aligned}
$$
This establishes~\eqref{ub_thm4_technical} and so \eqref{integral_square_main_Osv} follows.  The upshot is that \eqref{integral_square_main_Osv} together with \eqref{integral_square_main_O} implies that condition
\eqref{ebc_condition1} of Lemma~\ref{ebc} is satisfied with $X, f_n(x)$ and
 $ f_n$ given by~\eqref{Harman_choice_parameters} and $\phi_n$ by  \eqref{def_varphi_main}.   Also note that for any $n \in \N$, we trivially have that $f_n \leq \phi_n$, $f_n \leq 2 $ and
 \begin{equation*}
\Phi(N):=\sum_{n=1}^N\phi_n  \, \leq  \,  \Psi(N)\left(\log^+\Psi(N)+2\right) \, + \, E(N) \, ,
\end{equation*}
where $ E(N) $ is given by \eqref{error2sv}.
 The desired counting statement \eqref{countFSPresultsv} now follows on applying  Lemma~\ref{ebc}. Hence, this  completes the proof of Theorem~\ref{mainSV} assuming the truth of Proposition~\ref{indiethm4}.


\section{Preliminaries for `independence'  \label{secpresv} }
In this section we set out the ground work  for establishing the desired estimates for the measure of the intersection of the sets  $E_{q_n}^{\gamma} $; namely Propositions~\ref{indiethm1} \& \ref{indiethm4}.  Recall, that these estimates are at the heart of proving our main counting results; namely Theorems \ref{asmplacthm} \& \ref{mainSV}.

  Throughout this section  $(\varepsilon_n)_{n\in \N}$  will be a fixed sequence of real numbers in $(0,1]$.  Also, with reference to \eqref{Wdef}, for any $m, n \in \N$  with $ m < n$ let
\begin{equation} \label{wmndef}
W_{m,n}^{+}:= W_{q_m,\gamma,\varepsilon_m}^+W_{q_n,\gamma,\varepsilon_n}^+\, .
\end{equation}
\noindent Then, by definition
\begin{eqnarray*}
\mu(E_{q_m}^{\gamma}\cap E_{q_n}^{\gamma}) &\leq& \int_{0}^{1} W_{q_m,\gamma,\varepsilon_m}^+(x) W_{q_n,\gamma,\varepsilon_n}^+(x) \, \mathrm{d}\mu(x)   \nonumber
 \\ [2ex]
&=& \int_{0}^{1}W_{m,n}^{+}(x)\mathrm{d}\mu(x)
\end{eqnarray*}

\noindent Our aim is to obtain a sufficiently strong upper bound for the above integral. This we do by considering the Fourier series expansion of the function $ W_{m,n}^{+}$.  It is easily verified that for any $k \in \Z$,
\begin{eqnarray}  \label{greek}
\what{W}_{m,n}^{+}(k) & := &   \, \int_{0}^{1}
W_{q_m,\gamma,\varepsilon_m}^{+}(x)W_{q_n,\gamma,\varepsilon_n}^{+}(x) \, \exp(-2\pi kix) \mathrm{d}x    \nonumber \\[2ex]
 & = &  \sum_{j \in \Z} \what{W}_{q_m,\gamma,\varepsilon_m}^{+}(k) \
\what{W}_{q_n,\gamma,\varepsilon_n}^{+}(k-j)  \, .
\end{eqnarray}
Moreover, it is easily seen  that $\sum\limits_{k \in \Z} \big| \widehat{W}_{m,n}^{+}(k) \big|  < \infty $, so the Fourier series
$$ \sum_{k \in \Z }\widehat{W}_{m,n}^{+}(k)  \exp(2\pi kix)
$$
converges uniformly to $W_{m,n}^{+}(x)$ for all $ x \in \I$.
Hence, it  follows that
\begin{eqnarray*}
\int_0^1 W_{m,n}^{+}(x)(x)  \; \mathrm{d}\mu(x) \; =  \;  \sum_{k\in \Z} \, \widehat{W}_{m,n}^{+}(k)  \;  \widehat{\mu}(-k)   \, .
\end{eqnarray*}

\noindent    The upshot of this is that

\begin{eqnarray}
\mu(E_{q_m}^{\gamma}\cap E_{q_n}^{\gamma}) &\leq&
 \sum_{k \in \Z }\what{W}_{m,n}^{+}(k)\what{\mu}(-k) \nonumber
\\ [2ex]
&=& \what{W}_{m,n}^{+}(0) + \sum_{k\in \Z \setminus \{0\}} \what{W}_{m,n}^{+}(k)\what{\mu}(-k).  \label{mit1}  \vspace{2mm}
\end{eqnarray}

\noindent
To proceed, we consider the two terms on the right hand side  of \eqref{mit1} separately. By definition, for any pair of natural numbers $m, n$ we have that
\begin{equation} \label{ub_W0_lacunary}
\begin{aligned}
\what{W}_{m,n}^{+}(0) &:=   \, \int_{0}^{1}W_{q_m,\gamma,\varepsilon_m}^{+}(x)W_{q_n,\gamma,\varepsilon_n}^{+}(x)\mathrm{d}x \\[2ex]
&\leq  \, |(1+\varepsilon_m)E_{q_m}^{\gamma}\cap (1+\varepsilon_n)E_{q_n}^{\gamma}|\, ,  \vspace{2mm}
\end{aligned}
\end{equation}
where $ | \, . \, | $ is Lebesgue measure and  with reference
to \eqref{def_En}, for any constant $\kappa>0$ we let $\kappa E_{q}^{\gamma}:=  E_{q}^{\gamma}(\kappa \psi) $.
%
%
%
It is relatively straightforward to verify (see for example~\cite[Equation~3.2.5]{harman}\footnote{Equation~3.2.5 in \cite{harman} as stated is not correct   -- the `big O' error term is missing.} for the details) that for   any  $q, q' \in \N$
\begin{equation*}
|E_{q}^{\gamma}\cap E_{q'}^{\gamma}| \, =  \, 4\psi(q)\psi(q')+ O\left( (q,q') \ \min\Big(\frac{\psi(q)}{q},\frac{\psi(q')}{q'}\Big)  \right) \, .
\end{equation*}


\noindent Hence, it follows that
\begin{multline*} \label{main_q_Harman_ub}
\Big|(1+\varepsilon_m) E_{q_m}^{\gamma}\cap (1+\varepsilon_n) E_{q_n}^{\gamma} \Big| \,  = \, 4(1+\varepsilon_m)(1+\varepsilon_n)\psi(q_m)\psi(q_n)    \\[1ex]   + \ \  O\left( (q_m,q_n)\min\Big(\frac{\psi(q_m)}{q_m},\frac{\psi(q_n)}{q_n}\Big)  \right) \, .
\end{multline*}

\noindent  This together with~\eqref{ub_W0_lacunary}  implies that
\begin{equation}  \label{main_q_Harman_ub2}
 \what{W}_{m,n}^{+}(0)  \, \leq  \,  4(1+\varepsilon_m)(1+\varepsilon_n)\psi(q_m)\psi(q_n)  +  O\left( (q_m,q_n)\min\Big(\frac{\psi(q_m)}{q_m},\frac{\psi(q_n)}{q_n}\Big)  \right) \, .
\end{equation}

\noindent We now turn our attention to the  second term appearing on the right hand side of \eqref{mit1} which for convenience  we will  denote by $S_{m,n}$.  Note that in view of \eqref{fcoef} and \eqref{greek}, it follows that
\begin{eqnarray} \label{def_S_m_n}
S_{m,n} &:= & \sum_{k \in\mathbb{Z}\setminus \{0\} } \what{W}_{m,n}^{+}(k)\what{\mu}(-k) \nonumber \\[2ex]
&= & \mathop{ \mathop{\sum\sum}_{ s,t\in\mathbb{Z} } }_{ sq_m-tq_n \neq 0}
\what{W}_{q_m,\gamma,\varepsilon_m}^{+}(sq_m)
\what{W}_{q_n,\gamma,\varepsilon_n}^{+}(tq_n)
\what{\mu}\left(-(sq_m+tq_n)\right).
\end{eqnarray}

\subsection{Estimates for $S_{m,n}$} \label{s_Smn}

We start by providing a general upper bound estimate  for $S_{m,n}$ which is applicable to integer sequences  associated with both  Propositions~\ref{indiethm1} \& \ref{indiethm4}.

\begin{lem} \label{prop_sum_S_m_n}
Let $\mu$ be a probability  measure supported on a subset $F$ of
$ \ \I \, $.  Let $\ca{A}= (q_n)_{n\in \N} $ be an increasing sequence of natural numbers that satisfies the growth condition  \eqref{hyp_q} for some constants  $B \ge  1$  and $C > 0$.  Furthermore, assume that  $q_1 > 4 $.
Let $(\varepsilon_n)_{n\in \N}$  be a sequence of real numbers in $(0,1]$, $\alpha \in (0,1)$,  $\gamma\in\I  $ and $\psi:\mathbb{N}\rightarrow \I$ be a real, positive function.  Suppose there exists a constant $
A  \, >  \,  2  B $ so that \eqref{decaylac} is satisfied. Then, for any  $m,n \in \N$ with $m <  n$, we have
$$ |S_{m,n} |\ll   \   \frac{\psi(q_m)}{n^{\frac{A}{B}}\varepsilon_n^{1/2}}\, +\, \left( 1+ \frac{1}{\alpha^{A}} \right)
\frac{\psi(q_n)}{m^{\frac{A}{B}}\varepsilon_m^{1/2}}
\, + \, \frac{1}{n^{\frac{A}{B}}\varepsilon_m^{1/2}\varepsilon_n^{1/2}}   \,  +  \ |T(m,n)|, $$
where
\begin{equation} \label{def_S_prime}
 T(m,n) \  :=  \! \mathop{ \mathop{\sum\sum}_{ s,t\in\mathbb{Z}\setminus \{0\} } }_{1\leq |sq_m-tq_n|<q_m^{\alpha}} \!\!\! \what{W}_{q_m,\gamma,\varepsilon_m}^{+}(sq_m)\what{W}_{q_n,\gamma,\varepsilon_n}^{+}(tq_n)  \what{\mu}\left( sq_m-tq_n\right).
\end{equation}
\end{lem}

The next two lemmas provide estimates on the average size of the quantity $T(m,n)$  appearing in Lemma \ref{prop_sum_S_m_n}.  The first deals with  lacunary sequences (i.e. the context of Proposition~\ref{indiethm1}) and the second deals with  $\alpha$-separated sequences  (i.e.  the context of Proposition~\ref{indiethm4}).

\begin{lem}\label{prop_sum_S_m_n_lacunary}

Let $\mu$ be a probability  measure supported on a subset $F$ of
$ \ \I \, $.  Let $\ca{A}= (q_n)_{n\in \N} $ be a lacunary sequence of natural numbers.
Let $(\varepsilon_n)_{n\in \N}$  be a decreasing sequence of real numbers in $(0,1]$, $\alpha \in (0,1)$,  $\gamma\in\I  $ and $\psi:\mathbb{N}\rightarrow \I$ be a real, positive function.
 Then, for arbitrary   $a,b \in \N$ with $a <  b$, we have
$$ \mathop{\sum\sum}_{a\leq m < n\leq b} |T(m,n)|  \ \ll  \  \sum_{n=a}^b\frac{\psi(q_n)}{\varepsilon_n^{1/2}}  \, , $$
where $T(m,n)$ is defined by~\eqref{def_S_prime}.
\end{lem}

\begin{lem} \label{prop_sum_S_m_n_separation}
Let $\mu$ be a probability  measure supported on a subset $F$ of
$ \ \I \, $.  Let $\ca{A}= (q_n)_{n\in \N} $ be an $\alpha$-separated increasing sequence of natural numbers with the implicit  constant $m_0 = 1$.
Let $(\varepsilon_n)_{n\in \N}$  be a sequence of real numbers in $(0,1]$, $\gamma\in\I  $ and $\psi:\mathbb{N}\rightarrow \I$ be a real, positive function.  Suppose that for any $n \in \N$
\begin{equation} \label{psi_lb}
\psi(q_n) \geq n^{-9} \,  \end{equation}
and
\begin{equation} \label{psi_lbsv} \varepsilon_n^{-1} \leq 2n \, .
\end{equation}
 Then, for arbitrary   $a,b \in \N$ with $a <  b$, we have
\begin{equation} \label{Tmnprop5}
\mathop{\sum\sum}_{a\leq m < n\leq b} |T(m,n)| \ll \sum_{n=a}^b\psi(q_n),
\end{equation}
where $T(m,n)$ is defined by~\eqref{def_S_prime}.
\end{lem}

The rest of this section will be devoted to proving the  above three lemmas.

\begin{proof}[Proof of Lemma~\ref{prop_sum_S_m_n}]
We start by decomposing $S_{m,n }$ into three sums:

$$ S_{m,n}= S_1(m,n)+S_2(m,n)+S_3(m,n), \vspace{2mm} $$
where \vspace{2mm}
\begin{eqnarray*}
S_1(m,n) &:=&  \sum_{t\in \Z \setminus \{0\}}   \what{W}_{q_m,\gamma,\varepsilon_m}^+(0) \what{W}_{q_n,\gamma,\varepsilon_n}^{+}(tq_n)\what{\mu}(-tq_n), \\[2ex]
S_2(m,n) &:=&  \sum_{s\in \Z \setminus \{0\}}   \what{W}_{q_n,\gamma,\varepsilon_n}^+(0) \what{W}_{q_m,\gamma,\varepsilon_m}^{+}(sq_m)\what{\mu}(-sq_m), \\[2ex]
S_3(m,n) &:=& \mathop{ \mathop{\sum\sum}_{ s,t\in\mathbb{Z}\setminus \{0\} } }_{sq_m+tq_n\neq 0}\what{W}_{q_m,\gamma,\varepsilon_m}^{+}(sq_m)\what{W}_{q_n,\gamma,\varepsilon_n}^{+}(tq_n)\what{\mu}\left(-(sq_m+tq_n)\right).\vspace{2mm}
\end{eqnarray*}

\noindent
In order to find an upper bound for $S_1(m,n)$,  first note that by making use of \eqref{decaylac} and~\eqref{hyp_q} it follows that
$$
|\what{\mu}(-tq_n)|\ll \left(\log|t q_n|\right)^{-A}\ll n^{-\frac{A}{B}} \, .
$$
This together with~\eqref{fcoef_zero} and \eqref{ub_single_sum}
implies that
\begin{equation*}
|S_1(m,n)| \ \ll \  \frac{(2+\varepsilon_m)\psi(q_m)}{n^{\frac{A}{B}}}\sum_{t\in \Z }  \what{W}_{q_n,\gamma,\varepsilon_n}^{+}(tq_n)  \ \ll \
\frac{\psi(q_m)}{n^{\frac{A}{B}}\varepsilon_n^{1/2}}\, .
\end{equation*}

\noindent Similarly, we find that

\begin{equation*}
|S_2(m,n)| \ll \frac{\psi(q_n)}{m^{\frac{A}{B}}\varepsilon_m^{1/2}}\, \cdot
\end{equation*}

\noindent To deal with $S_3(m,n)$, we decompose further into two sums:
$$   S_3(m,n) \, =\, S_4(m,n) + S_5(m,n), $$
where
$$ S_4(m,n) \ := \mathop{ \mathop{\sum\sum}_{ s,t\in\mathbb{Z}\setminus \{0\} } }_{|sq_m-tq_n|\geq q_n/2} \! \what{W}_{q_m,\gamma,\varepsilon_m}^{+}(sq_m)\what{W}_{q_n,\gamma,\varepsilon_n}^{+}(tq_n)   \what{\mu}\left( sq_m-tq_n\right)  $$
and
\begin{equation} \label{def_S_7}
S_5(m,n) \ := \mathop{ \mathop{\sum\sum}_{ s,t\in\mathbb{Z}\setminus \{0\} } }_{1\leq |sq_m-tq_n|<q_n/2}   \!\!\!\! \what{W}_{q_m,\gamma,\varepsilon_m}^{+}(sq_m)\what{W}_{q_n,\gamma,\varepsilon_n}^{+}(tq_n)   \what{\mu}\left( sq_m-tq_n\right) .
\end{equation}

\noindent Regarding $S_4(m,n)$, by making use of \eqref{decaylac}, \eqref{hyp_q} and the restriction $|sq_m-tq_n|\geq q_n/2$ imposed on  $s,t\in\mathbb{Z}\setminus \{0\}$,  it follows that
$$   |\what{\mu}\left( sq_m-tq_n\right)| \ll n^{-\frac{A}{B}}. $$
This together with \eqref{ub_double_sum} implies that
$$ |S_4(m,n)| \ll \frac{1}{n^{\frac{A}{B}}\varepsilon_m^{1/2}\varepsilon_n^{1/2}} \cdot $$
To deal with $S_5(m,n)$, we decompose further into two sums:
$$ S_5(m,n) \, =  \,  S_6(m,n) \, + \, T(m,n) \, , $$
where
$$ S_6(m,n)\, :=\hspace{-5mm} \mathop{\mathop{\sum\sum}_{ s,t\in\mathbb{Z}\setminus \{0\} } }_{q_m^{\alpha}\leq |sq_m-tq_n|<q_n/2}\hspace{-3.8mm}\what{W}_{q_m,\gamma,\varepsilon_m}^{+}(sq_m)\what{W}_{q_n,\gamma,\varepsilon_n}^{+}(tq_n)\what{\mu}\left( sq_m-tq_n\right) \vspace{3mm}$$
 and $T(m,n)$ is defined by~\eqref{def_S_prime}.   Regarding $S_6(m,n)$, we first make use of \eqref{decaylac}, \eqref{hyp_q} and the lower bound  restriction $|sq_m-tq_n|\geq q_m^{\alpha}$ imposed on $  s,t\in\mathbb{Z} \setminus \{0\}  $, to find that
$$   |\what{\mu}\left( sq_m-tq_n\right) |\; \ll  \;  \alpha^{-A} \, m^{-\frac{A}{B}}. $$
Next,  we observe that the upper bound restriction $|sq_m-tq_n|\leq q_n/2$ imposed on the non-zero integers $s,t $ is equivalent to
\begin{equation} \label{index_restriction_first_bound}
\left|s\frac{q_m}{q_n}-t\right| < \frac12 \, .
\end{equation}
It is now easy to see that if  $s$ and $t$ satisfy \eqref{index_restriction_first_bound} then both necessarily  must have the same sign and also that for each fixed  integer $s$ there exists at most one non-zero integer $t=t_s$ satisfying~\eqref{index_restriction_first_bound}. Thus,
\begin{eqnarray*}
|S_6(m,n)|  & \ll &  \frac{1}{\alpha^A \, m^{A/B}} \,
\mathop{\mathop{\sum}_{ s\in\N \, : } }_{t_s {\rm \, exists}} |\what{W}_{q_m,\gamma,\varepsilon_m}^{+}(sq_m)  | \
 | \what{W}_{q_n,\gamma,\varepsilon_n}^{+}(t_sq_n) | \vspace{3mm}
 \end{eqnarray*}
 and on using~\eqref{W_ub_psi} to bound $|\what{W}_{q_n,\gamma,\varepsilon_n}^{+}(t_sq_n)|$ and  \eqref{ub_single_sum} to bound  $ |\what{W}_{q_m,\gamma,\varepsilon_n}^{+}(sq_m)| $, we find  that
$$ |S_6(m,n)|  \, \ll \, \frac{1}{\alpha^A} \, \frac{\psi(q_n)}{m^{\frac{A}{B}}\varepsilon_m^{1/2}} \, . \vspace{3mm}$$
The above upper bounds  for the absolute values of  $S_1(m,n), S_2(m,n),  S_4(m,n) $ and $S_6(m,n) $ together with the fact that
$$
|S_{m,n}|    \ \leq  \  |S_1(m,n)| \, + \, |S_2(m,n)| \, + \,  |S_4(m,n)| \, + \,     |S_6(m,n) |\, + \, |T(m,n)|  \, ,
$$
completes the proof of the proposition.
\end{proof}

\vspace*{1ex}

\begin{proof}[Proof of Lemma~\ref{prop_sum_S_m_n_lacunary}] The strategy is similar to that used above to estimate  $S_6(m,n)$.   To start with, observe that the restriction $|sq_m-tq_n|\leq q_m^{\alpha}$ imposed on the non-zero integers $s,t $ associated with $T(m,n)$  implies that
\begin{equation} \label{index_restriction_first_boundsv}
\left|s-t\frac{q_n}{q_m}\right| < 1 \, .
\end{equation}
Hence,  if  $s$ and $t$ satisfy \eqref{index_restriction_first_boundsv} then both necessarily  must have the same sign and also for each fixed integer $t$ there exists a set $S_t$ of at most two non-zero integers  $s$  satisfying \eqref{index_restriction_first_boundsv}.
Thus, on using the trivial bound  $|\what{\mu}(t)| \leq 1 $, it follows that
$$
|T(m,n)|   \ \leq  \ \mathop{\mathop{\sum}_{ t\in\N \, : } }_{s \in S_t} |\what{W}_{q_m,\gamma,\varepsilon_m}^{+}(s q_m) | \
| \what{W}_{q_n,\gamma,\varepsilon_n}^{+}(t q_n) |  \, . \vspace{1mm}
$$
Also note that \eqref{index_restriction_first_boundsv} and the fact that $t q_n/q_m > 1$ implies
\begin{equation*}
\frac12 \frac{q_m }{ q_n}  s  \leq  \frac{q_m}{q_n} \max \{1,  (s-1)\}  \, \leq \, t\, \leq \frac{q_m}{q_n} (s+1)\leq 2\frac{q_m}{q_n} s.
\end{equation*}
On using this together with \eqref{W_ub_psi} to bound $\what{W}_{q_n,\gamma,\varepsilon_n}^{+}(t q_n)$ and both \eqref{W_ub_psi} and \eqref{W_ub_1s2} to bound  $ \what{W}_{q_m,\gamma,\varepsilon_n}^{+}(s_tq_m) $,  we find  that for any integers $ 1 \leq m < n $
\begin{eqnarray*}
|T(m,n) | &\ll  & \sum_{t  \in \N }
  \ \min\left(\frac{q_m^2}{q_n^2}\cdot\frac{1}{t^2\psi(q_m)\varepsilon_m},\psi(q_m)\right)\psi(q_n)\\[2ex]
&\ll  &
\sum_{ 1 \, \leq \, t  \, \leq \frac{q_m}{q_n\psi(q_m)\varepsilon_m^{1/2}}} \psi(q_m)\psi(q_n)   \  \ +
\sum_{t > \frac{q_m}{q_n\psi(q_m)\varepsilon_m^{1/2}}}\frac{q_m^2}{q_n^2}\cdot\frac{1}{t^2\psi(q_m)\varepsilon_m}\psi(q_n)\\[3ex]
&\ll&
\frac{q_m}{q_n}\frac{\psi(q_n)}{\varepsilon_m^{1/2}}   \  \leq \ \frac{q_m}{q_n}\frac{\psi(q_n)}{\varepsilon_n^{1/2}} \, . \vspace{2mm}
\end{eqnarray*}

\noindent The last inequality makes use of the fact that $(\varepsilon_n)_{n\in \N}$  is a decreasing sequence of real numbers.
Now the fact that $(q_n)_{n\in\N}$ is  lacunary implies that $$ \sum\limits_{1\leq m<n}q_m/q_n \ll 1 ,$$
and so it follows that  for arbitrary   $a,b \in \N$ with $a <  b$, we have
$$ \mathop{\sum\sum}_{a\leq m < n\leq b} T(m,n)\  \ll \ \sum_{n=a}^b\sum_{m=a}^{n-1}\frac{q_m}{q_n}\frac{\psi(q_n)}{\varepsilon_n^{1/2}} \ \ll \  \sum_{n=a}^{b}\frac{\psi(q_n)}{\varepsilon_n^{1/2}}\, \cdot $$
\end{proof}

\begin{proof}[Proof of Lemma~\ref{prop_sum_S_m_n_separation}]

  To start with, observe that the restriction $|sq_m-tq_n|\leq q_m^{\alpha}$ imposed on the non-zero integers $s,t $ associated with $T(m,n)$  implies that
\begin{equation} \label{index_restriction_first_boundsv3}
\left|s\frac{q_m}{q_n}-t\right| < 1 \, .
\end{equation}
Hence,  if  $s$ and $t$ satisfy \eqref{index_restriction_first_boundsv} then both necessarily  must have the same sign and also for each fixed integer $s$ there exists a set $T_s$ of at most two non-zero integers  $t$  satisfying  \eqref{index_restriction_first_boundsv3}.     Thus,
we can  decompose  $T(m,n)$ into two sums:
$$ T(m,n) = T_1(m,n) + T_2(m,n), $$
where
$$ T_1(m,n):= \mathop{\mathop{ \mathop{\sum\sum}_{s, t\in\mathbb{N} } }_{1\leq s \leq m^3/\psi(q_m)}}_{1\leq |sq_m-tq_n|<q_m^{\alpha}}\hspace{-3mm}  \what{W}_{q_m,\gamma,\varepsilon_m}^{+}(sq_m)\what{W}_{q_n,\gamma,\varepsilon_n}^{+}(tq_n) \what{\mu}\left( sq_m-tq_n\right)
$$
and

$$ T_2(m,n):= \mathop{\mathop{ \mathop{\sum\sum}_{s, t\in\mathbb{N} } }_{s> m^3/\psi(q_m)}}_{1\leq |sq_m-tq_n|<q_m^{\alpha}}\hspace{-3mm}   \what{W}_{q_m,\gamma,\varepsilon_m}^{+}(sq_m)\what{W}_{q_n,\gamma,\varepsilon_n}^{+}(tq_n)  \what{\mu}\left( sq_m-tq_n\right) \, . \vspace{2mm}
$$
In view of~\eqref{psi_lb}  the condition   $s\leq m^3/\psi(q_m)$ in the definition of $T_1(m,n)$ implies that  $s \leq m^{12}$. In turn, this together with the fact that $(q_n)_{n\in\N}$ is $\alpha$-separated with the implicit  constant $m_0 = 1$ implies  that $T_1(m,n)$ is  an empty sum. Thus,
$$
T_1(m,n) = 0 \, .
$$

\noindent Regarding $T_2(m,n)$,  on using the trivial bound  $|\what{\mu}(t)|  \leq 1 $  together with  \eqref{W_ub_psi} to bound $|\what{W}_{q_n,\gamma,\varepsilon_n}^{+}(t q_n)|$ and \eqref{W_ub_1s2} to bound  $ |\what{W}_{q_m,\gamma,\varepsilon_n}^{+}(sq_m)| $, we obtain that for any integers $ 1 \leq m < n $
\begin{eqnarray*}
|T(m,n)| \  =  \  |T_2(m,n)|  & \ll  &   \mathop{\mathop{\sum}_{ s> m^3/\psi(q_m) :  } }_{t \in T_s}
|\what{W}_{q_m,\gamma,\varepsilon_m}^{+}(sq_m)| \ | \what{W}_{q_n,\gamma,\varepsilon_n}^{+}(tq_n)|
 \\[2ex]
& \ll  & \sum_{s> m^3/\psi(q_m)}\frac{1}{s^2\psi(q_m)  \, \varepsilon_m }\psi(q_n)  \ \ll  \  \frac{\psi(q_n)}{m^3 \, \varepsilon_m } \, .
\end{eqnarray*}
Now on making use of \eqref{psi_lbsv}, it follows that
$$
|T_2(m,n )| \ \ll \ \frac{\psi(q_n)}{m^2}
$$
and so in turn, for arbitrary   $a,b \in \N$ with $a <  b$, we have
$$ \mathop{\sum\sum}_{a\leq m<n\leq b} |T_2(m,n )| \, \ll\, \sum_{n=a}^{b}\psi(q_n). $$
\end{proof}

\section{Establishing Propositions~\ref{indiethm1} and~\ref{indiethm4} \label{svcountproof}}

To start with we work under the hypotheses of Lemma \ref{prop_sum_S_m_n}  which clearly both Proposition~\ref{indiethm1} and Proposition~\ref{indiethm4} satisfy.  With this in mind, for arbitrary  $a,b \in \N$ with $a <  b$, we have via \eqref{mit1} that
\begin{equation}\label{sumequation}
\mathop{\sum\sum}\limits_{a\leq m<n\leq b}\mu(E_{q_m}^{\gamma}\cap E_{q_n}^{\gamma}) \ \leq  \   \mathop{\sum\sum}\limits_{a\leq m<n\leq b}\what{W}_{m,n}^{+}(0) +  \mathop{\sum\sum}\limits_{a\leq m<n\leq b}S_{m,n}
\end{equation}
where $ W_{m,n}^{+}$ and $S_{m,n}$ are given by \eqref{wmndef} and \eqref{def_S_m_n} respectively.  Now let
\begin{equation} \label{def_epsilon_n_main}
\varepsilon_n:=\min\left(2^{-\delta},\, \left( \sum_{k=a}^{n}\psi(q_k) \right)^{-\delta}\right),
\end{equation}
where $0<\delta\leq 1$ is a parameter to be determined later. By definition, it follows that
\begin{equation} \label{epsilon_inverse_bound}
\varepsilon_n^{-1} \, \leq \, \max\left(2^{\delta},n^{\delta}\right) \, < \, 2n
\end{equation}
and so \eqref{psi_lbsv} associated with Lemma \ref{prop_sum_S_m_n_separation} is satisfied. We also observe that~\eqref{ie_Delta_2B} together with~\eqref{epsilon_inverse_bound} implies that
\begin{equation*} \label{epsilon_sum_inverse_bound}
\sum_{n=1}^{\infty}\frac{1}{n^{\frac{A}{B}}\varepsilon_n^{1/2}}  \, < \, \infty  \, .
\end{equation*}
\noindent Thus, for any  fixed $\alpha>0$, it follows on using  Lemma~\ref{prop_sum_S_m_n} that
\begin{equation} \label{sumofsmn}
\mathop{\sum\sum}_{a\leq m<n\leq b}|S_{m,n}| \, \ll \, \sum_{n=a}^{b}\psi(q_n) + \mathop{\sum\sum}_{a\leq m<n\leq b}\frac{1}{n^{\frac{A}{B}}\varepsilon_m^{1/2}\varepsilon_n^{1/2}}  +  \mathop{\sum\sum}_{a\leq m<n\leq b}|T(m,n)| \, ,
\end{equation}
\noindent where the implied constant is dependent on $\alpha$. We now estimate the second sum on the right hand side of \eqref{sumofsmn} by considering two cases.

\noindent Case 1: $\sum\limits_{k=a}^{b}\psi(q_k)>2$.  It follows that
$$ \frac{1}{n^{\frac{A}{B}}\varepsilon_m^{1/2}\varepsilon_n^{1/2}}\leq\frac{1}{n^{\frac{A}{B}}}\left(\sum_{k=a}^{b}\psi(q_k)\right)^{\delta},
$$
and so
\begin{equation} \label{sum_const_leq_linear_case1}
\mathop{\sum\sum}\limits_{a\leq m<n\leq b} \frac{1}{n^{\frac{A}{B}}\varepsilon_m^{1/2}\varepsilon_n^{1/2}} \leq  \mathop{\sum\sum}\limits_{a\leq m<n\leq b} \frac{1}{n^{\frac{A}{B}}}\left(\sum_{k=a}^{b}\psi(q_k)\right)^{\delta}  \ll \left(\sum_{k=a}^{b}\psi(q_n)\right)^{\delta}\leq \sum_{k=a}^{b}\psi(q_n).
\end{equation}

\medskip

\noindent Case 2: $\sum\limits_{k=a}^{b}\psi(q_k)\leq 2$. It follows that $\varepsilon_n=2^{-\delta}$ for all $a\leq n\leq b$,\vspace{-2mm}
hence
$$
\frac{1}{n^{\frac{A}{B}}\varepsilon_m^{1/2}\varepsilon_n^{1/2}}\ll\frac{1}{n^{\frac{A}{B}}}.
$$
On using~\eqref{psi_is_not_small} with $\tau:=A/2B$, it follows that
$$ \frac{1}{n^{\frac{A}{B}}}\leq\frac{\psi(q_n)}{n^{\frac{A}{2B}}}   \qquad \forall \ n \in \N $$
 and so
\begin{equation} \label{sum_const_leq_linear_case2}
\mathop{\sum\sum}\limits_{a\leq m<n\leq b} \frac{1}{n^{\frac{A}{B}}\varepsilon_m^{1/2}\varepsilon_n^{1/2}}\ll \mathop{\sum\sum}\limits_{a\leq m<n\leq b} {\psi(q_n)}{n^{-\frac{A}{2B}}} \leq \mathop{\sum\sum}\limits_{a\leq m<n\leq b} {\psi(q_n)}{m^{-\frac{A}{2B}}}  \ll \sum_{n=a}^b\psi(q_n).
\end{equation}

\medskip

\noindent The upshot of~\eqref{sum_const_leq_linear_case1} and~\eqref{sum_const_leq_linear_case2} is that both cases give rise to the same estimate which  together with \eqref{sumofsmn} gives
$$
\mathop{\sum\sum}_{a\leq m<n\leq b}|S_{m,n}| \, \ll \, \sum_{n=a}^{b}\psi(q_n)  +  \mathop{\sum\sum}_{a\leq m<n\leq b}|T(m,n)| \, .
$$
This, in turn with \eqref{sumequation} yields the estimate

\begin{equation} \label{sum_const_leq_linear}
\mathop{\sum\sum}\limits_{a\leq m<n\leq b}\mu(E_m^{\gamma}\cap E_n^{\gamma}) \leq  \mathop{\sum\sum}\limits_{a\leq m<n\leq b}\what{W}_{m,n}^{+}(0) +  \mathop{\sum\sum}\limits_{a\leq m<n\leq b}|T(m,n)| + O\left(\sum_{n=a}^{b}\psi(q_n) \right) .
\vspace{3ex}
\end{equation}

We now turn our attention to estimating the first  term on the right hand of \eqref{sum_const_leq_linear}. For this we will make use of   \eqref{main_q_Harman_ub2}.  With this in mind, first note that  by definition $(\varepsilon_n)_{n\in\N}$ is decreasing and so $\varepsilon_n\psi(q_m)\psi(q_n)\leq\varepsilon_m\psi(q_m)\psi(q_n)$ for all $m<n$. Moreover, $\varepsilon_n<1$ for all $n$ and so $\varepsilon_m\varepsilon_n\psi(q_m)\psi(q_n)<\varepsilon_m\psi(q_m)\psi(q_n)$. Thus
\begin{equation} \label{sum_eps_psipsi_developed}
4(1+\varepsilon_m)(1+\varepsilon_n)\psi(q_m)\psi(q_n) \, \leq \, 4\psi(q_m)\psi(q_n) + 12 \varepsilon_m \psi(q_m) \psi(q_n) \, .
\end{equation}

\noindent We estimate the second sum on the right hand side of the above by considering two cases.

\noindent Case 1:  $\sum\limits_{k=a}^b\psi(q_k)<2$.
It follows that
\begin{equation} \label{ub_sum_eps_psipsi_simple}
\mathop{\sum\sum}\limits_{a\leq m<n\leq b} \varepsilon_m\psi(q_m)\psi(q_n)\leq  \mathop{\sum\sum}\limits_{a\leq m<n\leq b}\psi(q_m)\psi(q_n)<\left(\sum_{n=a}^b\psi(q_n)\right)^2<2\sum_{n=a}^b\psi(q_n).
\end{equation}

\medskip

\noindent Case 2: $\sum\limits_{k=a}^b\psi(q_k)\geq 2 $. It follows that
\begin{equation} \label{ub_sum_eps_psipsi}
\begin{aligned}
\mathop{\sum\sum}\limits_{a\leq m<n\leq b} &\varepsilon_m\psi(q_m)\psi(q_n)= \mathop{\sum\sum}\limits_{a\leq m<n\leq b} \psi(q_m)\psi(q_n)\ \min\left(2^{-\delta}, \left(  \sum_{k=a}^{m}\psi(q_k)  \right)^{-\delta} \right) \\[2ex]
&\leq \
\max\left(2^{1-\delta},\left(\sum_{n=a}^b\psi(q_n) \right)^{1-\delta}\right)\mathop{\sum\sum}\limits_{a\leq m<n\leq b}\frac{\psi(q_m)\psi(q_n)}{\max\left(2, \sum_{k=a}^{m}\psi(q_k) \right)} \\[2ex]
&\leq \
\left(\sum_{n=a}^b\psi(q_n) \right)^{1-\delta}\sum_{a\leq n\leq b}\psi(q_n)\sum_{a\leq m\leq b} \frac{\psi(q_m)}{\max\left(2, \sum_{k=a}^{m}\psi(q_k) \right)}  \, .
\end{aligned}
\end{equation}
On using  Lemma~\ref{cor_sum_S_log} in Appendix~D~of~\S\ref{app} with $\gamma=2$, $s_k:=\psi(q_{k-a+1})$ and $a$ and $b$ replaced by $1$ and $b-a+1$ respectively,  we infer that
$$
\sum_{a\leq m\leq b} \frac{\psi(q_m)}{\max\left(2, \sum_{k=a}^{m}\psi(q_k) \right)} \ \leq \  \frac32 + \frac{1}{2\log\frac32}\log \left(\sum_{n=a}^b\psi(q_n) \right).
$$
This together with~\eqref{ub_sum_eps_psipsi} implies that
\begin{equation} \label{ub_sum_eps_psipsi_2}
\begin{aligned}
\mathop{\sum\sum}\limits_{a\leq m<n\leq b} \varepsilon_m\psi(q_m)\psi(q_n) \ &\leq \  \left(\sum_{n=a}^b\psi(q_n) \right)^{1-\delta}\left(\sum_{n=a}^b\psi(q_n) \right)\left(\frac32 + \frac{1}{2\log\frac32}\log \left(\sum_{n=a}^b\psi(q_n) \right)\right)\\[2ex]
\ &\ll \ \left( \sum_{n=a}^b\psi(q_n)\right)^{2-\delta}\log \left( \sum_{n=a}^b\psi(q_n) \right).
\end{aligned}
\vspace{2ex}
\end{equation}
Hence, on combining the estimates~\eqref{main_q_Harman_ub2},~\eqref{sum_eps_psipsi_developed},~\eqref{ub_sum_eps_psipsi_simple} and \eqref{ub_sum_eps_psipsi_2} with find that
\begin{multline} \label{ub_sum_eps_psipsi_combined}
\mathop{\sum\sum}\limits_{a\leq m<n\leq b}  W_{m,n}^{+}(0) \leq  4\left( \sum_{n=a}^{b}\psi(q_n)\right)^2 +  O\left(\left(\sum_{n=a}^b\psi(q_n) \right)^{2-\delta}\log^+ \left(\sum_{n=a}^b\psi(q_n) \right)\right)\\[2ex] + \  O\left( \mathop{\sum\sum}_{a\leq m<n\leq b}(q_m,q_n)\min\Big(\frac{\psi(q_m)}{q_m},\frac{\psi(q_n)}{q_n}\Big) \right)   \, .
\end{multline}

\vspace{3ex}

We stress that the above estimates, in particular \eqref{sum_const_leq_linear} and \eqref{ub_sum_eps_psipsi_combined} are valid under the hypotheses of  both Proposition~\ref{indiethm1} and Proposition~\ref{indiethm4}.

\vspace{3ex}

\begin{proof}[Completing the proof of Proposition~\ref{indiethm1}]
Working under the hypotheses of Proposition~\ref{indiethm1}, we can employ
Lemma \ref{prop_sum_S_m_n_lacunary} with $\alpha=1/2$  to obtain that
\begin{equation} \label{main_q_Smn_ub}
\mathop{\sum\sum}_{a\leq m <n\leq b} |T(m,n)| \, \ll \, \sum_{n=a}^{b}\frac{\psi(q_n)}{\varepsilon_n^{1/2}}\, \ll \,\frac{1}{\varepsilon_b^{1/2}}\sum_{n=a}^{b}\psi(q_n)\, \ll\, \left( \sum_{n=a}^{b}\psi(q_n)\right)^{1+\frac{\delta}{2}} .
\end{equation}
\noindent Also,  since the sequence $\ca{A}=(q_n)_{n\in\N}$ is lacunary, there exists a constant  $K>1$ such that for any integers $m < n$
$$
q_n \, \geq \, K^{n-m} \, q_m  \, ,
$$
and so
\begin{eqnarray} \label{sum_gcd_is_bounded}
\mathop{\sum\sum}\limits_{a\leq m<n\leq b} (q_m,q_n)\min\left(\frac{\psi(q_m)}{q_m},\frac{\psi(q_n)}{q_n}\right)& \leq & \mathop{\sum\sum}\limits_{a\leq m<n\leq b}q_m\,  \frac{\psi(q_n)}{q_n} \nonumber \\[2ex]
 & \leq & \sum_{n=a}^ b \psi(q_n) \sum_{m=1}^{n-1} \frac{q_m}{q_n}  \ \ll \  \sum_{n=a}^{b}\psi(q_n)  \,  .
\end{eqnarray}
\noindent Hence, on combining the estimates   \eqref{sum_const_leq_linear}, \eqref{ub_sum_eps_psipsi_combined}, \eqref{main_q_Smn_ub}  and \eqref{sum_gcd_is_bounded} we find that
\begin{multline*}
\mathop{\sum\sum}_{a\leq m<n\leq b} \mu(E_{q_m}^{\gamma}\cap E_{q_n}^{\gamma}) \ \leq \  4\left( \sum_{n=a}^{b}\psi(q_n)\right)^2 \,
+\ O\Bigg(\left(\sum_{n=a}^{b}\psi(q_n)\right)^{2-\delta}\hspace{-3mm}\log^+\left(\sum_{n=a}^{b}\psi(q_n)\right)\\ + \ \left(\sum_{n=a}^{b}\psi(q_n)\right)^{1+\frac{\delta}{2}} \Bigg).
\end{multline*}
 To complete the proof of Proposition \ref{indiethm1}, we set $\delta = 2/3$ in the above and apply Lemma~\ref{lem_sum_mu_good}.
\end{proof}

\vspace{3ex}

\begin{proof}[Completing the proof of Proposition~\ref{indiethm4}]
Working under the hypotheses of Proposition~\ref{indiethm4}, we can employ
Lemma \ref{prop_sum_S_m_n_separation} with $\alpha$   given by the $\alpha$-separated   sequence $\ca{A}$.   Note that condition \eqref{psi_lb} on
$\psi$ is guaranteed by \eqref{psi_is_not_small} while  \eqref{epsilon_inverse_bound} shows that condition
\eqref{psi_lbsv} on $ \varepsilon_n $ is satisfied.
On combining \eqref{Tmnprop5}, \eqref{sum_const_leq_linear} and  \eqref{ub_sum_eps_psipsi_combined} we find that
\begin{multline*} 
\mathop{\sum\sum}_{a\leq m<n\leq b} \mu(E_{q_m}^{\gamma}\cap E_{q_n}^{\gamma}) \ \leq \  4\left( \sum_{n=a}^{b}\psi(q_n)\right)^2 + O\left(\mathop{\sum\sum}_{a\leq m<n\leq b} (q_m,q_n)\min\left(\frac{\psi(q_m)}{q_m},\frac{\psi(q_n)}{q_n}\right)\right) \\[2ex]
+ O\left(\left(\sum_{n=a}^{b}\psi(q_n)\right)^{2-\delta}\hspace{-2mm}\log^+\left(\sum_{n=a}^{b}\psi(q_n)\right)+\sum_{n=a}^{b}\psi(q_n)\right).
\end{multline*}
To complete the proof of Proposition \ref{indiethm4}, we set $\delta = 1$ in the above  and use Lemma~\ref{lem_sum_mu_good}.
\end{proof}

\section{Deducing Theorem \ref{main_q}  from Theorem \ref{mainSV}   \label{deducing}}

Recall,  that any increasing sequence $\ca{A} \subseteq \cA_{\cS} $   satisfies the  growth condition  \eqref{hyp_q}  with $B = k$ and $C= (\log 2)/2 $ -- see Appendix~B~of~\S\ref{app} for the details.   Thus, Theorem \ref{main_q} will follow from Theorem \ref{mainSV} on showing that:
\begin{itemize}
\item[(i)] any such sequence is $\alpha$-separated,  and
\item[(ii)] the gcd term $E(N)$ appearing in the `error'  term of \eqref{countFSPresultsv} is less than the `main' term $\Psi(N)$; that is to say
    \begin{equation} \label{E_ll_Psi}
E(N)= O(\Psi(N))  \, .
\end{equation}
 \end{itemize}
The point is once we have (i) the hypotheses  of Theorem \ref{mainSV} are verified and the theorem implies that the counting function  $ R(x,N)  $ satisfies \eqref{countFSPresultsv} for $\mu$-almost all $x\in F$.  On the other hand, \eqref{countFSPresultsv} trivially coincides with \eqref{countFSPresult}  once we have (ii) and thus completes the proof of   Theorem \ref{main_q}.


%

\subsection{ Showing that any $\ca{A} \subseteq \cA_{\cS} $ is $\alpha$-separated \label{svalpha}}


The goal of this section is to prove the following statement.

\begin{prop} \label{prop_sum_S_m_n_F}
Let $\ca{A}= (q_n)_{n\in \N} \subseteq \cA_{\cS}  $  be an increasing sequence of natural numbers. Then, $\ca{A}$ is $\alpha$-separated for any $\alpha\in(0,1)$.
\end{prop}

The proof of the proposition we will make essential use of a fundamental theorem due to Baker $\&$ W\"ustholz \cite{BW} in the theory of linear forms in logarithms. The following statement is a simplified version of that appearing in  \cite{BW}.  It is more than adequate for the application we have in mind.

\medskip

\begin{thhjkbw} {\it
Let $n\in\mathbb{N}$,  $b_1 \ldots b_n\in\mathbb{Z}$ and  $a_1 \ldots  a_n \in\mathbb{N}$. Suppose that
$$
\Lambda:=\sum_{k=1}^n b_k\log a_k\ne 0.
$$
Then,
$$
\log|\Lambda|>-C(n)\cdot\prod_{k=1}^n\max\left(1,\log a_k\right)\cdot \log \Big( \max\left(1,|b_1|,\dots,|b_n|\right)  \Big)
$$
where
\begin{equation} \label{def_C_n}
C(n):=18(n+1)! \, n^{n+1} (32)^{n+2}\log(2n).
\end{equation}}
\end{thhjkbw}

\medskip

\begin{proof}[Proof of Proposition~\ref{prop_sum_S_m_n_F}]
Let $\alpha\in(0,1)$. The aim is to show that there  exists a constant $m_0 \in \N$  so that for any natural numbers $ m <  n$,  if
\begin{eqnarray}
1 &\leq& |sq_m-t q_n|< q_m^{\alpha}  \label{prop_BW_L}
\end{eqnarray}
 for some $s,t \in \N$ with
 \begin{equation}
s \leq  m^{12}  \label{prop_BW_s_2} \, ,
\end{equation}
then $m \le m_0$. With this in mind, first of all note that ~\eqref{prop_BW_s_2} and  ~\eqref{prop_BW_L} imply that
\begin{equation} \label{t_ub_1}
t \, <  \, \frac{q_m}{q_n}s+\frac{1}{2} \, \leq \, \frac{q_m}{q_n}m^{12}+\frac{1}{2}   \, < \,  m^{12}+\frac{1}{2}  \, .
\end{equation}
Hence, taking into account that $m$ and $t$ are integers, we have that
\begin{equation} \label{t_ub_2}
t\leq m^{12}.
\end{equation}
From the second  inequality appearing in \eqref{t_ub_1} and the fact that $t \in \N$, it follows that
$$ 1 \le  t<\frac{q_m}{q_n}m^{12} + \frac{1}{2}\, .   $$
Hence,
\begin{equation*} \label{ub_n_m}
q_n<2m^{12}q_m
\end{equation*}

\noindent which together with the fact that  $q_m<q_n$,  implies that
\begin{equation} \label{qm_qn_double_ie}
\log q_m < \log q_n < \log q_m + 12\log m + \log 2.
\end{equation}

\noindent Note that $sq_m \ne t q_n$ because of~\eqref{prop_BW_L} and assume for  the moment
that $sq_m< t q_n$.
Then, on using the fact that $ \exp (x) -1 \ge  x $ for any $ x > 0$,  it follows that
\begin{eqnarray} \label{pretty}
|sq_m-t q_n| & = & t q_n-s q_m \nonumber  \\[1ex]
& = & sq_m\left(\exp\left(\log t+\log q_n-\log s-\log q_m\right)-1\right) \nonumber \\[1ex]
&\geq &  sq_m\left(\log t+\log q_n-\log s-\log q_m\right).
\end{eqnarray}
We now proceed to  estimate the quantity involving the logarithm terms.   On using the fact that $ q_m, q_n \in \cA_{\cS} $,   it follows via Theorem~BW that
\begin{equation}
|\log t+\log q_n-\log s-\log q_m | \, \geq \,  \exp\Big(-C(k+2)\log s\log t \prod_{p\in\cS}\log p \log\max_{a\in\mathcal{E}_{m,n}} a \Big) \, .   \label{lf_estimate_2sv}
\end{equation}
Here $C(k+2)$ is the constant associated with Theorem~BW
and  $\mathcal{E}_{m,n}$ is the set of exponents of prime powers in the canonical factorisation of the integers $q_m, q_n \in  \cA_{\cS} $.  By definition, for any $a \in\mathcal{E}_{m,n}$ we have $2^a\leq q_n$ and so
 $a\leq  \log q_n/ \log 2$. This together with the upper bound  estimates ~\eqref{prop_BW_s_2} and~\eqref{t_ub_2}, implies  that
\begin{eqnarray} \label{lf_estimate_3}
\log s\log t \prod_{p\in\cS}\log p \log\max_{a\in\mathcal{E}_{m,n}} a &\leq  & 144\left(\log m\right)^2\Big(\prod_{p\in\cS}\log p\Big)\log\left(\log q_n/\log 2\right)   \nonumber \\[2ex]
&\ll  &  k^2 \, \left(\log\log q_m\right)^3,
\end{eqnarray}
where in the last step we have also used \eqref{qm_qn_double_ie} and \eqref{hyp_q} with $B=k$.  Hence
\begin{equation}
|\log t+\log q_n-\log s-\log q_m |  \, \geq \,  \exp\left(-  \tilde{C} \left(\log\log q_m\right)^3 \right) \, ,   \label{lf_estimate_2slv}
\end{equation}
where the constant $\tilde{C}$ depends on the set $\cS$ only. This together with   \eqref{pretty} yields that
\begin{equation} \label{lf_estimate_4}
|sq_m-t q_n|=t q_n-s q_m\geq\exp\left(\log q_m-\tilde{C}\left(\log\log q_m\right)^3\right)  \, .
\end{equation}
Now if  $sq_m> t q_n$, the above above argument can easily be modified to show that~\eqref{lf_estimate_4} still holds. Indeed, on using the fact that $q_m<q_n$,   we find that
\begin{eqnarray*} \label{pretty*}
|sq_m-t q_n|  & = & s q_m  -  t q_n  \\[1ex]
&\geq &  tq_n\left(\log s+\log q_m-\log t-\log q_n\right) \\[2ex]
&  \stackrel{\eqref{lf_estimate_2slv}}{\geq} & t q_n \exp\big(-  \tilde{C} \left(\log\log q_m\right)^3 \big)
\\[2ex]
&  \geq &  \exp\big( \log q_m -  \tilde{C} (\log\log q_m)^3 \big)  \, .
\end{eqnarray*}

\medskip

Comparing~\eqref{lf_estimate_4} with~\eqref{prop_BW_L}, we find that
\begin{equation*}
\log q_m-\tilde{C}\left(\log\log q_m\right)^3 \leq \alpha\log q_m  .
\label{uiop}
\end{equation*}
We can rewrite this as
\begin{equation} \label{lf_estimate_5}
\tilde{C}\left(\log\log q_m\right)^3 \geq (1-\alpha)\log q_m.
\end{equation}
As $\alpha<1$, it is evident that this inequality can only hold for $m$ not exceeding some integer  $m_0$ that depends on  $\cS$ and $\alpha$ only.  This completes the proof of the proposition.

\end{proof}


\subsubsection{A stronger version of Proposition~\ref{prop_sum_S_m_n_F} involving Property D
\label{rem_Hypothesis_F}  }

In this section we show  that the proof of  Proposition~\ref{prop_sum_S_m_n_F} can be easily adapted to prove the analogous statement for sequences $ \ca{A} $ satisfying Property D  -- see Remark \ref{remhypF} of \S\ref{bl} for the defintiion.  Note that in the proof of Proposition~\ref{prop_sum_S_m_n_F} we made direct use of the fact  that the growth condition  \eqref{hyp_q} is satisfied for $\cA \subseteq \cA_{\cS}$.  By definition,  this condition automatically holds for  sequences satisfying Property D.

Formally, we establish the following generalisation of  Proposition~\ref{prop_sum_S_m_n_F}.

\begin{thpropa}
Let $\ca{A}= (q_n)_{n\in \N} $  be an increasing sequence of natural numbers that satisfies Property D. Then, $\ca{A}$ is $\alpha$-separated for any $\alpha\in(0,1)$.
\end{thpropa}

\begin{proof}[Proof (sketch)]  The proof is exactly the same as that of Proposition~\ref{prop_sum_S_m_n_F} up to and including the inequality given by \eqref{pretty}.  The main modifications after that are as follows:

\begin{itemize}
  \item The product ${\prod_{p\in\cS}}$ appearing in \eqref{lf_estimate_2sv} and thereafter needs to be replaced by ${\prod_{p\in\mathcal{P}_{m,n}}}$ in which $\mathcal{P}_{m,n}$ denotes the set of all prime divisors of $q_m \, q_n$. Also the constant associated with Theorem~BW is $C(2D+2)$ where $D$ is constant coming from Property D (namely part (a) of condition (ii)).  Note that since $\cA $ satisfies Property D, we have that for $n_0 <m<n $ (which, without loss of generality, we  can assume)
      \begin{equation} \label{yuck}
        \#   \mathcal{P}_{m,n}\leq 2\cdot D \, .
      \end{equation}

  \item The upper bound for $p\in\mathcal{P}_{m,n}$ coming from Property D (namely part (b) of condition (ii))  and \eqref{yuck} have to be added to the list of the upper bound inequalities~\eqref{prop_BW_s_2} and~\eqref{t_ub_2} used to derive the analogue of \eqref{lf_estimate_3}; namely, for $n_0 <m<n $ with $m $ sufficiently large
$$
\begin{aligned}
\log s\log t \prod_{p\in\mathcal{P}_{m,n}} \!\!\! \log p \; \log\max_{k\in\mathcal{E}_{m,n}} k &\leq \frac{144}{(\log 2)^{2 D}}\left(\log m\right)^2
\Big( \left(\log q_n\right)^{\frac{1-\epsilon}{2D}} \Big)^{2D}
    \log\left(\log q_n \! - \! \log 2\right)\\[2ex]
&\ll B^2 \left(\log q_m\right)^{1-\epsilon} (\log\log q_m)
\\[2ex]
&\ll \left(\log q_m\right)^{1-\epsilon/2},
\end{aligned}
$$
where $B$ is the constant associated with the growth condition \eqref{hyp_q}.

\end{itemize}
With the above main modifications in mind, we continue exactly as  in the proof of Proposition~\ref{prop_sum_S_m_n_F} and obtain the following analogue of
\eqref{lf_estimate_5}
\begin{equation*}
\tilde{C}\left(\log q_m\right)^{1-\epsilon/2} \geq (1-\alpha)\log q_m \, ,
\end{equation*}
where $\tilde{C}$ is a constant that depends on $B$ and $D$ only.
As $\alpha<1$, it is evident that this inequality can only hold for $m$ not exceeding some integer  $m_0$ that depends only  on  the constants associated with Property D, $\alpha$ and $\epsilon$.  This completes the proof of the proposition.
\end{proof}

%

\subsection{ Showing that $E(N)= O\big(\Psi(N) \big) $ \label{s_gcd}}

The goal  of this section is to establish \eqref{E_ll_Psi}.  As we shall soon see,  this is an immediate consequence of the following statement.

\begin{thm} \label{thm_gcd}
 Let $\ca{A}= (q_n)_{n\in \N} \subseteq \cA_{\cS} $  be an increasing sequence of natural numbers.    Then, there exists a constant $C$ which depends only on the cardinality $k$ of  $\cS$, such that for any integer $n \ge 2$
\begin{equation} \label{thm_gcd_sum}
\sum_{m=1}^{n-1}
\frac{(q_m,q_n)}{q_n}\leq C  \, .
\end{equation}
\end{thm}

\bigskip

\noindent Clearly,  Theorem~\ref{thm_gcd} implies that
$$
\begin{aligned}
E(N) &:=\ \mathop{\sum\sum}_{1\leq m<n\leq N} (q_m,q_n)\min\left(\frac{\psi(q_m)}{q_m},\frac{\psi(q_n)}{q_n}\right)
\\[2ex]
&\leq
\mathop{\sum\sum}_{1\leq m<n\leq N} \frac{(q_m,q_n)}{q_n}\psi(q_n)
\  \ll  \ 
\sum_{n=1}^N\psi(q_n):=\Psi(N)
\end{aligned}
$$

\noindent and so yields the desired goal.
Note the left hand side of~\eqref{thm_gcd_sum} can only increase if we enlarge our sequence $\ca{A}= (q_n)_{n\in \N}  $.   So, without loss of generality, we can assume that $ \ca{A} = \cA_{\cS}$ during the course of establishing Theorem~\ref{thm_gcd}.  With this in mind, we start by proving a couple of useful lemmas.

\begin{lem} \label{lem_tech_log}
For any $s\in\R$, there exists a constant $C_s$, which depends on $s$ only, such that, for any $r\in\N$, and any integers $n_1,\dots,n_r \geq 2,$
$$ \sum_{t_1=1}^{\infty}\dots\sum_{t_r=1}^{\infty}\frac{\left(t_1\log_2 n_1+\dots+t_r\log_2 n_r\right)^s}{n_1^{t_1}\cdot\dots\cdot n_r^{t_r}}<C_s  . $$
\end{lem}
\begin{proof}

First note that the function defined for $x\geq 1$ by $x\to \dfrac{\left(\log_2 x\right)^s}{\sqrt{x}} $ is bounded above by a constant. Define
$$ C_s = \sup_{x\geq 1}\frac{(\log_2 x)^s}{\sqrt{x}} \, >0 \,. $$
Then,
\begin{eqnarray*}
\sum_{t_1=1}^{\infty}\dots\sum_{t_r=1}^{\infty}\frac{\left(t_1\log_2 n_1+\dots+t_r\log_2 n_r\right)^s}{n_1^{t_1}\cdot\dots\cdot n_r^{t_r}} & \leq & \sum_{t_1=1}^{\infty}\dots\sum_{t_r=1}^{\infty}\frac{C_s}{n_1^{t_1/2}\cdot\dots\cdot n_r^{t_r/2}} \\[2ex]
&\leq & C_s\cdot\frac{1}{(\sqrt{n_1}-1)\cdot\dots\cdot(\sqrt{n_r}-1)} \\[2ex]
&\leq & C_s ,
\end{eqnarray*}
and this proves the lemma.
\end{proof}
\begin{lem} \label{lem_tech_subsum}
Let $K\geq 1$  be a real number and $n \in \N$.
Then
\begin{equation} \label{lem9concl}
\sum_{\substack{q_m | q_n \\q_m\cdot K<q_n}}\frac{(q_m,q_n)}{q_n}\, \leq \, \frac{\left(\log_2 K+2\right)^{k-1}}{K}\cdot\prod_{i=1}^{k}\frac{p_i}{p_i-1}.
\end{equation}
\end{lem}
\begin{proof}
The statement  is obviously true when $n=1$. Observe that the condition that $q_m|q_n$ and $q_m K<q_n$ is equivalent to
\begin{equation} \label{product_K}
\frac{q_n}{q_m}\equiv \prod_{i=1}^k p_i^{a_i} > K
\end{equation}
for some integers $a_1,\ldots,a_k\geq 0$. Since $n$ is fixed, the $k$--tuple $(a_1,\ldots,a_k)$ depends exclusively on $m$.  We will say that the divisor
 $q_m $ of $q_n$ is  \emph{maximal} if and only if \eqref{product_K} holds and for any other divisor $q_l $ of $q_n$ such that $q_lK< q_n$ and $q_m|q_l$ we necessarily have that $q_m=q_l$.   It then follows that
 \begin{eqnarray*}
\sum_{\substack{q_m | q_n\\q_m K<q_n}}\frac{(q_m,q_n)}{q_n}  & \leq & \sum_{q_m \text{maximal}} \ \sum_{q_l|q_m}\frac{(q_l,q_n)}{q_n}
 \ = \  \sum_{q_m \text{maximal}} \ \sum_{q_l|q_m}\frac{q_l}{q_n} \\[3ex]
&=& \sum_{q_m \text{maximal}} \ \sum_{q_l|q_m}\frac{q_l}{q_m}\frac{q_m}{q_n} 
 \ \leq \  \sum_{q_m \text{maximal}} \ \frac{1}{K}\sum_{q_l|q_m}\frac{q_l}{q_m} \\[3ex]
& \leq & \sum_{q_m \text{maximal}} \ \frac{1}{K}  \sum_{c_1=0}^{\infty}\cdots\sum_{c_k=0}^{\infty} \frac{1}{p_1^{c_1}\cdots p_k^{c_k}}  \\[2ex]
& \leq & \sum_{q_m \text{maximal}} \ \frac{1}{K}\prod_{i=1}^{k}\frac{p_i}{p_i-1} \, , \\[2ex]
\end{eqnarray*}
 where the outer sum on the right hand side is over all maximal divisors $q_m$ of  $q_n$.  Thus,  the proof of the lemma is reduced to showing that  the number of such  divisors is bounded above by  $\left(\lfloor\log_2 K\rfloor+2\right)^{k-1}$; that is
 \begin{equation} \label{paris}
\sum_{q_m \text{maximal}}  \!\!\!\! 1 \ \leq \ \left(\lfloor\log_2 K\rfloor+2\right)^{k-1} \, .
\end{equation}

With this in mind, observe that~\eqref{product_K} gives
\begin{equation} \label{sum_K}
\sum_{i=1}^k a_i\log p_i > \log K \,
\end{equation}

\noindent and  so $q_m$ is maximal if and only if the corresponding solution $(a_1,\dots,a_k)\in\Z_{\geq 0}^k$ to inequality~\eqref{sum_K} is \emph{minimal}, in the sense that for any other solution $(b_1,\dots,b_k)\in\Z_{\geq 0}^k$ with $b_i\leq a_i$ for all $i=1,\dots,k$, we necessarily have that $b_i=a_i$  for all $i=1,\dots,k$.    It is easily versified that if  $(a_1,\dots,a_k)$ is a minimal solution to \eqref{sum_K}, then
 \begin{equation} \label{paris2}
 a_1+\dots+a_k\leq \log_2 K+1.
 \end{equation}
Indeed,  to show that this is so, assume on the contrary  that $(a_1,\dots,a_k)$ is minimal and $a_1 + \ldots + a_k > \log_2K + 1$. Then, without loss of generality, we may assume $a_1 \geq 1$. Then
$$ (a_1 - 1)\log p_1 + a_2\log p_2 + \ldots + a_k\log p_k \geq (a_1 - 1) + a_2 + \ldots + a_k \geq \log_2 K .
$$
This means that $(a_1-1, a_2, \ldots, a_k)$ is a solution to  \eqref{sum_K} and thus contradicts the fact that $(a_1,\ldots, a_k)$ is minimal.

\medskip

By definition, \eqref{paris} is equivalent to the statement that number of minimal solutions to~\eqref{sum_K} is bounded above by $\left(\lfloor\log_2 K\rfloor+2\right)^{k-1}$.
This we now proceed to prove. Define the map from the set of minimal solutions $(a_1,\dots,a_k)\in\Z_{\geq 0}^k$  of~\eqref{sum_K} to the set of $k$-tuples $(b_1,\ldots, b_k)\in\Z_{\geq 0}^k$  satisfying $ b_1+\ldots+b_k = \lfloor\log_2 K\rfloor+1$ by
\begin{equation} \label{map_to_splits}
(a_1,\dots,a_k)\to  (b_1,\ldots, b_k):= \left(a_1,\dots,a_{k-1}, \lfloor\log_2 K\rfloor+1-\sum_{i=1}^{k-1}a_i\right).
\end{equation}
In view of \eqref{paris2}, $b_k \geq 0$ and so the map is well-defined. The map is also injective. Indeed, assume $(x_1,\dots,x_k)$ and $(y_1,\dots,y_k)$ are two distinct minimal solutions to \eqref{sum_K} with the same image under the map~\eqref{map_to_splits}. Then necessarily $x_i=y_i$ for all $i=1,\dots,k-1$, and since the solutions are distinct, either $x_k < y_k$ or $x_k>y_k$. This means that one of the solutions is not minimal, which is a contradiction. Thus the map defined via~\eqref{map_to_splits} is injective, whence the number of minimal solutions to~\eqref{sum_K} is at most equal to the number of $k$-tuples satisfying $ b_1+\ldots+b_k = \lfloor\log_2 K\rfloor+1$; namely
$$ \binom{\lfloor\log_2 K\rfloor+k}{k-1}=\frac{\left(\lfloor\log_2 K\rfloor+k\right)\cdot\dots\cdot \left(\lfloor\log_2 K\rfloor+1\right)}{(k-1)!}\leq \left(\lfloor\log_2 K\rfloor+2\right)^{k-1}. $$
This  thereby completes the proof of the lemma.

\end{proof}

\begin{proof}[Proof of Theorem~\ref{thm_gcd}]
As already mentioned,  it suffices to proves the theorem with $ \ca{A} = \cA_{\cS}$. With this in mind,   we are  given  an integer $n \ge 2 $ and so this fixes
\begin{equation} \label{def_q_n}
q_n=\prod_{i=1}^k p_i^{a_i}  \hspace{6mm} (a_1,\ldots,a_k\geq 0).
\end{equation}
For any integer $m<n$,  we write
\begin{equation} \label{def_q_m}
q_m=\prod_{i=1}^k p_i^{b_i}  \hspace{6mm} (b_1,\ldots,b_k\geq 0) \, ,
\end{equation}
where the exponents $b_1,\ldots,b_k$ depend on the index $m$. Obviously,  since $m < n$ \begin{equation} \label{logexpon}
\sum_{i=1}^{k} b_i\log p_i < \sum_{i=1}^{k} a_i\log p_i \,  ,
\end{equation}
and for each such $m $ we set
$$ \cP(q_m):=\left\{ 1\leq i \leq k :  b_i < a_i \right\} . $$

\noindent It then follows that
\begin{eqnarray*}
\sum_{m=1}^{n-1}\frac{(q_m,q_n)}{q_n} &=& \mathop{\mathop{\sum}_{(b_1,\ldots,b_k) \in \Z_{\geq 0}^k  : }}_{ \text{\eqref{logexpon} holds}}  \ \ \frac{p_1^{\min(a_1,b_1)}\cdots p_k^{\min(a_k,b_k)}}{p_1^{a_1}\cdots p_k^{a_k}} \\[2ex]
&=& \sum_{\cT \subseteq \{1, \ldots,k \}}  \quad  \mathop{\mathop{\mathop{\sum}_{(b_1,\ldots,b_k) \in \Z_{\geq 0}^k  :}}_{ \text{\eqref{logexpon} holds $\& $ }} }_{\text{ }\cP(q_m)=\cT}   \ \ \frac{\prod_{i\in \cT}p_i^{b_i}\cdot \prod_{i\notin \cT}p_i^{a_i} }{\prod_{i\in \cT}p_i^{a_i} \cdot \prod_{i\notin \cT}p_i^{a_i} }  \, . \\
\end{eqnarray*}
In view of the fact that \eqref{logexpon} is imposed as a condition on the inner sum, we can assume that $\cT \neq \emptyset$. With this in mind, it follows that

\begin{eqnarray*}
\sum_{m=1}^{n-1}\frac{(q_m,q_n)}{q_n} &\leq &
 \sum_{\cT \subset \{1, \ldots,k \}}  \quad   \sum_{b_i\geq a_i \,  : \  i\notin \cT}   \quad     \mathop{\mathop{\sum}_{ b_i< a_i \, : \ i\in \cT \ \text{$\&$}  }}_{ \prod_{i\in\cT}p_i^{b_i}\cdot K < \prod_{i\in\cT} p_i^{a_i}}   \frac{\prod_{i\in\cT}p_i^{b_i}}{ \prod_{i\in\cT}p_i^{a_i}}   \\[3ex]
&  &  \qquad  \qquad  +   \  \quad
\mathop{\mathop{\sum}_{ b_i< a_i \, : \  i\in \cT= \{1, \ldots,k \} \ \text{$\&$}   }}_{ \prod_{i\in\cT}p_i^{b_i}  < \prod_{i\in\cT} p_i^{a_i}}  \ \ \frac{\prod_{i\in\cT}p_i^{b_i}}{ \prod_{i\in\cT}p_i^{a_i}}  \  ,
\end{eqnarray*}
where
$ K:= \prod_{i\notin\cT}p_i^{b_i-a_i}  \, . $  Now on appealing to Lemma~\ref{lem_tech_subsum} and then Lemma~\ref{lem_tech_log}, we find that
\begin{eqnarray*}
\sum_{m=1}^{n-1}\frac{(q_m,q_n)}{q_n} &\leq &
 \sum_{\cT \subset \{1, \ldots,k \}}  \quad   \sum_{b_i\geq a_i \,  : \  i\notin \cT}   \
\frac{\left(\log_2 K +2\right)^{k-1}}{K}\prod_{i\in\cT}\frac{p_i}{p_i-1}   \\[1ex]
&  &  \hspace*{10ex} +   \  \quad  2^{k-1}   \prod_{i=1}^k\frac{p_i}{p_i-1}  \\[2ex]
& \leq  & \sum_{\cT \subset \{1, \ldots,k \}}  \quad   \sum_{b_i\geq a_i \,  : \  i\notin \cT}   \
\frac{\left(2 \log_2 K \right)^{k-1}}{K}   \ \cdot  \  2^k   \quad + \quad  2^{k-1} \cdot 2^k \\[3ex]
& \leq  &  2^{k-1} \cdot 2^k \Big( C_{k-1}  \sum_{\cT \subset \{1, \ldots,k \}} \!\!\!\! 1 \   +  \ 1 \Big)  \; ,
\end{eqnarray*}
where  $C_{k-1}>0$ is the constant associated with Lemma~\ref{lem_tech_log}.
This together with the fact that there  are at most  $2^k$ different subsets
 $\cT$ of  $\{1, \ldots,k \}$ implies the desired statement.
\end{proof}


\section{Appendices \label{app}}

\subsection*{Appendix A:  Theorem DEL $ \Longrightarrow $ Corollary DEL }
\stepcounter{subsection}

\noindent The goal  is to deduce  Corollary DEL from   Theorem DEL -- the fundamental theorem of Davenport, Erd\"os $\&$ LeVeque in the theory of uniform distribution.

\medskip

 We are given that  $\ca{A}=(q_n)_{n \in \N}$ is a  lacunary sequence of natural numbers.   Thus, there exists a constant $K>1$ such that any  integers $m<n$
\begin{equation} \label{lacunarityratio}
q_n - q_m\, =\, q_m\left(\frac{q_n}{q_m} - 1 \right)\, \geq \,  K^{n-m} -1 \, .
\end{equation}
Now let $f$ be as in Corollary DEL and consider the associated function $F : \N \to  \R^+ $ given by $$ F(n) := f(K^{n} - 1) \, .   $$

\noindent Note that $F $ is  decreasing  (since $f$ is decreasing) and so it follows that  the convergence condition \eqref{lacdelsum} is equivalent to the condition that
\begin{equation}\label{lacdelsumequiv}
\sum_{n=1}^{\infty}\frac{F(n)}{n}\, <\, \infty \, .
\end{equation}
Also, by the decay condition   \eqref{lacdelmu} and the fact that  $\widehat{\mu}(0)=1$,  it follows that  for any integer $h\neq 0$
\begin{eqnarray*}
\sum_{m,n=1}^{N}\widehat{\mu}(h(q_n-q_m)) \, &=& \, \sum_{n=1}^{N}\widehat{\mu}(0) \, + \ 2 \!\! \mathop{\sum\sum}_{1\leq m<n\leq N}\widehat{\mu}(h(q_n-q_m)) \\[2ex]
&\stackrel{\eqref{lacdelmu}}{\ll}  & N \,+\, \mathop{\sum\sum}_{1\leq m<n\leq N}f\big(|h(q_n-q_m)|\big) \\[2ex]
& \leq & N \,+\, \mathop{\sum\sum}_{1\leq m<n\leq N}f(q_n-q_m) \\[2ex]
& \stackrel{\eqref{lacunarityratio}}{\leq}  & N\, + \, \mathop{\sum\sum}_{1\leq m<n\leq N}F(n-m) \\[2ex]
& \leq & N \, + \, \sum_{n=1}^{N}\sum_{m=1}^{n}F(m) 
 \ =  \ N \,+\, \sum_{n=1}^{N}(N+1-n)F(n)  \\[2ex]
& \ll &  N\sum_{n=1}^{N}F(n) \, .
\end{eqnarray*}

\noindent The upshot of this is that
\begin{eqnarray*}
\sum_{N=1}^{\infty}\frac{1}{N^3}\sum_{m,n=1}^{N}\widehat{\mu}(h(q_n-q_m))  & \ll  & \sum_{N=1}^{\infty}\frac{1}{N^2}\sum_{n=1}^{N}F(n) = \sum_{n=1}^{\infty}\sum_{N=n}^{\infty}\frac{F(n)}{N^2}  \\[2ex]
& \ll  &  \sum_{n=1}^{\infty}\frac{F(n)}{n}  \ \stackrel{\eqref{lacdelsumequiv}}{<}  \ \infty  \, .
\end{eqnarray*}
Thus,  Theorem DEL  implies that the sequence $\ca{A}=(q_n)_{n \in \N}$ is uniformly distributed modulo one for $\mu$--almost all $x\in F$.  This completes the proof of Corollary DEL.

\subsection*{Appendix B: The growth rate of $\cA_{\cS}$}
\stepcounter{subsection}

The goal is to  show the sequence  $\cA_{\cS}=(q_n)_{n\in \N}$ defined by \eqref{kprimes}, satisfies the growth condition
\begin{equation*} \label{growthrateofqn}
\log q_n \,  >  \,  C  \,  n^{1/k}    \qquad \forall \ n \ge 2 \, ,
\end{equation*}
where $k$ is the cardinality of the finite set  $\cS:=\{p_1, \ldots,p_k\}  $ of distinct primes and   $C= (\log 2)/2$.  It is easily seen that this is an immediate consequence of the following counting statement: for  $X \ge2$
\begin{equation} \label{countingrate}
\pi_{\cS}(X) : = \#\left\{ q\in \cA_{\cS} : q\leq X \right\}  \ \leq  \  C^{-1} (\log X)^k   \, .
\end{equation}

\noindent Indeed,  given   $ q_n \in \cA_{\cS} $ with $n \ge 2$, put  $X = q_n$.  Then $ X \ge 2$ and  \eqref{countingrate} implies  that
$$ \pi_{\cS}(q_n)  \, =  \,  n \leq C^{-1} (\log q_n)^k \, $$
and we are done since $C^{1/k}  \ge C$.

We prove \eqref{countingrate} by induction on $k$. When $k=1$,  we have only one prime  $p$ and by definition   $\pi_{\cS}(X) : = \#\left\{ p^{n-1} : p^{n-1} \leq X \right\}$.   The condition that $ p^{n-1} \leq X  $ implies that
$$
n \le \frac{\log X}{\log p}  + 1  \ \le \ \    \frac{2}{\log 2} \, \log X   \,  =  \,  C^{-1} \, \log X \ .
$$
This verifies \eqref{countingrate} when $k=1$. Now assume \eqref{countingrate} is true for any set of $k$ distinct primes and let $\cS=\{p_1,\ldots,p_k,p_{k+1}\}$ be a set of $k+1$ distinct primes.  Write $\cS_k$ for the set $  \{p_1,\ldots,p_k\}$ of $k$ distinct primes. It follows that
\begin{eqnarray*}
\pi_{\cS}(X) &:=& \mathop{\sum_{q\leq X}}_{q\in \cA_{\cS}}1 \, = \,  \mathop{\sum_{ (a_1,\ldots,a_{k+1}) \in \Z_{\geq 0}^{k+1}  : \,      }}_{\prod_{i=1}^{k+1} p_i^{a_i } \leq X} \hspace{-9mm} 1  \hspace{9mm} = \  \mathop{\sum_{ (a_1,\ldots,a_{k+1}) \in \Z_{\geq 0}^{k+1}  : \,      }}_{\sum_{i=1}^{k+1} a_i\log p_i \leq \log X} \hspace{-9mm} 1  \\[3ex]
& \leq & \sum_{0\le a_{k+1}\leq\frac{\log X}{\log p_{k+1}}}  \quad  \mathop{\sum_{ (a_1,\ldots,a_{k}) \in \Z_{\geq 0}^{k}  : \,      }}_{\sum_{i=1}^{k} a_i\log p_i \leq \log X} \hspace{-9mm} 1   \hspace{9mm} =    \quad  \sum_{0\le a_{k+1}\leq\frac{\log X}{\log p_{k+1}}} \hspace{-4mm} \pi_{\cS_k}(X)    \\[2ex]
& \leq & \sum_{0\le a_{k+1}\leq\frac{\log X}{\log p_{k+1}}}  \hspace{-4mm} C^{-1}  (\log X )^{k} \hspace{6mm} \text{ (by the induction hypothesis)} \\[2ex]
& \leq & \frac{C^{-1}}{\log p_{k+1}}(\log X)^{k+1}  \, <  \,  C^{-1}  \, (\log X)^{k+1} \, .
\end{eqnarray*}
This completes the inductive step and establishes \eqref{countingrate} for arbitrary $k \in \N$.

\subsection*{Appendix C: Example of `bad' sequences satisfying Property~D \label{propDbad}}
\stepcounter{subsection}

The goal is to  construct an increasing sequence $\cA= (q_n)_{n\in \N} $  of natural numbers satisfying Property D and an associated function $\psi:\N \rightarrow \I$ such that for all integers $ N\geq N_0$
\begin{equation}\label{bald}
\sum_{j=1}^N\psi(q_j)\sum_{m=1}^j\frac{(q_m,q_j)}{q_j}>\exp\left(c \sum_{j=1}^N\psi(q_j)\right) \, .
\end{equation}
Here $c>0$ and $N_0  \geq 1 $ are absolute constants.  This `strongly'  implies the claim
associated with \eqref{baden}  in Remark \ref{remhypF}.  Thus with reference to Theorem \ref{mainSV}, for arbitrary sequences satisfying Property~D, we can not  reduce \eqref{countFSPresultsv} to \eqref{countFSPresult}
as in the situation when $\ca{A} \subseteq \cA_{\cS}$.

\bigskip

\noindent {\em Step 1: Constructing the sequence $\cA$. \ } To start with, let  $(n_t)_{t\in\N}$ be an increasing sequence of natural numbers  satisfying the following conditions:

\begin{enumerate}
\item[$\bullet$] The integer $n_1$ is large enough so that
\begin{equation} \label{regular_2}
\ln\ln n_1>2\ln 2+1  \,
\end{equation}
and
\begin{equation} \label{regular_1}
2\log n +2\log\log n < n^{1/5}   \qquad \forall \ \ n \geq n_1 \, .
\end{equation}
\item[$\bullet$]
For all $t\in\N$,
\begin{equation} \label{lb_step}
2n_{t}\leq n_{t+1}.
\end{equation}
\end{enumerate}

\noindent It can be easily verified that \eqref{regular_1} implies \eqref{regular_2} but  it will be useful to have both explicitly stated.  Also,  note that in view of \eqref{regular_2} and~\eqref{lb_step}, it follows that
\begin{equation} \label{bu}
n_{t} \geq 2^{t-1}n_1\geq 2^{t-1} e^{4e}   \qquad \ \forall \ \    t \in \N \, .
\end{equation}
In particular, this implies
\begin{equation} \label{bu2}
\ln n_{t} \geq t\ln 2 + 4e - \ln 2 > (t+14)\ln 2.
\end{equation}

\vspace*{2ex}

\noindent Next, let  $\cP$ denote the set of all prime numbers and for $t\in\N$,  let
$$
\cP_t:=\left\{ p\in\cP \, : \,  3\leq p\leq n_t\log n_t \right\}.
$$
It follows from Rosser's theorem~\cite{Rosser1939}, that
\begin{equation} \label{ub_Rosser}
\# \cP_t  \,  <  \, n_t.
\end{equation}
Also a simple consequence of the well known lower bound estimate 
$$
\sum_{\substack{p \in \cP \\ p \leq n}} \frac{1}{p} \, \geq  \,  \ln\ln ( n+1) - \ln (\pi^2/6)   \,  ,
$$
is that
\begin{equation} \label{lb_sum_reciprocal_primes}
\sum_{p\in\cP_t}\frac{1}{p} \, \geq \, \ln\ln n_t-\ln 2-\frac12.
\end{equation}

\vspace*{2ex}

Now,   for each $t\in\N$ define
$$
\tilde{q}_t:=2^{n_t}  \, ,
$$
and in turn,  for any $p\in \cP_t $ let
\begin{equation} \label{def_tilde_q_t_p}
\tilde{q}_{t,p}:=  \, \tilde{q}_t \; 2^{-u_p-1} \, p  \, = \, 2^{n_t-u_p-1} \, p  \,
\end{equation}
where $
u_p:=\lfloor\log_2 p\rfloor  \, .
$ Then,  by definition
$
2^{u_p}\leq p<2^{u_p+1}
$
and it follows that
\begin{equation} \label{double_ie_tilde_q}
\frac12 \tilde{q}_t < \tilde{q}_{t,p}<\tilde{q}_t.
\end{equation}
Also, in view of~\eqref{regular_1} and \eqref{regular_2} we have that for every $t\in\N$ and $p\in\cP_t$
\begin{equation} \label{p_is_small}
\log p\leq \log n_t + \log\log n_t\leq \frac{ n_t^{1/5}}{2} =  \frac{\left(\log\tilde{q}_{t}\right)^{1/5}}{2(\log 2)^{1/5}}
\leq\left(\log\tilde{q}_{t,p}\right)^{1/5}.
\end{equation}
Trivially, the above upper bound estimate also holds for $p =2$.
Also note that since $q_1 \geq 2$,  for every $t \in \N$
$$
(\log  2)^5 < \log 2  \, \leq \, \log\tilde{q}_{1} \, \leq \, \log\tilde{q}_{t}
$$
and so
\begin{equation} \label{correct}
\log 2  <  \left(\log\tilde{q}_{t}\right)^{1/5}.
\end{equation}

\medskip

\noindent The desired sequence $ \cA := (q_j)_{j\in\N}$ is  precisely the elements of the set
$$
\{\tilde{q}_t  : t \in \N \} \ \cup \ \left\{\tilde{q}_{t,p}\mid t \in \N, \,   p\in\cP_t\right\}
$$
listed in increasing order of size.

\medskip

\noindent{\em Step 2: Verifying $\cA$ satisfies  Property~D.  \ }  
 By construction, each element of $\cA$ trivially has at most two prime divisors. Also, in view of~\eqref{p_is_small}  and~\eqref{correct},  any prime divisor $p$ of an element $q_j \in \cA$ satisfies
\begin{equation} \label{ub_prime_divisors}
\log p\leq \left(\log q_j\right)^{1/5}.
\end{equation}
This verifies  condition~(ii) of Property D with $D=2$.
It now remains to verify condition~(i) of Property D.
By construction, every element of the sequence $(q_j)_{j\in\N}$ is either equal  to $\tilde{q}_t$ for some $t \in \N$ or equal to  $\tilde{q}_{t,p}$ for some $t\in\N$ and $p\in\cP_t$.  Denote by $\pi$ the bijective map from the set of integers $ j \in \N$ to the set of couples $(t,p)$ with $t\in\N$ and $p\in\left(\cP_t\cup\{2\}\right)$ so that
$$
q_j=\tilde{q}_{\pi(j)}  := \, \tilde{q}_{t,p} ,
$$
Here and throughout, we use the notation $\tilde{q}_{t,2}:=\tilde{q}_{t}$.
Note that for any $t \in \N$, we have that $ \tilde{q}_{t,p} \leq \tilde{q}_{t} $ for every $p\in\left(\cP_t\cup\{2\}\right)$. Thus for any $ j \in \N$, given  $ \pi(j) = (t,s) $ it follows that
\begin{equation} \label{gc_1}
j \, \leq \, \sum_{k=1}^{t} \# \cP_k \, \stackrel{\eqref{ub_Rosser}}{<} \, \sum_{k=1}^{t}n_k   \, \stackrel{\eqref{lb_step}}{<} \, 2n_t  \, .
\end{equation}
On the other hand,
$$
q_j=\tilde{q}_{t,p}   \,  \stackrel{\eqref{double_ie_tilde_q}}{\geq} \,  \frac12\tilde{q}_t
$$
and so
\begin{equation} \label{gc_2}
\log q_j\geq\log \left(\tilde{q}_t/2\right)\geq n_t-1\geq n_t/2.
\end{equation}

\noindent On combining~\eqref{gc_1} and~\eqref{gc_2}, we obtain that
\begin{equation} \label{gc_result}
\log q_j>j/4.
\end{equation}
In other words, $\cA$ satisfies the growth condition  \eqref{hyp_q}  with   $B = 1$  and $C=1/4$. This verifies
condition~(i) of Property D.

\bigskip

\noindent {\em Step 3: A useful gcd estimate.  \ } Let $j\in\N$ be such that $\pi(j)=(t,2)$ for some $t\in\N$; that is,  $q_j=\tilde{q}_{t}=2^{n_t}$.
Note that for any $t \in \N$, we have that $ \tilde{q}_{t,p} < \tilde{q}_{t} $ for every $p\in\cP_t$.
Hence,
$$
\begin{aligned}
\sum_{m=1}^j\frac{(q_m,q_j)}{q_j}
&\geq \ \sum_{p\in\cP_t}\frac{(\tilde{q}_{t,p},\tilde{q}_t)}{\tilde{q}_t} \ \geq \ \sum_{p\in\cP_t}\frac{2^{n_t-u_p-1}}{2^{n_t}} \ = \ \frac12\sum_{p\in\cP_t}\frac{1}{2^{u_p}}\\[3ex]
&\geq \ \frac12\ \sum_{p\in\cP_t}\frac{1}{p} \  \stackrel{\eqref{lb_sum_reciprocal_primes}}{\geq}
\ \frac12\ln\ln n_t-\frac12\ln2- \frac{1}{4} \  \stackrel{\eqref{regular_2}}{\geq} \ \frac14\ln\ln n_t  \, .
\end{aligned}
$$
The upshot of this is that whenever  $j \in \N$ is such that $q_j=\tilde{q}_t$ for some $t \in \N$, then
\begin{equation} \label{gcd_lb}
\sum_{m=1}^j\frac{(q_m,q_j)}{q_j}\geq\frac14\ln\ln n_t.
\end{equation}

\bigskip

\noindent {\em Step 4: Constructing the function $\psi$. \ } Working with the sequence $\cA=(q_j)_{j\in\N}$ coming from Step~1,  the goal is to construct a suitable function $\psi$ so that \eqref{bald} is satisfied.   To begin with we split  $\cA$ into two classes.   We define $\cI_1$ to be the set of indices $j\in\N$ such that $\pi(j)=(t,p)$ for some $t\in\N$ and  $p\in\cP_t$. In other words, $j \in \cI_1$ if and only if the corresponding element $q_j \in \cA $ is not a power of $2$. We let  $\cI_2 := \N \setminus \cI_1$.  Thus,  $\cI_2$ is the set  of  indices $j\in\N$ such that $\pi(j)=(t,2)$ for some $t\in\N$.
For any index $j\in\cI_2$, in order to emphasize the dependence on $j$, let us denote by $t_j$ the unique  integer associated with  $ \pi(j)$.  Thus, by definition
$$
q_j=\tilde{q}_{t_j}=2^{n_{t_{j}}}.
$$
Note that in view of  \eqref{gcd_lb}, for any  $j\in\cI_2$
\begin{equation} \label{lb_gcd_concrete}
\sum_{m=1}^j\frac{(q_m,q_j)}{q_j} \, \geq  \, \frac14\ln\ln n_{t_j}  \stackrel{\eqref{bu2}}{>} \  \frac14 \ln\left( (t_j+14)\ln 2 \right) \,  .
\end{equation}

\noindent  We define the function $\psi:\N \rightarrow \I$ on the sequence $\cA=(q_j)_{j\in\N}$ as follows:

\begin{enumerate}
\item[$\bullet$] For $j\in\cI_1$, we let $$\psi(q_j):=2^{-j}$$.
\item[$\bullet$]
For $j\in\cI_2$, we let
$$ \psi(q_j):=\left\{
\begin{array}{ll}
   1 & \mbox{if} \;\;\; t_j=1 \;
      ,\\[2ex]
  \frac{1}{t_j\log t_j} & \mbox{if} \;\;\; t_j \geq 2\; .
\end{array}\right.$$
\end{enumerate}

\vspace*{2ex}

\noindent First of all, note that
$$
\sum_{j\in\cI_1}\psi(q_j)   \,  \leq \,  \sum_{j\in  \N } 2^{-j} \, \leq  \, 1   \, .
$$

\noindent Then,
it follows that for any integer $N\in\cI_2$
\begin{eqnarray} \label{final_ub}
\sum_{j=1}^{N}\psi(q_j) & =  & \sum_{\substack{1\leq j\leq N\\j\in\cI_1}}\psi(q_j)\ +\sum_{\substack{1\leq j\leq N\\j\in\cI_2}}\psi(q_j) \,  \leq   \, 2 \ + \ \sum_{\substack{1\leq j\leq N\\j\in\cI_2}} \frac{1}{t_j \log t_j} \nonumber  \\[2ex]
& = & 2 +   \sum_{i=2}^{t_N} \frac{1}{i \log i} \ \ll \ \max \{1, \log\log t_N \}\, ,
\end{eqnarray}
where  $t_N$  is the unique  integer associated with  $ \pi(N) $ so that $q_N=\tilde{q}_{t_N}$. On the other hand, it follows that for any integer $N\in\cI_2$

\begin{eqnarray} \label{final_lb}
\sum_{j=1}^N\psi(q_j)\sum_{m=1}^j\frac{(q_m,q_j)}{q_j}
&> &
\sum_{\substack{1\leq j\leq N\\j\in\cI_2}}\psi(q_j)\sum_{m=1}^j\frac{(q_m,q_j)}{q_j}  \nonumber
\\[2ex]
&   \stackrel{\eqref{lb_gcd_concrete}}{\geq} & \ln (15 \ln 2)  \ + \  \sum_{\substack{2\leq j\leq N\\j\in\cI_2}}  \frac{ \ln (t_j\ln 2)}  {t_j\log t_j   }  \nonumber
\\[3ex]
& \gg & \sum_{\substack{1\leq j\leq N\\j\in\cI_2}}  \frac{ 1}  {t_j  }    \   =  \
\sum_{i=1}^{t_N}  \frac{1}{i}   \ \gg \  \max \{1,\log t_N  \}  \, .
\end{eqnarray}
This together with \eqref{final_ub} implies the desired estimate \eqref{bald} for any integer $N\in\cI_2$.  We now show that inequalities~\eqref{final_ub} and~\eqref{final_lb} are valid for any integer $N$ satisfying
\begin{equation} \label{N_is_large}
N\geq n_{100} \, .
\end{equation}
With this in mind, given such an $N$, define $k\in\N$ by
\begin{equation} \label{db_regularization_2}
n_k\leq \log_2 q_N <n_{k+1}.
\end{equation}

\noindent Now let $N_1$ be the integer such that $q_{N_1}=\tilde{q}_{k}=2^{n_k}$ and let $N_2$ be the integer  such that $q_{N_2}=\tilde{q}_{k+1}=2^{n_{k+1}}$. Note that by definition,  both $N_1,N_2\in\cI_2$ and  $t_{N_1} = k $ and $t_{N_2} = k+1 $. Also, in view of~\eqref{db_regularization_2}
$$
q_{N_1}\leq q_N< q_{N_2}   \, .
$$

\noindent Thus,
\begin{equation} \label{final_ub_2}
\sum_{j=1}^{N}\psi(q_j) < \sum_{j=1}^{N_2}\psi(q_j)  \stackrel{\eqref{final_ub}}{\ll} \log\log t_{N_2} =  \log\log (k+1) \ll \log\log k,
\end{equation}
where in the last  step  we use the fact that~\eqref{N_is_large} implies $k\geq 100$.  On the other hand, it follows that
\begin{equation} \label{final_lb_2}
\sum_{j=1}^N\psi(q_j)\sum_{m=1}^j\frac{(q_m,q_j)}{q_j}\ > \ \sum_{j=1}^{N_1}\psi(q_j)\sum_{m=1}^j\frac{(q_m,q_j)}{q_j}
\ \stackrel{\eqref{final_lb}}{\gg} \ \log t_{N_1} \, = \,  \log k.
\end{equation}
On combining~\eqref{final_ub_2} and~\eqref{final_lb_2} we obtain the desired inequality \eqref{bald} for all $N \geq N_0 := n_{100}$.

\subsection*{Appendix D: Some basic results on  sums of sequences  \label{elemseq}}
\stepcounter{subsection}

In this appendix, we collect together various elementary lemmas concerning sums of sequences that are used  at various points in the main body of the paper; in particular,  during the course of establishing  Lemma \ref{lem_sum_mu_good} and Propositions~\ref{indiethm1}~$\&$~\ref{indiethm4}.

\begin{lemA} \label{cor_cummulative_log_limit}
Let $(s_n)_{n\in\N}$ be a sequence of real numbers  contained in $\I$ and let
$$
S_n:=\sum_{k=1}^n s_k.
$$
Let $a,b\in\N$ with $2\leq a<b$ and  let $\gamma>0$.  Suppose  that
\begin{equation} \label{S_am1_geq_1}
S_{a-1}\geq \gamma.
\end{equation}
Then,
$$
\frac{\gamma}{\gamma+1}  \, \big( \log S_b-\log S_{a-1} \big)   \ \leq  \ \sum_{k=a}^b \frac{s_k}{S_k}  \ \leq \  \frac{1}{\gamma\log\left(\frac{\gamma+1}{\gamma}\right)}  \,  \big(  \log S_b-\log S_{a-1} \big) \, .
$$
\end{lemA}

\vspace*{3ex}

\begin{proof}
For any integer  $n\geq a$, let
\begin{equation} \label{cor_cummulative_log_limit_def_a_n}
a_n:=\frac{s_n}{S_n}
\end{equation}
and
\begin{equation} \label{cor_cummulative_log_limit_def_b_n}
b_n:=\log S_n-\log S_{n-1} \,  =  \  \log(1+\frac{s_n}{S_{n-1}})  \, .
\end{equation}

\noindent The proof of the lemma will follow on showing that
\begin{equation} \label{ezsv}
\frac{\gamma}{\gamma+1}  \, b_n   \ \leq  \ a_n  \ \leq \  \frac{1}{\gamma\log\left(\frac{\gamma+1}{\gamma}\right)}  \,  b_n \, .
\end{equation}
First note that
 $b_n=0$ if and only if $s_n=0$, which, in turn, is true if and only if $a_n=0$. Thus, \eqref{ezsv} is trivially true  if $b_n=0$ . We can therefore  assume that $b_n>0$.
 Then,
$$
\frac{a_n}{b_n} \;  = \; \frac{\frac{s_n}{S_n}}{\log\left(1+\frac{s_n}{S_{n-1}}\right)} \, = \, \frac{S_{n-1}}{S_n} \cdot \frac{\frac{s_n}{S_{n-1}}}{\log\left(1+\frac{s_n}{S_{n-1}}\right)} \; .
$$

\noindent By~\eqref{S_am1_geq_1}, for all $n\geq a$ we have that
$$
\frac{\gamma}{\gamma+1}\leq\frac{S_{n-1}}{S_n}\leq 1 \,
$$
and together with the fact that the function $x\to\frac{x}{\log(1+ x)}$ is monotonically increasing for all $x>0$, it follows  that
$$
1 \leq\frac{\frac{s_n}{S_{n-1}}}{\log\left(1+\frac{s_n}{S_{n-1}}\right)} \leq \frac{1}{\gamma\log\left(\frac{\gamma+1}{\gamma}\right)}   \, .
$$
On  multiplying the last two double inequalities, we find that
$$
\frac{\gamma}{\gamma+1} \leq \frac{S_{n-1}}{S_n}\cdot\frac{\frac{s_n}{S_{n-1}}}{\log\left(1+\frac{s_n}{S_{n-1}}\right)} \leq \frac{1}{\gamma\log\left(\frac{\gamma+1}{\gamma}\right)} \, .
$$
In other words,
$$
 \frac{\gamma}{\gamma+1}  \, \leq  \,  \frac{a_n}{b_n} \, \leq  \, \frac{1}{\gamma\log\left(\frac{\gamma+1}{\gamma}\right)}   \,
$$
and  the desired statement \eqref{ezsv}  follows.
\end{proof}

\bigskip

\begin{lemA} \label{cor_sum_S_log}
Let $(s_n)_{n\in\N}$ and $S_n$ be  as in Lemma~\ref{cor_cummulative_log_limit}.  Let $\gamma>0$  and  let
$$
\tilde{S}_n:=\max\left(\gamma,S_n\right).
$$
Then, for any  $a,b\in\N$ with  $a <  b$, we have that
\begin{equation}  \label{cor_sum_S_log_result}
\sum_{k=a}^{b}\frac{s_k}{\tilde{S}_k}<1+\frac{1}{\gamma}+\frac{\log S_b-\log S_a}{\gamma\log\left(\frac{\gamma+1}{\gamma}\right)}.
\end{equation}
\end{lemA}
\begin{proof}
Denote by $m \in \N $ the smallest integer such that $\tilde{S}_{m}>\gamma$. This  implies that $S_m > \gamma$,  $S_n \leq  \gamma $ if  $n < m$  and
\begin{equation} \label{cor_sum_S_log_case1}
S_m=s_{m}+  S_{m-1} \leq \gamma+1.
\end{equation}

\noindent We now split the proof into three cases depending on the size of $m$.

\begin{itemize}
  \item[(i)] If  $m\leq a-1$, then it follows that  $a\geq 2$ and  that $S_{a-1}>\gamma$. Thus,  Lemma~\ref{cor_cummulative_log_limit} implies that
      $$
      \sum_{k=a}^{b}\frac{s_k}{\tilde{S}_k} \, \leq \, \sum_{k=a}^{b}\frac{s_k}{S_k} \, \leq \,   \frac{\log S_b-\log S_a}{\gamma\log\left(\frac{\gamma+1}{\gamma}\right)}  \,
      $$
      and this proves~\eqref{cor_sum_S_log_result}.
  \item[(ii)] If $m\geq b$, then it follows that  $S_b\leq\gamma+1$.  Hence,
\begin{equation} \label{cor_sum_S_log_case2}
\sum_{k=a}^{b}\frac{s_k}{\tilde{S}_k} \, =  \,  \sum_{k=a}^{b}\frac{s_k}{\gamma}  \, \leq  \,  \frac{S_b}{\gamma}  \, \leq \, 1+\frac{1}{\gamma}  \,
\end{equation}
and this proves~\eqref{cor_sum_S_log_result}.
  \item[(iii)]
If $a\leq m < b$,  the previous two cases can naturally be utilised to yield that
$$
\sum_{k=a}^{b}\frac{s_k}{\tilde{S}_k} \, = \, \sum_{k=a}^{m}\frac{s_k}{\tilde{S}_k} \, +\sum_{k={m+1}}^{b}\frac{s_k}{\tilde{S}_k} \, \leq  \,  1+\frac{1}{\gamma}+\frac{\log S_b-\log S_{m+1}}{\gamma\log\left(\frac{\gamma+1}{\gamma}\right)}  \, .
$$
This together with the fact that $S_{m+1} \geq  S_a$ proves~\eqref{cor_sum_S_log_result}.
\end{itemize}
\vspace*{-5ex}
\end{proof}

\bigskip

\begin{lemA} \label{lem_sum_S_square}
Let $(s_n)_{n\in\N}$ and $S_n$ be  as in Lemma~\ref{cor_cummulative_log_limit}. Let $a,b\in\N$ with $2\leq a<b$ and suppose  that $S_{a-1}>0$.  Then
$$
\sum_{k=a}^b \frac{s_k}{S_k^2} \, \leq  \,  \frac{1}{S_{a-1}} - \frac{1}{S_b}  \, .
$$
\end{lemA}
\begin{proof}
The desired inequality immediately follows from the observation that for any integer $k\geq 2$,
\begin{equation*}
\frac{1}{S_{k-1}}-\frac{1}{S_k} \, = \, \frac{s_k}{S_{k-1}S_k} \, =  \,
\frac{S_k}{S_{k-1}}\cdot\frac{s_k}{S_k^2}
\end{equation*}
and so
\begin{equation*}
\frac{s_k}{S_k^2}  \, \leq  \, \frac{1}{S_{k-1}}-\frac{1}{S_k}  \, .
\end{equation*}
\end{proof}

\vspace*{3ex}

\begin{lemA} \label{cor_sum_S_square} Let $(s_n)_{n\in\N}$ and $S_n$ be  as in Lemma~\ref{cor_cummulative_log_limit}.  Let $\gamma>0$  and  let
$
\tilde{S}_n:=\max\left(\gamma,S_n\right).
$
Then
$$
\sum_{k=1}^{\infty}\frac{s_k}{\tilde{S}_k^2}  \, <  \, \frac{2\gamma+1}{\gamma^2}.
$$
\end{lemA}
\begin{proof}
As in the proof of Lemma \ref{cor_sum_S_log}, let $m \in \N$ be the smallest integer such that $\tilde{S}_{m}>\gamma$. This implies  that $S_m > \gamma$,  $S_n \leq  \gamma $ if  $n < m$  and
$
S_m \leq \gamma+1.
$
Then, by making use of Lemma~\ref{lem_sum_S_square} with $a=m+1$,  it is easily verified that
$$
\sum_{k=1}^{\infty}\frac{s_k}{\tilde{S}_k^2} \, =  \,
\sum_{k=1}^{m}\frac{s_k}{\tilde{S}_k^2} \, +\sum_{k=m+1}^{\infty}
\frac{s_k}{\tilde{S}_k^2} \, \leq \, \frac{\gamma+1}{\gamma^2}+\frac{1}{S_m}<\frac{\gamma+1}{\gamma^2}+\frac{1}{\gamma} \, = \, \frac{2\gamma+1}{\gamma^2} \, .
$$
\end{proof}

\vspace*{4ex}

\noindent{\bf Acknowledgements:  \, } SV would like to take this opportunity to thank Alison Butterworth and Dave Scott  for their calmness, frankness and amazing power to make sense of the pressures in a crazy world. You guys have been absolute bricks!  Next, comes the inevitable  mention of the now seventeen year old dynamic duo - Ayesha and Iona.  You  continue to push the boundaries of my world. Long may it last.  Also many congratulations for your brilliant achievements last year. I remain a very proud and lucky Papa V!  Finally, many thanks to  Bridget for making me appreciate the finer aspects of life - including of course leprechauns and  soda bread!

EZ acknowledges the comprehensive support of Aljona, Alyssa and Alina. The discussions about causes of everything in the world has helped a lot with understanding of the limits of my knowledge and patience, while joint visits to playgrounds has provided a very lively example of equidistribution (in the playground). Many thanks for all the inspiration and motivation you give!

\vspace{13mm}

\vspace{7mm}


\vspace{5mm}

\noindent Andrew D. Pollington: National Science Foundation

\vspace{-2mm}

\noindent\phantom{Andrew D. Pollington: }Arlington, VA 22230, USA.


\noindent\phantom{Andrew D. Pollington: }e-mail: adpollin@nsf.gov


\vspace{5mm}

\noindent Sanju Velani: Department of Mathematics,
University of York,

\vspace{-2mm}

\noindent\phantom{Sanju Velani: }Heslington, York, YO10
5DD, England.


\noindent\phantom{Sanju Velani: }e-mail: sanju.velani@york.ac.uk


\vspace{5mm}

\noindent Agamemnon Zafeiropoulos: Department of Analysis and Computational Number Theory,

\vspace{-2mm}

\noindent\phantom{Agamemnon Zafeiropoulos: }Graz University of Technology, TU Graz, Austria.


\noindent\phantom{Agamemnon Zafeiropoulos: }e-mail: mzafeiropoulos@gmail.com

\vspace{5mm}

\noindent Evgeniy Zorin: Department of Mathematics,
University of York,

\vspace{-2mm}

\noindent\phantom{Evgeniy Zorin: }Heslington, York, YO10
5DD, England.


\noindent\phantom{Evgeniy Zorin: }e-mail: evgeniy.zorin@york.ac.uk


\end{document}